\documentclass[11pt,a4paper,leqno]{amsart}
\usepackage[margin=25mm]{geometry}

\date{November 20, 2025}
\usepackage[all]{xy}
\usepackage[title]{appendix}
\usepackage{amssymb}
\usepackage{amsmath}
\usepackage{amsthm}
\usepackage{tikz}
\usepackage{tikz-cd}
\usepackage{mathtools}
\usepackage{framed}
\usepackage{comment}
\definecolor{darkblue}{rgb}{0,0,0.6}
\usepackage[breaklinks, ocgcolorlinks,colorlinks=true, citecolor=darkblue, filecolor=darkblue, linkcolor=darkblue, urlcolor=darkblue]{hyperref}
\usepackage[capitalize]{cleveref}

\newtheorem{theorem}{Theorem}[section]

\newtheorem{proposition}[theorem]{Proposition}
\newtheorem*{proposition*}{Proposition}
\newtheorem{lemma}[theorem]{Lemma}
\newtheorem{corollary}[theorem]{Corollary}
\newtheorem*{corollary*}{Corollary}

\newtheorem{question}[theorem]{Question}
\newtheorem*{question*}{Question}

\makeatletter
\newtheorem*{rep@theorem}{\rep@title}
\newcommand{\newreptheorem}[2]{
	\newenvironment{rep#1}[1]{
		\def\rep@title{#2 \ref{##1}}
		\begin{rep@theorem}}
		{\end{rep@theorem}}}
\makeatother

\newreptheorem{theorem}{Theorem}
\newreptheorem{lemma}{Lemma}
\newreptheorem{proposition}{Proposition}
\newreptheorem{corollary}{Corollary}

\newtheorem{thmx}{Theorem}

\theoremstyle{definition}
\newtheorem{definition}[theorem]{Definition}
\newtheorem{example}[theorem]{Example}
\newtheorem*{theorem*}{Theorem}
\newtheorem*{acknowledgement*}{Acknowledgements}

\newtheorem{convention}[theorem]{Convention}

\theoremstyle{remark}
\newtheorem{remark}[theorem]{Remark}
\newtheorem*{remark*}{Remark}

\usepackage{rotating}

\usepackage{color}
\usepackage{enumerate}

\usepackage{cite}
\usepackage[nodisplayskipstretch]{setspace}
\usepackage{placeins}

\newenvironment{clist}[1]
{\begin{enumerate}[\normalfont #1]}
{\end{enumerate}}

\newcommand{\wh}{\widehat}
\newcommand{\wt}{\widetilde}

\usepackage{mathrsfs}
\usepackage{adjustbox}

\usepackage{nccmath}

\makeatletter
\@namedef{subjclassname@2020}{
  \textup{2020} Mathematics Subject Classification}
\makeatother

\usetikzlibrary{decorations.markings}

\renewcommand{\thesubsection}{\text{\thesection\alph{subsection}}}

\newcommand{\cy}[1]{\Z/{#1}}

\newcommand{\mymatrix}[4]{\left (\vcenter
{\xymatrix@C-2pc@R-2pc{#1&#2\\#3&#4} }
\right )}

\newcommand{\R}{\mathbb{R}}
\newcommand{\Z}{\mathbb{Z}}

\newcommand{\ol}{\overline}
\newcommand{\hh}{\string^}
\renewcommand{\mod}{\,\,\text{\normalfont mod}\,}

\DeclareMathOperator{\Aut}{Aut}
\DeclareMathOperator{\Hom}{Hom}
\DeclareMathOperator{\id}{id}

\DeclareMathOperator{\rank}{rank}

\DeclareMathOperator{\IM}{im}

\DeclareMathOperator{\Tors}{Tors}
\DeclareMathOperator{\hCob}{hCob}

\DeclareMathOperator{\Ab}{Ab}
\DeclareMathOperator{\Iso}{Iso}
\DeclareMathOperator{\Sesq}{Sesq}
\DeclareMathOperator{\Herm}{Herm}
\DeclareMathOperator{\Isom}{Isom}

\DeclareMathOperator{\DiagIsom}{Diag}
\DeclareMathOperator{\rest}{rest}

\DeclareMathOperator{\Br}{Br}
\DeclareMathOperator{\Alg}{Alg}

\DeclareMathOperator{\HT}{HT}
\DeclareMathOperator{\PHT}{PHT}

\DeclareMathOperator{\alg}{alg}

\DeclareMathOperator{\TriIsom}{Tri}

\renewcommand{\L}{L}
\renewcommand{\l}{\Lambda}
\newcommand{\dAlg}{d\text{\normalfont Alg}}
\newcommand{\dHT}{d\text{\normalfont HT}}
\newcommand{\sM}{\mathscr{M}}
\newcommand{\Malg}{\mathscr{M}^{\alg}}
\newcommand{\dalg}{\mathscr{D}^{\alg}}

\newcommand{\La}{\Lambda}
\newcommand{\wX}{\widetilde X}

\newcommand{\wK}{\widetilde K}
\newcommand{\wbar}[1]{\overline{#1} }
\newcommand{\bZ}{\mathbb Z}
\newcommand{\bQ}{\mathbb Q}
\newcommand{\ZG}{\bZ G}
\newcommand{\RG}{R G}

\newcommand{\wH}{\widehat H}

\DeclareMathOperator{\Met}{Met}
\newcommand{\ep}{\varepsilon}

\newcommand{\bd}{\partial}
\DeclareMathOperator{\coker}{coker}
\newcommand{\la}{\langle}
\newcommand{\ra}{\rangle}
\newcommand{\sx}{\hphantom{x}}

\newcommand{\SU}{SU^{\varepsilon}}
\newcommand{\RU}{RU^{\varepsilon}}
\newcommand{\U}{U^{\varepsilon}}
\newcommand{\EU}{EU^{\varepsilon}}
\DeclareMathOperator{\Sym}{Sym}
\DeclareMathOperator{\SymCat}{\mathfrak{S}}
\DeclareMathOperator{\ad}{ad}

\newcommand{\vv}{\,|\,}

\newcommand
{\eqncount}{\setcounter{equation}{\value{theorem}}
\addtocounter{theorem}{1}}

\begin{document}

\title[Four-manifolds, two-complexes and the quadratic bias invariant]{Four-manifolds, two-complexes and the \\ quadratic bias invariant}

\author{Ian Hambleton}
\address{Department of Mathematics \& Statistics, McMaster University
L8S 4K1, Ontario, Canada} 
\email{hambleton@mcmaster.ca}

\author{John Nicholson}
\address{School of Mathematics and Statistics, University of Glasgow, United Kingdom}
\email{john.nicholson@glasgow.ac.uk}

%57K41 Invariants of 4-manifolds (including Donaldson and Seiberg-Witten invariants)
%57N65 Algebraic topology of manifolds
%20F05 Generators, relations, and presentations of groups
%55U15 Chain complexes in algebraic topology
%57R67 Surgery obstructions, Wall groups

\begin{abstract} 
Kreck and Schafer produced the first examples of stably diffeomorphic closed smooth $4$-manifolds which are not homotopy equivalent. They were constructed by applying the doubling construction to $2$-complexes over certain finite abelian groups of odd order.
By extending their methods, we formulate a new homotopy invariant on the class of $4$-manifolds arising as doubles of $2$-complexes with finite fundamental group.
As an application we show that, for any $k \ge 2$, there exist a family of $k$ closed smooth $4$-manifolds which are all stably diffeomorphic but are pairwise not homotopy equivalent.
\end{abstract}

\maketitle
\vspace{-5mm}

\section{Introduction} \label{s:intro}

Two closed smooth 4-manifolds $M$, $N$  are said to be \textit{stably diffeomorphic} if there exists $r \ge 0$ and a diffeomorphism $M \# r(S^2 \times S^2) \cong N \# r(S^2 \times S^2)$. Kreck's modified surgery \cite{Kr99} gives techniques to classify $4$-manifolds up to stable diffeomorphism, and these methods have been  applied to study manifolds over a range of fundamental groups \cite{Hambleton:1988,Hambleton:2009,Hambleton:2018,Hambleton:2019,Hambleton:2024,Hambleton:2023,KLPT17,KPT20}.

The existence of exotic smooth structures shows that simply-connected oriented $4$-manifolds which are stably diffeomorphic need not be diffeomorphic, 
but it follows from results of Donaldson \cite{Donaldson:1983a} and Wall  \cite{Wall:1964} that 
such $4$-manifolds are $h$-cobordant and hence homotopy equivalent. 

In contrast, Kreck-Schafer \cite{Kreck:1984} produced the first examples of closed smooth $4$-manifolds which are stably diffeomorphic, but not even homotopy equivalent. 
Their examples arose from the following  \emph{doubling construction}: for a finite $2$-complex $X$, let $M(X)$ be the boundary of a smooth regular neighbourhood of an embedding $X \hookrightarrow \R^5$ (see \cref{s:doubling-construction} for more details). 
The construction has the following properties (see \cite[Section 2]{Kreck:1984}): 
\begin{enumerate}
\item If $X$ and $Y$ are finite 2-complexes such that $X \simeq Y$ are homotopy equivalent, then $M(X)$ and $M(Y)$ are $h$-cobordant, hence homotopy equivalent.
\item If $\chi(X) = \chi(Y)$ and $\pi_1(X) \cong \pi_1(Y)$, then $M(X)$ and $M(Y)$ are stably diffeomorphic.
\end{enumerate}
Their main result was that there exist pairs of finite $2$-complexes $X$, $Y$ with $\chi(X) = \chi(Y)$ and $M(X) \not\simeq M(Y)$ such that $\pi_1(X) \cong \pi_1(Y)$ is elementary abelian of odd order. 
To achieve this, they defined a homotopy invariant (in $\cy 2$) for doubles $M(X)$ with $\pi_1(X)$  finite of odd order.

\smallskip
In this paper we will define and study the \textit{quadratic bias invariant}, which generalises the invariant of Kreck-Schafer.
Let $G$ be a finite group and $X$ a finite 2-complex with $\pi_1(X) \cong G$ which is \textit{minimal} in the sense that $\chi(X)$ is minimal over such complexes.
The \textit{bias invariant} was defined by Metzler \cite{Metzler:1976} to be a class $\beta(X)$ in an abelian group $B(G) := (\Z/m)^\times/\langle \pm D(G) \rangle$ where $m = m_G$ (\cref{def:modulus}) and $D(G)$ is the image of a certain map $\varphi\colon\Aut(G) \to (\Z/m)^\times/\{\pm 1\}$ (\cref{def:bias-original}).
The quadratic bias invariant will be a class $\beta_Q(M(X))$ in a quotient group $B_Q(G)$ of $B(G)$ 
(\cref{def:quadbias-high}).
Let $\sM_4(G)$ denote the set of homotopy types of $4$-manifolds $M(X)$ where $X$ is a minimal finite 2-complex with $\pi_1(X) \cong G$. We will show:

\begin{thmx} \label{thmx:main-general}
The quadratic bias invariant is a homotopy invariant. In particular, for a finite group $G$, the quadratic bias invariant defines a map
\[ \beta_Q\colon  
\sM_4(G) \to B_Q(G). \]
Furthermore, $\beta_Q(M(X)) = q(\beta(X))$ where $q \colon B(G) \twoheadrightarrow B_Q(G)$ is the natural surjection.	
\end{thmx}

The extent to which the quadratic bias can be used to distinguish manifolds depends on choosing the quotient $B_Q(G)$ of $B(G)$ so that $\beta_Q(M(X))$ is both a homotopy invariant and can also be computed in non-trivial examples.
\cref{s:bias-manifolds,s:qbias_Lgroups} are directed towards this goal.
A description of the subgroup $N(G)= \ker(q\colon B(G) \to B_Q(G))$ can be found in \cref{remark:Aut(G)-is-easy} (ii).

If $H_2(G;\bZ)$ has a certain special form  we can explicitly compute the quadratic bias obstruction group $B_Q(G)$.
Recall that a finite group $G$ is \textit{efficient} if $\chi_{\min}(G) = 1 + d(H_2(G;\Z))$, where $\chi_{\min}(G)$ denotes the minimal Euler characteristic of a finite $2$-complex $X$ with $\pi_1(X) \cong G$, and 
$d(\cdot)$ denotes the minimal number of generators of a group \cite{Ha00}.

\begin{thmx} \label{thmx:main-B_Q(G)}
Let $G$ be a finite group such that $H_2(G;\Z) \cong (\Z/m)^d$ 
for some $m\geq 1$, $d \ge 3$.  If $G$ is efficient, then there is an isomorphism
\[ B_Q(G) \cong \frac{(\Z/m)^\times}{\pm (\cy m)^{\times 2} \cdot D(G)}\]
where $D(G) = \IM(\varphi_G\colon  \Aut(G) \to  (\Z/m)^\times/\{\pm 1\})$. If $G$ is not efficient, then $B_Q(G)=0$.
\end{thmx}

It is known that finite abelian groups are efficient \cite[Proposition 5]{Sieradski:1979}, but Swan constructed a group of the form $G \cong (\cy 7)^3 \rtimes \cy 3$ which is not efficient \cite[p196]{Swan:1965}.

\begin{remark} We would expect the structure of $B_Q(G)$ to be much more complicated in general. Its definition involves new ideas related to the unitary isometries of multi-scaled hyperbolic forms arising from the decomposition of $H_2(G;\bZ)$ into cyclic factors (see \cref{prop:sevenfour}).
\end{remark}

Using Theorems \ref{thmx:main-general} and \ref{thmx:main-B_Q(G)}, we are now able to effectively compare closed smooth 4-manifolds of the form $M(X)$ up to homotopy equivalence. This allows us to establish the following, which answers a recent question of Kasprowski-Powell-Ray in the affirmative \cite[Question 1.5]{KPR22}.

\begin{thmx} \label{thmx:main-examples}
For each $k \geq 2$, there exist closed smooth $4$-manifolds $M_1, M_2, \cdots, M_k$ which are all stably diffeomorphic but not pairwise homotopy equivalent.
\end{thmx} 

These examples can be taken to be stably parallelisable and have finite abelian fundamental groups of the form $G = (\cy m)^d$, for any $d \ge 3$ odd, and any $m \geq 3$ with sufficiently many distinct prime factors. This generalises the case $k=2$ which was established by Kreck-Schafer \cite{Kreck:1984} over fundamental groups of the form $(\Z/p)^d$ for a prime 
 $p\equiv 1 \mod 4$. We also construct examples over non-elementary abelian groups $(\Z/m)^d \times \Z/t$ where $d \ge 4$ is even and certain values of $m$ and $t$ with $m \ne t$.

\begin{remark} 
The result of \cref{thmx:main-examples} is optimal for manifolds with finite fundamental group, since by \cite[Corollary 1.5]{Hambleton:1988} there are only finitely many homeomorphism types of closed $4$-manifolds with a given finite fundamental group and Euler characteristic.
It remains open whether there exists an infinite collection of such manifolds with arbitrary fundamental group.
\end{remark}

To obtain examples over other fundamental groups, we require a pair of finite $2$-complexes $X$, $Y$ with finite fundamental group $G$ and $\chi(X)=\chi(Y)$ but which are not homotopy equivalent. Such examples have previously only been known to exist when $G$ is either a finite abelian group \cite{Metzler:1976,Sieradski:1979} or a group with periodic cohomology \cite{Dy76,Ni20-II,Ni21}. 
However, if $G$ has periodic cohomology, then $H_2(G;\Z)=0$ \cite[Corollary 2]{Sw71} and so the bias invariant contains no information. 

In spite of this, we will establish the following (see \cref{thm:Q8xCp^3}). This serves to demonstrate that the quadratic bias invariant is computable for non-abelian fundamental groups.  
\begin{thmx} \label{thmx:Q8xCp^3}
Let $G = Q_8 \times (\Z/p)^3$ where $p$ is a prime such that $p \equiv 1 \mod 8$. Then:
\begin{clist}{(i)}
\item
There exist minimal finite $2$-complexes $X$, $Y$ with fundamental group $G$ which are homotopically distinct.
\item	
There exist closed smooth $4$-manifolds $M$, $N$ with fundamental group $G$ which are stably diffeomorphic but not homotopy equivalent. 
\end{clist}
\end{thmx}

This also gives the first example of a non-abelian finite group $G$ which does not have periodic cohomology such that there exists homotopically distinct finite $2$-complexes $X$, $Y$ with fundamental group $G$ and $\chi(X)=\chi(Y)$.
Part (ii) follows from (i) by taking $M = M(X)$, $N = M(Y)$ and applying the conditions of \cref{thmx:main-B_Q(G)} to show that $\beta_Q(M) \neq \beta_Q(N)$.

A number of interesting questions remain concerning the doubles $M(X)$.
If $X \simeq Y$, then $M(X)$ and $M(Y)$ are $h$-cobordant.
More generally, we ask: 

\begin{question} \label{question:diffeo}
If $X\simeq Y$, then are $M(X)$ and $M(Y)$ diffeomorphic?
\end{question}

An important special case is when $X$  is a point and $Y$ is the presentation $2$-complex of a potential counterexample to the Andrews-Curtis conjecture \cite{AC65}. In this case we have $X \simeq Y$, $M(X) = S^4$, $M(Y) \simeq S^4$ and so \cref{question:diffeo} is equivalent to the question of whether $M(Y)$ is an exotic $4$-sphere.
Such examples were considered by Akbulut-Kirby in \cite{AK85} and, for one such example, \cref{question:diffeo} was shown by Gompf to have an affirmative answer \cite{Go91}.

It was recently shown by Freedman-Krushkal-Lidman that all Seiberg-Witten invariants vanish for doubles $M(X)$ \cite[Proposition 1.3]{Freedman:2024}.
We could also ask \cref{question:diffeo} for simple homotopy equivalent $2$-complexes,  since in that case $M(X)$ and $M(Y)$ are $s$-cobordant \cite[p15]{Kreck:1984}.

\subsection{Comparison with quadratic $2$-type}

In \cite{Hambleton:1988,Ba88}, it was shown that a closed topological 4-manifold $M$ with $\pi_1(M)$ finite of odd order is determined up to homotopy equivalence by its quadratic 2-type 
\[ Q(M) = [\pi_1(M), \pi_2(M), k_M, S_M] \]
where $k_M \in H^3(\pi_1(M);\pi_2(M))$ denotes the $k$-invariant and $S_M\colon  \pi_2(M) \times \pi_2(M) \to \Z[\pi_1(M)]$ denotes the equivariant intersection form. 
An isometry $\cong$ of two such quadruples is an isomorphism of pairs $\pi_1$, $\pi_2$ respecting the $k$-invariant and inducing an isometry on $S$.
For arbitrary finite fundamental groups, there is a possibly non-zero additional invariant which lies in $\Tors(\Gamma(\pi_2(M)) \otimes_{\Z G} \Z)$. 
However, we will show (see \cref{thm:main-Q2type-body}):

\begin{theorem} \label{thm:main-Q2type}
Let $G$ be a finite group and let $M_1, M_2 \in \sM_{4}(G)$.
If $Q(M_1) \cong Q(M_2)$ are isometric, then $\beta_Q(M_1) = \beta_Q(M_2)$.
\end{theorem}

By \cref{{thmx:main-examples}} and \cite{KPR24,KNR22,Hambleton:1988,Ba88}, there are examples of stably diffeomorphic homotopy distinct 4-manifolds $M_1, \cdots, M_k$ such that $\Tors(\Gamma(\pi_2(M_i)) \otimes_{\Z G} \Z) \ne 0$. It is open whether this additional invariant is needed in order to determine such manifolds up to homotopy equivalence.

Although $\beta_Q(M)$ is determined by $Q(M)$, it is not immediately clear how to compute $\beta_Q(M)$ from $Q(M)$. It is also not clear how to use $Q(M)$ directly in order to establish quantitative results such as the ones given in Theorems \ref{thmx:main-examples} and \ref{thmx:Q8xCp^3} (ii). The following is currently open:

\begin{question} \label{question:pi2}
Do there exist closed smooth $4$-manifolds $M$, $N$ with fundamental group $G$ which are stably diffeomorphic but such that $\pi_2(M)$ and $\pi_2(N)$ are not $\Aut(G)$-isomorphic?
\end{question}

Similarly it is not known whether such manifolds exist such that $S_M$ and $S_N$ are not isometric modulo the action of $\Aut(G)$.
See \cref{s:preliminaries} for the definition of $\Aut(G)$-isomorphic.

\subsection{Results in higher dimensions}

In \cref{def:quadbias-high}, the invariant $\beta_Q$ is generalised to the doubles of finite $(G,n)$-complexes (see \cref{s:preliminaries}).
We define the quadratic bias invariant  
$$\beta_Q \colon \sM_{2n}(G) \to B_Q(G,n)$$
for all $n \ge 2$ where $\sM_{2n}(G)$ is the set of homotopy types of doubles of minimal finite $(G,n)$-complexes.
We obtain an analogue of \cref{thmx:main-general} (see \cref{thm:main-high}) as well as an analogue of \cref{thmx:main-B_Q(G)} (see \cref{thm:BQ-main}) when $G$ is a finite group such that $H_n(G;\Z) \cong (\Z/m)^d$ for some $m \ge 1$, $d \ge 3$ and $(G,n)$ satisfies the \textit{strong minimality hypothesis} (see \cref{def:min-hyp}).
The result is similar to \cref{thmx:main-B_Q(G)} for $n$ even, but we show that $B_Q(G,n) = 0$ for $n$ odd.

We also obtain explicit examples in higher dimensions, complementing the results of Conway-Crowley-Powell-Sixt \cite{CCPS22,CCPS23} which for dimensions $4n >4$ were either simply connected with $H_2(M) \neq 0$, or had infinite fundamental group $\pi_1(M) \cong \Z$.
See Theorems \ref{thm:final-examples} and \ref{thm:final-examples-eqforms}.

\begin{theorem} \label{thm:main-higher-dim}
For all $n \geq 2$ even, and all $k \geq 2$, there exist closed smooth $2n$-manifolds $M_1, M_2, \dots, M_k$  with non-trivial finite fundamental group which are all stably diffeomorphic, but not pairwise homotopy equivalent.
Furthermore, for the case $k=2$, the manifolds can be taken to have isometric equivariant intersection forms. \end{theorem}

\begin{remark}\label{rem:KSgap}
This result also addresses two gaps in the paper of Kreck-Schafer \cite{Kreck:1984}, where this result was given in the case $k=2$. Firstly, the proof that examples exist in the case $n >2$ is incomplete since it relies on a formula from \cite[Proposition 8]{Sieradski:1979} which is incorrect, as pointed out in \cite[p305]{Linnell:1993} (see \cite[p36-38]{Kreck:1984}).

Secondly, the examples constructed by \cite{Kreck:1984} were claimed to have isometric equivariant intersection form (see \cite[p21]{Kreck:1984}). The equivariant intersection forms were shown to be hyperbolic, but the possibility that they were hyperbolic over non-isomorphic $\Z G$-modules was not considered. This also affects the remark after \cite[Theorem 4.1]{Hambleton:1993} and was mentioned in \cite[p2]{CCPS22}.
\end{remark}

 In the case $n=2$, since the examples in \cref{thm:main-higher-dim} are distinguished by the quadratic bias invariant (see \cref{s:examples-manifolds}), it follows from \cref{thm:main-Q2type} that their quadratic $2$-types are not isometric. Hence the closed smooth $4$-manifolds $M_1$ and $M_2$ with isometric equivariant intersection form that we construct are distinguished by their $k$-invariants. This gives the first examples of stably diffeomorphic closed smooth $4$-manifolds which are distinguished up to homotopy equivalence by their $k$-invariants.

\subsection*{Organisation of the paper} 

We begin by recalling the necessary background on the bias invariant. We give preliminaries on CW-complexes (\cref{s:preliminaries}), then define the bias invariant in the setting of finite $(G,n)$-complexes and establish its main properties (\cref{s:bias-complexes}).

We next introduce the quadratic bias invariant. We give preliminaries on hermitian forms (\cref{s:forms}) and the doubling construction (\cref{s:doubling-construction}).
In \cref{s:bias-manifolds}, we define the quadratic bias invariant and prove \cref{thmx:main-general} (see \cref{thm:main-high}) and \cref{thm:main-Q2type} (see \cref{thm:main-Q2type-body}).

Finally, we evaluate the quadratic bias invariant and apply it to examples. We establish \cref{thmx:main-B_Q(G)} (see \cref{thm:BQ-main}), give details concerning the bias invariant for complexes (\cref{s:examples-complexes}), and then prove Theorems \ref{thmx:main-examples} and \ref{thmx:Q8xCp^3} (\cref{s:examples-manifolds}).
Some technical calculations are contained in Appendices \ref{ss:L-odd}, \ref{ss:L-theory-computations} and \ref{s:funct}, which give information about surgery obstruction groups and numerical functions describing the action of $\Aut(G)$ on the polarised quadratic bias invariant.

\subsection*{Conventions}

Our rings $R$ have identity, and we work in the category of 
 finitely generated left $R$-modules and left $R$-module homomorphisms.
We will assume all CW-complexes are connected with basepoint, and 
maps between CW-complexes will be cellular and basepoint-preserving.
Manifolds will be assumed to be smooth, closed, oriented and connected.

\begin{acknowledgement*} 

IH was partially supported by an NSERC Discovery Grant.
JN was supported by the Heilbronn Institute for Mathematical Research and a Rankin-Sneddon Research Fellowship from the University of Glasgow.
We gratefully acknowledge that this research was partially supported by the Fields Institute for Research in Mathematical Sciences in April 2023, the Isaac Newton Institute in May 2024, the University of Glasgow in June 2024, and the Heilbronn Institute in November 2024. The authors would like to thank Diarmuid Crowley, Jens Harlander, Daniel Kasprowski and Mark Powell for helpful conversations.  We would also like to thank the referee for many useful suggestions which improved the exposition of this article.
\end{acknowledgement*} 

\setcounter{tocdepth}{1}
\tableofcontents

\vspace{-7mm}

\section{Preliminaries on CW-complexes}
\label{s:preliminaries}

We begin by establishing some conventions and definitions.
Let $G$ be a group. A \textit{$(G,n)$-complex} is an $n$-dimensional CW-complex $X$ such that $\pi_1(X) \cong G$ and $\wt X$ is $(n-1)$-connected. Equivalently, it is the $n$-skeleton of a $K(G,1)$-space.
Note that a $(G,2)$-complex is equivalently a $2$-complex $X$ with $\pi_1(X) \cong G$.
We say a group $G$ has \text{type $F_n$} if there exists a finite $(G,n)$-complex. In particular, a group has type $F_2$ if and only if it is finitely presented. If $G$ is a finite group, then $G$ has type $F_n$ for all $n \ge 1$.

For a group $G$, a $G$-polarised space is a pair $(X,\rho_X)$ where $X$ is a space and $\rho_X\colon \pi_1(X) \xrightarrow[]{\cong} G$ is a group isomorphism which we often refer to as a \textit{polarisation}. 
If $h\colon  X \to Y$ is a map, then we can view $\pi_1(h) \in \Aut(G)$ using the $G$-polarisations $\rho_X$ and $\rho_Y$; formally, we use $\pi_1(h)$ to denote $\rho_Y \circ \pi_1(h) \circ \rho_X^{-1}$.
Two $G$-polarised spaces $X$ and $Y$ are said to be \textit{polarised homotopy equivalent} if there exists a homotopy equivalence $h\colon X \to Y$ such that $\pi_1(h) = \id_G \in \Aut(G)$.
We will assume all $(G,n)$-complexes are $G$-polarised.

Let $\HT(G,n)$ denote the set of homotopy types of finite $(G,n)$-complexes and let $\PHT(G,n)$ the set of polarised homotopy types of finite $(G,n)$-complexes. 

There is an action of $\Aut(G)$ on $\PHT(G,n)$ where $\theta \in \Aut(G)$ maps $(X,\rho) \in \PHT(G,n)$ to $(X,\theta \circ \rho)$. It follows easily that
\[ \HT(G,n) \cong \PHT(G,n) / \Aut(G). \]

The following will be our algebraic model for finite $(G,n)$-complexes. We will view $\Z$ as a $\Z G$-module with a trivial $G$-action. 

\begin{definition}
Let $n \ge 2$ and let $G$ be a group. An \textit{algebraic $n$-complex} over $\Z G$ is a chain complex $C = (C_*,\bd_*)$ of (finitely generated) free $\Z G$-modules $C_*$ equipped with a choice of $\Z G$-module isomorphism $H_0(C_*) \cong \Z$ such that
\begin{clist}{(i)}
\item
$C_i = 0$ for $i < 0$ or $i > n$.
\item
$H_i(C_*)=0$ for $0 < i < n$.
\end{clist}

Let $\Alg(G,n)$ denote the set of algebraic $n$-complexes over $\Z G$ considered up to the equivalence relation where $C \simeq C'$ if there exists a chain map $f\colon  C \to C'$ such that $H_0(f) = \id_{\Z}$ and $H_n(f)$ is a $\Z G$-isomorphism. We refer to this equivalence relation as \textit{chain homotopy equivalence}.
\end{definition}

If $C = (C_*,\bd_*) \in \Alg(G,n)$, then define $\chi(C) := \sum_{i=0}^n (-1)^i \rank_{\Z G}(C_i)$
where $\rank_{\Z G}(C_i)$ denotes the rank of $C_i$ as a free $\Z G$-module. 
This is a chain homotopy invariant and so does not depend on the choice of representative in $\Alg(G,n)$.

Let $\theta \in \Aut(G)$. If $M$ is a $\Z G$-module, let $M_\theta$ denote the $\Z G$-module with the same underlying abelian group but with $G$-action given by $g \cdot m := \theta(g) \cdot m$ for $g \in G$ and $m \in M$. 
We say two $\Z G$-modules $M$ and $N$ are \textit{$\Aut(G)$-isomorphic}, written $M \cong_{\Aut(G)} N$, if there is an isomorphism $M \cong N_\theta$ for some $\theta \in \Aut(G)$.

The class of algebraic $n$-complexes over $\Z G$ admit an action by $\Aut(G)$ (see \cite[Section 6]{Ni20-I}). If $C = (C_*,\partial_*) \in \Alg(G,n)$, then define
\[ C_\theta  = \big ((C_n)_\theta \xrightarrow[]{\bd_n} (C_{n-1})_\theta \xrightarrow[]{\bd_{n-1}} \cdots \xrightarrow[]{\bd_2} (C_1)_\theta \xrightarrow[]{\bd_1} (C_0)_\theta \big ). \]
Since each $C_i$ is a free $\Z G$-module, we have that $(C_i)_\theta \cong C_i$ (see \cite[Lemma 6.1 (i)]{Ni20-I}) and so $C_\theta \in \Alg(G,n)$. This action is well-defined on chain homotopy types and so induces an action of $\Aut(G)$ on $\Alg(G,n)$ (see \cite[Section 6]{Ni20-I}).

If $X$ is a finite CW-complex, then $C_*(\wt X)$ is a chain complex over $\Z[\pi_1(X)]$. Given a $G$-polarisation $\rho\colon  \pi_1(X) \to G$, we can then convert this to a chain complex over $\Z G$. We will denote this by $C_*(\wt X, \rho)$ when we want to emphasise the choice of polarisation. 

The following two propositions are standard and show that, in order to study finite $(G,n)$-complexes up to homotopy equivalence, it suffices to study algebraic $n$-complexes over $\Z G$ up to chain homotopy equivalence, and the action of $\Aut(G)$ on this class. For a convenient reference, see \cite[Proposition 5.1 \& Lemma 6.2]{Ni20-I}.

\begin{proposition} \label{prop:he-vs-che}
Let $n \ge 2$ and let $G$ be a group of type $F_n$. Then:
\begin{clist}{(i)}
\item
If $X$ is a finite $(G,n)$-complex, and $\rho\colon  \pi_1(X) \to G$ is a polarisation, then $C_*(\wt X,\rho)$ is an algebraic $n$-complex over $\Z G$. Furthermore, $\chi(C_*(\wt X,\rho)) = \chi(X)$.
\item
If $X$, $Y$ are finite $(G,n)$-complexes and $\theta \in \Aut(G)$, then there exists a homotopy equivalence $f\colon X \to Y$ such that $\pi_1(f) = \theta$ if and only if there exists a $\Z G$-chain homotopy equivalence $h\colon C_*(\wt X) \to C_*(\wt Y)_\theta$.	
\end{clist}
\end{proposition}

\begin{proposition} \label{prop:d2-problem}
Let $n \ge 2$ and let $G$ be a group of type $F_n$. Then the map
\[ \mathscr{C}\colon  \PHT(G,n) \to \Alg(G,n), \quad (X,\rho) \mapsto C_*(\wt X,\rho)  \]
is injective. Furthermore, it induces an injective map $\ol{\mathscr{C}}\colon  \HT(G,n) \to \Alg(G,n)/\Aut(G)$ where the action of $\Aut(G)$ on $\Alg(G,n)$ is as defined above. 
\end{proposition}

\begin{definition} \label{def:minimal-Gn}
	Let $n \ge 2$. If $G$ has type $F_n$, define
\[\chi_{\min}(G,n) = \min\{(-1)^n \chi(X)\colon \text{$X$ a finite $(G,n)$-complex} \}.\]
This value always exists (see, for example, \cite[Proposition 2.3 (ii)]{Ni23}).
We say that a finite $(G,n)$-complex $X$ is \textit{minimal} if $(-1)^n \chi(X) = \chi_{\min}(G,n)$, and we let $\HT_{\min}(G,n)$ denote the set of homotopy types of minimal finite $(G,n)$-complexes. Similarly, let $\Alg_{\min}(G,n)$ denote the set of chain homotopy types of algebraic $n$-complexes $C$ over $\Z G$ such that $(-1)^n\chi(C) = \chi_{\min}(G,n)$.
\end{definition}

\section{The bias invariant for $(G,n)$-complexes} \label{s:bias-complexes}

Throughout this section, we will fix an integer $n \ge 2$ and a finite group $G$. All algebraic $n$-complexes will be assumed to be over $\Z G$. All group homology groups will be taken to have coefficients in $\Z$ (with the trivial group action) unless otherwise mentioned. We let $d(G)$ denote the minimal number of generators of a group $G$

Throughout, we will make use of $0$th Tate cohomology. For a finite group $G$ and a $\Z G$-module $A$, recall that the corresponding $0$th Tate cohomology group is defined to be
\[ \wh H^0(G;A) := A^G / (N\cdot A)\]
where $N := \sum_{g \in G} g \in \Z G$ denotes the group norm and $N \cdot A:= \{N \cdot a\vv a\in A\} \le A^G$.
Then $(\,\cdot\,)\hh := \wh H^0(G; -)\colon \Z G\text{-mod} \to \Ab$ is a functor from $\Z G$-modules to abelian groups, and factors through the functor $(\,\cdot\,)^G \colon \Z G\text{-mod} \to \Ab$.

\subsection{The strong minimality hypothesis} \label{ss:min-hyp}

Since $G$ is finite, $H_n(G)$ is a finite abelian group and so $d(H_n(G)) < \infty$. The following is the geometric analogue of the minimality hypothesis \cite{Sieradski:1979}. 

\begin{definition} \label{def:min-hyp}
Let $n \ge 2$ and let $G$ be a finite group. We say that the pair $(G,n)$ satisfies the \textit{strong minimality hypothesis} if  $\chi_{\min}(G,n) = (-1)^n + d(H_n(G))$. 
\end{definition}

This can be viewed as a higher dimensional generalisation of efficiency, restricted to the case of finite groups.
Recall that a finitely presented group $G$ is \textit{efficient} if 
\[ \chi_{\min}(G) = 1 - r(H_1(G)) + d(H_2(G)),\] 
where $r(A)$ denotes the torsion free rank of $A$ \cite[p166]{Ha00}. If $G$ is finite, then $r(H_1(G))=0$ and so $(G,2)$ satisfies the strong minimality hypothesis if and only if $G$ is efficient.

\begin{remark} 
For $n \ge 2$ and a finite group $G$,  the pair $(G,n)$ satisfies the \textit{minimality hypothesis} if $\chi_{\min}^{\alg}(G,n) = (-1)^n+d(H_n(G))$, where $\chi_{\min}^{\alg}(G,n) = \min\{ (-1)^n\chi(C): C \in \Alg(G,n)\}$ (see \cite[Section 3]{Sieradski:1979}). If $n > 2$, then Wall \cite[Theorem E]{Wa65} shows that $\chi_{\min}(G,n) = \chi_{\min}^{\alg}(G,n)$ and so the minimality hypothesis and strong minimality hypothesis coincide in this case. If $n = 2$, then $\chi_{\min}^{\alg}(G,2) \le \chi_{\min}(G,2)$ and equality holds if $G$ has the D2 property \cite[Section 3]{Wa65}, \cite{Hambleton:2019a}, \cite{Jo03}.
It is currently not known whether or not $\chi_{\min}^{\alg}(G,2) = \chi_{\min}(G,2)$ for all finitely presented groups $G$.
\end{remark}

Note that $\chi_{\min}(G,2) \geq \chi_{\min}^{\alg}(G,2) \geq 1 + d(H_2(G))$. Thus, if $(G,2)$ satisfies the strong minimality hypothesis, then $(G,2)$ satisfies the minimality hypothesis.

It is known that finite abelian groups and finite $p$-groups satisfy the minimality hypothesis in all dimensions \cite[Proposition 5]{Sieradski:1979}, but finite $p$-groups are not known to be efficient in general.  Moreover, Swan \cite[p196]{Swan:1965} constructed a group of the form $G \cong(\cy 7)^3 \rtimes  \cy 3$ such that $(G,2)$ does not satisfy the minimality hypothesis, and so $G$ is not efficient.

We will now give an alternate formulation of the strong minimality hypothesis, which will be useful in defining the bias invariant in \cref{s:bias-complexes}. Define the \textit{invariant rank} to be $r_{(G,n)} := \rank_{\Z}(L^G)$ where $L = \pi_n(X)$ for $X$ any minimal $(G,n)$-complex. This does not depend on the choice of $X$.
The following can be proven similarly to \cite[Proposition 4]{Sieradski:1979}.

\begin{proposition} \label{prop:min-hyp}
Let $n \ge 2$ and let $G$ be a finite group. Then $\chi_{\min}(G,n) = (-1)^n + r_{(G,n)}$. In particular, $(G,n)$ satisfies the strong minimality hypothesis if and only if $r_{(G,n)} = d(H_n(G))$.
\end{proposition}

If $X$ is a minimal $(G,n)$-complex and $L = \pi_n(X)$, then it is a consequence of dimension shifting that $H_n(G) \cong \wh L$,  and so the natural surjection $L^G \twoheadrightarrow \wh L$ implies that $r_{(G,n)} \ge d(H_n(G))$.

\begin{definition} \label{def:modulus}
Let $n \ge 2$, let $G$ be a finite group and let $H_n(G) \cong \Z/m_1 \oplus \cdots \oplus \Z/m_r$ where $m_i \mid m_{i+1}$ for all $i \ge 1$ and $r =r_{(G,n)}$. Such a decomposition exists since $r_{(G,n)} \ge d(H_n(G))$, but we need not have $m_i \ne 1$ for all $i$.
Define the \textit{modulus} to be $m_{(G,n)} = m_1$.
\end{definition}

By the classification of finite abelian groups, $m_{(G,n)}$ does not depend on the choice of identification of $H_n(G)$. As usual, when $n=2$, we will write $r_G := r_{(G,2)}$ and $m_G := m_{(G,2)}$.

\begin{remark} \label{remark:m=1}
It follows by comparing Definitions \ref{def:min-hyp} and \ref{def:modulus} that 	$m_{(G,n)} \ne 1$ if and only if $(G,n)$ satisfies the strong minimality hypothesis.
\end{remark}

\subsection{The bias invariant for algebraic $n$-complexes} \label{ss:bias-algebraic}

The aim of this section will be to define the bias invariant, which was introduced by Metzler in \cite{Metzler:1976}. 
We will formulate our definitions so that they are analogous to those made in the definition of the quadratic bias invariant in \cref{s:bias-manifolds}.
Our treatment resembles the one given by Sieradski-Dyer \cite[p202]{Sieradski:1979} and is the version implicitly used by Kreck-Schafer \cite[p34]{Kreck:1984}. 
A convenient reference is \cite[Section 2]{Schafer:1996}.

Observe that, if two algebraic $n$-complexes $C_*$ and $D_*$ are $\Z G$-chain homotopy equivalent, then $\chi(C_*) = \chi(D_*)$. We will therefore restrict to the case of chain complexes with equal Euler characteristic.
The following is \cite[Lemma 1 (\S 2)]{Schafer:1996} specialised to the case $j=0$ and $G'=G$.

\begin{proposition} \label{prop:bias-tate}
Let $C_*$, $D_*$ be algebraic $n$-complexes such that $\chi(C_*) = \chi(D_*)$ and let $h\colon C_* \to D_*$ be a chain map such that $H_0(h) = \id_{\Z}$. Then the map
\[ H_n(h)\hh\colon H_n(C_*)\hh \to H_n(D_*)\hh \]
is an isomorphism and is independent of the choice of $h$. Let this be denoted by $\sigma(C_*,D_*)$.
\end{proposition}

Since $C_*$ and $D_*$ are algebraic $n$-complexes, it follows from standard homological algebra that a chain map such that $H_0(h)=\id_{\Z}$ always exists. Hence $\sigma(C_*,D_*)$ is always defined.

\begin{definition} \label{def:ref-2-complex}
Fix an algebraic $n$-complex $\ol{C}_*$, which we will refer to as the \textit{reference complex}. Let $L := H_n(\ol{C}_*)$, let $\wh L = H_n(\ol{C}_*)\hh$ and let
 $\psi \colon L^G \twoheadrightarrow \wh L$ denote the quotient map.  

Let $C_*$ be an algebraic $n$-complex such that $\chi(C_*) = \chi(\ol{C}_*)$ and let 
\[ \psi_{C_*}\colon H_n(C_*)^G \twoheadrightarrow H_n(C_*)\hh\] 
denote the quotient map. Since $L^G \cong H_n(C_*)^G$ and $\wh L \cong H_n(C_*)\hh$ are isomorphic as abelian groups, we can find isomorphisms $\ol{\tau}_{C_*}$ and $\tau_{C_*}$ such that $\psi_{C_*} \circ \ol{\tau}_{C_*} = \tau_{C_*} \circ \psi$. 
 We do this by choosing bases for $L^G$ and $H_n(C_*)^G$ whose images under the quotient maps give generating sets for invariant factor decompositions of the finite abelian groups $\wh L$ and $H_n(C_*)\hh$ (see \cref{remark:wlog-psi-diagonal} for more details). The reference maps $\ol{\tau}_{C_*}$ and $\tau_{C_*}$ should be regarded as fixed once and for all.
\end{definition}

Before defining the bias, we will first need the following definition. 

\begin{definition} \label{def:psi-auto}  For abelian groups $A_1$ and $A_2$, we will let $\Hom(A_1,A_2)$ (resp. $\Iso(A_1,A_2)$) denote the set of homomorphisms (resp. isomorphisms) from $A_1$ to $A_2$.

Let $\psi_1\colon A_1 \twoheadrightarrow B_1$ and $\psi_2\colon A_2 \twoheadrightarrow B_2$ be surjective homomorphisms for abelian groups $B_1$ and $B_2$.
We say an homomorphism $\varphi \in \Hom(A_1,A_2)$ is a \textit{$(\psi_1,\psi_2)$-homomorphism} if $\varphi(\ker(\psi_1)) \subseteq \ker(\psi_2)$.
Let $\Hom_{\psi_1,\psi_2}(A_1,A_2) \subseteq \Hom(A_1,A_2)$ denote the subset consisting of $(\psi_1,\psi_2)$-homomorphisms. 
There is an induced function
\[ (\psi_1,\psi_2)_*\colon \Hom_{\psi_1,\psi_2}(A_1,A_2)  \to \Hom(B_1,B_2), \quad \varphi \mapsto (x \mapsto \psi_2(\varphi(\wt x))) \]
where $\wt x \in A_1$ is any lift of $x \in B_1$, i.e. $\psi_1(\wt x) = x$.

Similarly, an isomorphism $\varphi \in \Iso(A_1,A_2)$ is a \textit{$(\psi_1,\psi_2)$-isomorphism} if $\varphi(\ker(\psi_1)) = \ker(\psi_2)$, and we let $\Iso_{\psi_1,\psi_2}(A_1,A_2) \subseteqq \Iso(A_1,A_2)$ denote the subset consisting of $(\psi_1,\psi_2)$-isomorphisms. The function $(\psi_1,\psi_2)_*$ defined above restricts to a function $\Iso_{\psi_1,\psi_2}(A_1,A_2)  \to \Iso(B_1,B_2)$.

In the case where $A_1=A_2=:A$, $B_1=B_2=:B$ and $\psi_1=\psi_2=:\psi$, a $(\psi,\psi)$-isomorphism will be referred to as a \textit{$\psi$-automorphism}. The set of $\psi$-automorphisms forms a subgroup $\Aut_\psi(A) \le \Aut(A)$ and the induced function $\psi_*$ is a group homomorphism
$\psi_*\colon \Aut_{\psi}(A)  \to \Aut(B)$.
\end{definition}

\begin{definition} \label{def:bias-schafer}
Fix a reference complex $\ol{C}_*$ with $(-1)^n\chi(\ol{C}_*) = \chi$ and let $L = H_n(\ol{C}_*)$. Define the \textit{polarised bias obstruction group} to be
\[ PB(G,n,\chi) := \frac{\Aut(\wh L)}{\Aut_\psi(L^G)}. \]
This depends on $\chi$ but, up to isomorphism, does not depend on the choice of reference complex $\ol{C}_*$ (see \cref{def:ref-2-complex}). If $\chi = \chi_{\min}(G,n)$, then we define $PB(G,n) := PB(G,n,\chi)$.

Let $C_*$, $D_*$ be algebraic $n$-complexes such that $(-1)^n\chi(C_*) = (-1)^n\chi(D_*) = \chi$.
	Define the \textit{bias invariant} to be:
\[ \beta(C_*,D_*) := [\tau_{D_*}^{-1} \circ \sigma(C_*,D_*) \circ \tau_{C_*}] \in PB(G,n,\chi). \] 
where $[\,\cdot\,]\colon \Aut(\wh L) \twoheadrightarrow PB(G,n,\chi)$ is the quotient map and $\tau_{C_*}$, $\tau_{D_*}$ are as above.
\end{definition}

It follows from \cref{prop:bias-tate} that the vanishing of $\beta(C_*,D_*)$ in its respective obstruction group does not depend on the choice of reference complex.

We will now establish the following, where $m_{(G,n)}$ is as defined in \cref{def:modulus}. This implies that the bias invariant contains no information in the non-minimal case.

\begin{proposition} \label{prop:PB(Gnl)} \mbox{}
\begin{clist}{(i)}
\item There is an isomorphism $PB(G,n) \cong (\Z/m)^\times/\{\pm 1\}$ where $m =m_{(G,n)}$. 
\item If $\chi > \chi_{\min}(G,n)$, then $PB(G,n,\chi) = 0$.
\end{clist}
\end{proposition}

This is a consequence of the following, which was pointed out by Webb \cite[Corollary 3.2]{Webb:1981}.  In what follows, we write $a \mid b$ for integers $a, b \ge 1$ if $a$ divides $b$. We write $m_1 \mid \cdots \mid m_d$ for integers $m_1, \cdots, m_d \ge 1$ if $m_i \mid m_{i+1}$ for all $1 \le i \le d-1$.

\begin{lemma} \label{lemma:webb}
Let $A$ be a finite abelian group, let $d \ge 1$ and let $\psi\colon \Z^d \twoheadrightarrow A$ be a surjective homomorphism. 
Suppose $A$ has invariant factors $m_1 \mid \cdots \mid m_d$ (possibly with some $m_i=1$),  i.e. $A \cong \Z/m_1 \times \cdots \times \Z/m_d$.
Consider the following composition:
\[ \rho\colon \Aut(A) \to \Aut(\Z/m_1 \otimes_{\Z} A) \xrightarrow[]{\det} (\Z/m_1)^\times \twoheadrightarrow (\Z/m_1)^\times/\{\pm 1\}.  \]
Then $\rho$ is surjective and $\IM(\psi_*\colon \Aut_\psi(\Z^d) \to \Aut(A)) = \ker(\rho)$. In particular, we have that $\IM(\Aut_\psi(\Z^d)) \unlhd \Aut(A)$ is a normal subgroup and $\rho$ induces an isomorphism
\[ \rho_*\colon \frac{\Aut(A)}{\Aut_\psi(\Z^d)} \to (\Z/m_1)^\times/\{\pm 1\}. \] 
Note that $\rho$ only depends on $A$ and $d$, and not on the choice of $\psi$.
\end{lemma}

\begin{proof}[Proof of \cref{prop:PB(Gnl)}]
Let $\ol{C}_*$ be a reference complex with $(-1)^n\chi(\ol{C}_*) = \chi$, let $L = H_n(\ol{C}_*)$, let $\psi \colon L^G \twoheadrightarrow \wh L$ denote the quotient map, let $d= \rank_{\Z}(L^G)$ and suppose $\wh L$ has invariant factors $m_1 \mid \cdots \mid m_d$ (possibly with some $m_i=1$). Then \cref{lemma:webb} implies that $PB(G,n,\chi) \cong (\Z/m_1)^\times/\{\pm 1\}$.

(i) If $\chi = \chi_{\min}(G,n)$, then $r_{(G,n)} = \rank_{\Z}(L^G)$ and so $m_1 = m_{(G,n)}$ (see \cref{def:modulus}).

(ii) Suppose $\chi > \chi_{\min}(G,n)$, so that $r := \chi - \chi_{\min}(G,n) > 0$. Let $\ol{D}_*$ be a reference complex
 with $(-1)^n\chi(\ol{D}_*) = \chi_{\min}(G,n)$, and let $L_0 = H_n(\ol{D}_*)$. Since $L^G$ and $\wh L$ depend only on $\chi$ and not on $\ol{C}_*$ (see \cref{def:ref-2-complex}), we have that $L^G \cong (L_0 \oplus \Z G^r)^G \cong L_0^G \oplus \Z^r$ and $\wh L \cong (L_0 \oplus \Z G^r)\hh \cong \wh L_0$. Since $\rank_{\Z}(L_0^G) \ge d(\wh L_0)$, it follows that $\rank_{\Z}(L^G) \ge d(\wh L) + r > d(\wh L)$. Hence $m_1=1$.
\end{proof}

\begin{remark} \label{remark:wlog-psi-diagonal}
The quotient map $\psi\colon  L^G \twoheadrightarrow \wh L$ is induced by the identification $\wh L \cong L^G/(N \cdot L)$.
If $d = \rank_{\Z}(L^G)$, then we have isomorphisms $L^G \cong \Z^d$ and $\wh L \cong \Z/m_1 \times \cdots \times \Z/m_d$ for some $m_i$ such that $m_i \mid m_{i+1}$ for all $i$.
In fact, we can choose these isomorphisms such that the induced map $\psi\colon  \Z^d \twoheadrightarrow \Z/m_1 \times \cdots \times \Z/m_d$ is a direct sum of quotient maps $\Z \twoheadrightarrow \Z/m_i$ for $1 \le i \le d$. 
To see this, note that  since $\wh L$ is a finite group,   the kernel $\ker\psi$ of the surjection $\psi\colon  L^G \twoheadrightarrow \wh L$ must be of the form  $\ker \psi \cong \Z^d$.

Since $\Z$ is a principal ideal domain, we can choose bases of $\ker \psi$ and $L^G$ so that the inclusion map $\ker\psi \hookrightarrow L^G$ is given by a $d \times d$ matrix with non-zero diagonal entries $n_1, \cdots, n_d$ such that $n_{i} \mid n_{i+1}$ for all $i$. By comparing cokernels, it follows that $n_i=m_i$ for all $i$.
By taking the induced generating set for $\wh L$, we obtain identifications such that $\psi$ is as required.
\end{remark}

In light of \cref{prop:PB(Gnl)} (ii), we will now restrict to the case of minimal complexes.
The bias invariant has the following two basic properties. The first can be extracted from \cite[Theorem 1.13]{Latiolais:1993}, and the second is a consequence of the independence of the choice of chain map.

\begin{proposition} \label{prop:bias-an-invariant}
The bias invariant is a chain homotopy invariant. In particular, if $D_*$ is an algebraic $n$-complex with $(-1)^n\chi(D_*)=\chi_{\min}(G,n)$, and $m = m_{(G,n)}$, then the bias invariant defines a map
\[ \beta\colon \Alg_{\min}(G,n) \to (\Z/m)^\times/\{\pm 1\}, \quad C_* \mapsto \beta(C_*,D_*). \]
\end{proposition}

\begin{proposition} \label{lemma:bias-composition}
Let $C_*$, $D_*$, $E_* \in \Alg_{\min}(G,n)$ and $m = m_{(G,n)}$, then 
\[ \beta(C_*,E_*) = \beta(C_*,D_*) \cdot \beta(D_*,E_*) \in (\Z/m)^\times/\{\pm 1\}. \]
\end{proposition}

We conclude this section by noting that Metzler's original formulation differs to the one presented above. We will present it here since it will be useful for the explicit computations in \cref{ss:examples-non-abelian}. A convenient reference for this formulation is \cite{Latiolais:1993}.

Observe that, if $C_*$, $D_* \in \Alg_{\min}(G,n)$, then $H_n(\Z \otimes_{\Z G} C_*)$ and $H_n(\Z \otimes_{\Z G} D_*)$ are free abelian groups of the same rank. The following is a consequence of \cite[Lemma 1.8 \& Exercise 1.10]{Latiolais:1993} and \cite[p163]{Schafer:1996}, and gives another formulation of the bias invariant.

We use the identification $PB(G,n) \cong (\Z/m)^\times/\{\pm 1\}$ given in \cref{prop:PB(Gnl)}.

\begin{proposition} \label{prop:bias-versions}
Let $C_*$, $D_*$ be algebraic $n$-complexes such that $(-1)^n\chi(C_*) = (-1)^n\chi(D_*)=\chi_{\min}(G,n)$ and let $h\colon C_* \to D_*$ be a chain map such that $H_0(h) = \id_{\Z}$. Fix identifications $H_n(\Z \otimes_{\Z G} C_*) \cong \Z^r$ and $H_n(\Z \otimes_{\Z G} D_*) \cong \Z^r$, and view $H_n(\id_{\Z} \otimes h)\colon H_n(\Z \otimes_{\Z G} C_*) \to H_n(\Z \otimes_{\Z G} D_*)$ as an element of $M_r(\Z)$. Let $m = m_{(G,n)}$ and let $\ol{\cdot}\colon \Z \to (\Z/m)/\{\pm 1\}$ denote the quotient map. Then:
\[ \beta(C_*,D_*) = \ol{\det(H_n(\id_{\Z} \otimes h))} \in (\Z/m)^\times/\{\pm 1\}. \]
\end{proposition}

\subsection{The bias invariant for $(G,n)$-complexes} \label{ss:bias-complexes}

The following is implicit in \cite[Definition 1.11]{Latiolais:1993}. 
The proof is a consequence of \cref{lemma:bias-composition} and is omitted for brevity.

\begin{proposition} \label{prop:bias-Aut(G)-action}
Let $G$ be a finite group and let $n \ge 2$. If $X$ is a finite $(G,n)$-complex with $(-1)^n\chi(X) = \chi$, then the map
\[ \varphi_{(G,n,\chi)} \colon  \Aut(G) \to PB(G,n,\chi), \quad \theta \mapsto \beta(C_*(\wt X),C_*(\wt X)_\theta) \]
is a group homomorphism and is independent of the choice of $X$. If $\chi = \chi_{\min}(G,n)$, then we define $\varphi_{(G,n)} := \varphi_{(G,n,\chi)}$.
\end{proposition}

We will now use \cref{prop:he-vs-che} to adapt the notion of bias to finite $(G,n)$-complexes. 

\begin{definition}[Bias invariant for $(G,n)$-complexes]\label{def:bias-original}
Let $\varphi_{(G,n,\chi)}$ be the group homomorphism given in \cref{prop:bias-Aut(G)-action}. Define
\[ D(G,n,\chi) := \IM(\varphi_{(G,n,\chi)}\colon \Aut(G) \to PB(G,n,\chi)) \le PB(G,n,\chi) \] 
and define the \textit{bias obstruction group} for $(G,n)$-complexes to be 
\[B(G,n,\chi) := \frac{PB(G,n,\chi)}{D(G,n,\chi)}. \]
If $\chi = \chi_{\min}(G,n)$, then we define $D(G,n) := D(G,n,\chi)$ and $B(G,n) := B(G,n,\chi)$. When $n=2$, we write $B(G) := B(G,2)$ and $D(G):=D(G,2)$.

Let $X$, $Y$ be finite $(G,n)$-complexes with $(-1)^n\chi(X)=(-1)^n\chi(Y)=\chi$. Define the \textit{bias invariant} to be
\[ \beta(X,Y) := [\beta(C_*(\wt X), C_*(\wt Y))] \in B(G,n,\chi) \]
where $[\, \cdot \, ]\colon PB(G,n,\chi) \twoheadrightarrow B(G,n,\chi)$ is the quotient map. 
\end{definition}

By \cref{prop:PB(Gnl)} (ii), we have that $B(G,n,\chi)=0$ for $\chi  > \chi_{\min}(G,n)$. In particular, if $X$ and $Y$ are finite $(G,n)$-complexes with $(-1)^n\chi(X) = (-1)^n\chi(Y) > \chi_{\min}(G,n)$, then $\beta(X,Y)=0$.
We will therefore primarily be interested in \emph{minimal} finite $(G,n)$-complexes $X$, for which $(-1)^n \chi(X) = \chi_{\min}(G,n)$ has the minimum possible value (see \cref{def:minimal-Gn}). 
We have:

\begin{proposition}
If $\ol{X}$ is a reference minimal $(G,n)$-complex, then the bias defines a map
\[ \beta\colon\HT_{\min}(G,n) \to B(G,n) \cong \frac{(\Z/m)^\times}{\pm D(G,n)} , \quad X \mapsto \beta(X,\ol{X}). \]
Thus, $\beta$ is an invariant of minimal $(G,n)$-complexes up to homotopy equivalence.
\end{proposition}

It follows from work of Metzler \cite{Metzler:1976}, Sieradski \cite{Sieradski:1977}, Sieradski-Dyer \cite{Sieradski:1979}, Browning \cite{Browning:1979c,Browning:1979} and Linnell \cite{Linnell:1993} that, for any finite abelian group $G$, the bias invariant completely classifies minimal $(G,n)$-complexes up to homotopy equivalence (\cref{thm:bias-surj-abelian}).
More details on the computation of the bias invariant, particularly for finite abelian groups, can be found in \cref{s:examples-complexes}.

\begin{remark} \label{remark:bias-minimality}
If $(G,n)$ does not satisfy the strong minimality hypothesis (\cref{def:min-hyp}), then \cref{remark:m=1} implies that $m_{(G,n)}=1$ and so $B(G,n)=0$. In particular, the strong minimality hypothesis is a necessary condition for the non-vanishing of the bias invariant.
\end{remark}

\begin{remark}
It is currently open whether or not there exists a finite group $G$ and finite $(G,n)$-complexes $X$, $Y$ with $(-1)^n\chi(X) = (-1)^n\chi(Y) > \chi_{\min}(G,n)$ but $X \not \simeq Y$ (see \cite[Question 7.4]{Ni20-I}). The above shows that the bias invariant cannot be used to distinguish such examples if they exist.
It is known that such examples do not exist if $(-1)^n\chi(X) = (-1)^n\chi(Y) > \chi_{\min}(G,n)+1$ \cite{Dy78,Wi73} or if $n$ is even \cite{Browning:1979}.
\end{remark}

\section{Preliminaries on hermitian forms} 
\label{s:forms}

\subsection{Hermitian forms on $R$-modules}

A convenient reference is Scharlau \cite{Scharlau:1985} (see also \cite[Section 6.2]{Teichner:1992}). 
Let $R$ be a ring with involution and let $L$ be an $R$-module. We will mostly be concerned with  commutative rings $R=  \bZ$ or $\cy m$, with trivial involution,  or $R  =\ZG$, with involution $g \mapsto g^{-1}$, for $g \in G$ a finite group.
A \textit{sesquilinear form} on $L$ is a bilinear form 
\[ h\colon L \times L \to R\] 
such that $h(a \cdot m, b \cdot n) = a \cdot h(m,n) \cdot \ol{b} \in R$ for all $a,b \in R$ and $m,n\in L$. Define $\Sesq(L)$ to be the set of sesquilinear forms on $L$.
This is an abelian group under pointwise addition of functions. The term  \textit{quadratic form}  and notation $(L,q)$ for $q \in \Sesq(L)$ is standard in the algebraic theory. 

Let $T_\ep\colon \Sesq(L) \to \Sesq(L)$, where $\ep \in \{\pm 1\} \subseteq R$, denote the transpose operator, defined by $(Th)(m,n) = \ep \ol{h(n,m)}$, for all $m, n \in L$.
An \textit{$\ep$-hermitian form} on $L$ is a sesquilinear form $h$ such that $T_\ep h = h$, or  $h(m,n) = \ep\ol{h(n,m)}$ for all $m,n \in L$. Define $\Herm_\ep(L) = \ker(1-T_\ep)$ to be the set of $\ep$-hermitian forms on $L$, which is a subgroup of $\Sesq(L)$. The standard terminology is \emph{symmetric} for $\ep = +1$ and \emph{skew-symmetric} for $\ep = -1$.
When $\ep = +1$, we will just write  $\Herm(L)$ to simplify the notation.

Let $\ad\colon \Sesq(L) \to \Hom_R (L,L^*)$, $h \mapsto (m \mapsto h(-,m))$, 
denote the \emph{adjoint} map, where we consider  $L^*$ as a left  $R$-module via the involution on $R$.
We say that an  $\ep$-hermitian form $h \in \Herm_\ep(L)$ is \textit{non-singular} if $\ad(h)$ is an $R$-isomorphism.  

\begin{definition}\label{def:meta}
Let $L$ be an $R$-module and let $h \in \Herm_\ep(L)$. Then the associated \textit{metabolic form} is
\[ \Met_\ep (L,h)\colon (L^* \oplus L) \times (L^* \oplus L) \to R, \quad (\varphi_1,m_1),(\varphi_2,m_2) \mapsto \ep\ol{\varphi_1(m_2)} + \varphi_2(m_1)+h(m_1,m_2). \]
We have $\Met_\ep (L,h) \in \Herm_\ep(L^* \oplus L)$.
In matrix notation, we can write this as $\Met_\ep(L,h) = \left(\begin{smallmatrix} 0 & 1 \\ \ep & h \end{smallmatrix}\right)$ where the off-diagonal maps are understood to denote the maps induced by evaluation $(\varphi, m) \mapsto \varphi(m)$.
The \textit{hyperbolic form} on $L$ is defined as $H_\ep(L) := \Met_\ep (L,0)$.
\end{definition}

A form $ h \in \Herm_\ep(L)$ is called \emph{even} if $h = q + T_\ep q \in \IM(1+T_\ep)$ for some $q \in \Sesq(L)$. In this case, we say the $h$ admits a quadratic refinement $q$. A form $h \in  \Herm_\ep(L)$ is called \emph{weakly even} if $h(m,m) = a + \ep\bar a$, for some $a \in R$ and all $m \in L$. These notions are equivalent if $L$ is a projective $R$-module, but not in general. 

In the other direction,  a form  $q \in \Sesq(L)$ is called a \emph{non-singular $\ep$-quadratic form} if its associated hermitian form 
$h = q + T_\ep q$ is non-singular.

\subsection{Hermitian forms on $\Z G$-modules} \label{ss:Herm-ZG}

Let $G$ be a finite group and let $L$ be a $\Z G$-module. 
We will now define an alternative notion of dual module.
For a $\Z G$-module $L$, define $L^* := \Hom_{\Z}(L,\Z)$ to be the $\Z G$-module where, for $r \in \Z G$ and $\varphi \in \Hom_{\Z}(L,\Z)$, we let $(r \cdot \varphi)(m) := \varphi(\ol{r} \cdot m)$ for all $m \in L$.  
This coincides with the notion of dual module defined previously via the $\Z G$-isomorphism:
\[ \psi\colon \Hom_{\Z}(L,\Z) \to \Hom_{\Z G}(L,\Z G), \quad \varphi \mapsto \widehat\varphi,  \text{\ with \ } \widehat\varphi(m)=  \sum_{g \in G} \varphi(g^{-1} \cdot m)g ) \]
(see, for example, \cite[VI.3.4]{Brown:1994}). The inverse is given by $\widehat\varphi \mapsto \varepsilon_1\circ \widehat\varphi$, where for $h \in G$, we let $\varepsilon_h\colon \Z G \to \Z$ denote the map $\sum_{g \in G} n_g g \mapsto n_h$.

\begin{definition}\label{def:fourtwo}
 Let $\Sym_\ep(L)$ denote the set of homomorphisms
$$\Sym_\ep(L)= \{\varphi \in \Hom_{\Z}(L \otimes_{\Z} L,\Z)^G : \varphi(m \otimes n) = \ep\varphi(n \otimes m), \text{\ for all\ } m,n \in L\}$$
 where the action of $G$ on $\Hom_{\Z}(L \otimes_{\Z} L,\Z)$ is defined by $(g \cdot \varphi)(m,n) = \varphi(gm,  gn)$.

In other words, $\Sym_\ep(L)$ is the set of $\ep$-symmetric bilinear forms $h \colon L \times L \to \bZ$ such that  $h(g \cdot m,g \cdot n) = h(m,n)$ for all $g \in G$ and $m,n \in L$.
\end{definition}

There is a canonical isomorphism of abelian groups (with inverse induced by $\psi$):
\[ \Xi\colon \Sesq(L) \xrightarrow[]{\cong} \Hom_{\Z}(L \otimes_{\Z} L,\Z)^G, \quad h \mapsto (m \otimes n \mapsto \varepsilon_1(h(m,n))), \]
given by composing with $\varepsilon_1\colon \Z G \to \Z$.
The following is a consequence of the definition of $\Xi$ on restriction to $\Herm_\ep(L) \subseteq \Sesq(L)$.

\begin{proposition} The map $\Xi$ induces a natural correspondence $\Herm_\ep(L)\cong \Sym_\ep(L)$.
\end{proposition}

\begin{convention}
From now on, unless stated otherwise, we will use the natural correspondence $\Xi\colon \Herm_\ep(L)\cong \Sym_\ep(L)$ to work with $\Z$-valued forms rather than $\Z G$-valued forms. In particular, 
$\Met_\ep (L,h)$ and $H_\ep(L)$ will mean the $\bZ$-valued equivalent metabolic or hyperbolic forms.
\end{convention}

 Let $R$ be a commutative ring and let $R G$ denote the group ring with involution $g \mapsto g^{-1}$. Let $\SymCat_\ep(G, R)$ denote the category whose objects are forms $h \in \Sym_\ep(L)$, for some $\RG$-module $L$, with morphisms from $(L,h)$ to $(L',h')$ given by $\RG$-module homomorphisms $f \colon L \to L'$ such that $h'(f(m), f(n)) = h(m,n)$, for all $m,n \in L$.  If $G = 1$ is the trivial group, we will write $\SymCat_\ep(R) := \SymCat_\ep(1, R)$.
 We will later be interested in $\SymCat_\ep(R)$ in the case where $R = \Z$ or $\Z/m$ for $m \in \Z$ (see \cref{def:fixTate}).

 We will now restrict to the case where $R = \Z$.
There is an action of $\Aut(G)$ on $\SymCat_\ep(G, \Z)$ defined as follows. Let $\theta \in \Aut(G)$ and $h \in \Sym_\ep(L)$. Let $L_\theta$ be the $\Z G$-module whose underlying abelian group $L$ and with $G$-action defined by $g\cdot m := \theta(g)\cdot m$, for all $g \in G$ and $m \in L$. Then $h_\theta \in \Sym_\ep(L_\theta)$ is the induced form on $L_\theta$. 
 
\begin{example} \label{example:intform}
Let $G$ be a finite group. If $M$ is a closed oriented $4$-dimensional Poincar\'{e} complex with $\pi_1(M) \cong G$, then the equivariant intersection form $S_M$ is a hermitian form on $\pi_2(M)$. That is, $S_M \in \Herm(\pi_2(M))$. Using the convention above, we will view this as a bilinear form
\[ S_M\colon \pi_2(M) \times \pi_2(M) \to \Z \]
which is $G$-invariant and symmetric, so that $S_M \in \Sym(\pi_2(M))$.
The hermitian form $S_M$ satisfies the Bredon condition  \cite[Theorem 7.4]{Bredon:1972} that $S_M(\tau \cdot m,m) \equiv 0 \mod 2$ for all $m \in \pi_2(M)$ and $\tau \in G$ of order two. This leads to the observation that the $\Z G$-valued hermitian form is weakly even 
if  the universal covering of $M$ is a spin manifold. Note that an even metabolic form is isometric to a hyperbolic form, but this is not true for weakly even metabolic forms.
\end{example}

The following result is well-known.  Recall that a $\Z G$-lattice is a $\Z G$-module which is finitely generated and torsion free as an abelian group.

\begin{proposition}\label{prop:oddmetabolic}
Let $L$ be an $\ZG$-lattice, where $G$ is a finite group of odd order.  If $(L,h)$ is a weakly even form, then $(L,h)$ is even and the metabolic form $\Met_\ep(L, h)$ is hyperbolic.
\end{proposition}

\begin{proof} A metabolic form $\Met_\ep(L,h)$ is weakly even if the form $(L, h)$ is a weakly even $\ep$-symmetric  form on the $\La =\ZG$ module $L$, where $L$ is a $\ZG$-lattice.
Since $h$ is $\ep$-symmetric, it defines a class
$$[h] \in \wH^0(\cy 2; \Sesq(L))$$
which vanishes if and only if $h = q + T_\ep q$ for some $q \in \Sesq(L)$.  
Since the obstruction group is an $\mathbb F_2$ vector space, and the reduction mod 2 map 
$$\wH^0(\cy 2; \Sesq(L)) \to \wH^0(\cy 2; \Sesq(L/2))$$
is an injection (since $\Sesq(L)$ is torsion free as an abelian group), it follows that $(L, h)$ is an even form if and only if its reduction
$(L/2, \bar h)$ is an even form with values in $\La/2 = \mathbb F_2[G]$. 

 If $G$ has odd order, the ring $ \mathbb F_2[G]$ is semisimple \cite[Theorem 3.14]{CR81}, so the $\La/2$-module $L/2$ is projective. Hence $(L/2, \bar h)$ is an even form by \cite[Proposition 3.4]{Bass:1973}, and so is $h$.
\end{proof}

\begin{definition}\label{def:fixTate} The \emph{fixed-point functor} $\SymCat_\ep(G, \bZ) \to \SymCat_\ep(\bZ)$ is defined on objects by $h \mapsto h^G$, where $h^G\in \Sym_\ep(L^G)$ is the restriction of the form $h$ to the fixed set $L^G \subset L$.  The  \emph{Tate functor} $\SymCat_\ep(G, \bZ) \to \SymCat_\ep(\cy {|G|})$ is defined on objects by $h \mapsto \wh h$ where 
$\wh h(\wbar{m},\wbar{n}) = [h^G(m,n)] \in \cy {|G|}$ for all $m,n \in L^G$, and where $\wbar{\cdot}$ denotes the quotient map $L^G \to L^G/(N \cdot L) = \wh H^0(G;L)$ and $[\,\cdot\,]$ denotes the quotient map $\bZ \to \cy {|G|}$.
The functors can be defined on the morphisms in the natural way.
\end{definition}

Note that $\wh \Z = \wh H^0(G;\Z) \cong \Z/|G|$ is a ring and, if $L$ is a $\Z G$-module, then $\wh L$ is a $\wh \Z$ module in a natural way. For example, we can take the action induced by the cup product
\[ \wh H^0(G;\Z) \times \wh H^0(G;L) \xrightarrow[]{- \cup -} \wh H^0(G;\Z \otimes_{\Z} L) \xrightarrow[]{\cong} \wh H^0(G;L). \]

The following is \cite[Exercise VI.7.3]{Brown:1994}.

\begin{proposition} \label{prop:tate-identify}
Let $L$ be a $\Z G$-lattice. Then the cup product and the evaluation map $L^* \otimes_{\Z} L \to \Z$, $\varphi \otimes m \mapsto \varphi(m)$ induce a non-singular duality pairing
\[ \wh H^0(G;L^*) \times \wh H^0(G;L) \xrightarrow[]{-\cup -} \wh H^0(G;L^* \otimes_{\Z} L) \to \wh H^0(G;\Z). \]
In particular, there is an isomorphism of abelian groups $\wh H^0(G;L^*) \cong \Hom_{\Z}(\wh H^0(G;L),\Z/|G|)$.
\end{proposition}

If $A$ is a $\Z/|G|$-module, then we will write $A^* := \Hom_{\Z}(A,\Z/|G|)$. The above shows that, if $L$ is a $\Z G$-lattice, then there is a canonical identification 
 $\theta\colon (L^*)\hh \cong (\wh L)^*$. If $f\colon L _1 \to L_2$ is an isomorphism of $\ZG$-lattices, then the duality pairing implies that $(f^*)^{-1}\oplus f \colon L_1^* \oplus L_1 \to L_2^* \oplus L_2$ induces an isomorphism commuting with the identifications $(L_i^*)\hh \cong (\wh L_i)^*$, for $i = 1,2$.

\begin{proposition} \label{prop:two-hat-forms}
Let $h \in \Sym_\ep(L)$ be a $\Z$-valued hermitian form on a $\Z G$-module $L$. Then the form $\wh h \in \Sym_\ep(\wh H^0(G;L))$ defined in \textup{\cref{def:fixTate}} coincides with the form induced by the cup product on Tate cohomology:
\[ \wH^0(G;L) \times \wH^0(G;L) \xrightarrow{\ \cup \ } \wH^0(G; L \otimes_\bZ L) \xrightarrow{\ h \ } \wH^0(G; \bZ) \cong \cy {|G|}.\]
If $L$ is a $\Z G$-lattice and $h$ is non-singular, then $h^G$ is non-degenerate and $\wh h$ is non-singular. 
\end{proposition}

\begin{proof} The cup product is induced by the composite
$ L^G \otimes L^G \subseteq (L \otimes L)^G \xrightarrow{h} \bZ$, 
and the form $\wh h$ is induced by the restriction 
$h^G \colon L^G \otimes L^G \to \bZ$ of $h$ to the fixed subspaces. Hence the form $\wh h$ coincides with the form induced by the cup product. 

If $L$ is a $\Z G$-lattice and $h$ is non-singular, then for any non-zero $x \in L^G$, there exists $y \in L$ such that $h(x,y) \neq 0$, implying that $h(x, Ny) = h(Nx,y) = |G|\cdot h(x,y)  \neq 0$, Hence $h^G$ is 
non-degenerate.  To show that $\wh h \colon \wh L \times \wh L \to \cy{|G|}$ is non-singular, we consider its adjoint $\ad \wh h \colon \wh L \to (\wh L)^*$. By  \cref{prop:tate-identify}, there is an isomorphism $\theta\colon (L^*)\hh \xrightarrow{\ \cong\ }(\wh L)^*$, such that
$\ad \wh h = \theta\circ (\ad h)\hh$. To check this formula, let $i \colon L^G \to L$ denote the inclusion, and note that $\theta(\phi)(\bar y) =\overline{(i^*\phi)(y)} \in \cy{|G|}$, for any $\phi \in   (L^*)^G$ and $y \in L^G$, where $i^* \colon (L^*)^G \to (L^G)^*$ is the induced restriction map. Then for $x \in L^G$, we set $\phi := \ad h(x) \in (L^*)^G$ and see that $(i^* \ad h)(x)(y) = h(x,y)$, implying the formula above. 

 Since  $h$ is non-singular, $\ad h$ is an isomorphism, and hence  
$(\ad h)\hh $ is also an isomorphism. Therefore $\ad \wh h=\theta\circ (\ad h)\hh$ is an isomorphism and $\wh h$ is non-singular. 
\end{proof}

\subsection{Evaluation forms}

Let $G$ be a finite group and let $L$ be a $\Z G$-module. The \textit{evaluation form} on $L$ is the bilinear form 
\[ e_L\colon (L^* \oplus L) \times (L^* \oplus L) \to \Z, \quad (\varphi_1,m_1), (\varphi_2,m_2) \mapsto \ep\varphi_1(m_2)+\varphi_2(m_1),\]
where $L^* = \Hom_\bZ(L, \bZ)$ and $e_L \in \Sym_\ep(L^* \oplus L)$, for $\ep \in \{\pm 1\}$ is an $\ep$-symmetric form.
When the module $L$ is understood from the context, we will often write $e = e_L$. This coincides with the hyperbolic form previously defined, so that $e_L $ is isometric to $H_\ep(L)$ and  $\Met_\ep(L, \phi) \cong \left(\begin{smallmatrix} 0 & 1 \\ \ep & \phi \end{smallmatrix}\right)$, where $\phi \in \Sym_\ep(L)$. We will now define two related evaluation forms.

\begin{definition}\label{def:eG}
The restriction of  the evaluation form $e_L$  to the fixed set 
$(L^* \oplus L)^G = H^0(G; L^*\oplus L)$ induces  a bilinear form:
\[ e_L^G\colon ((L^*)^G \oplus L^G) \times ((L^*)^G \oplus L^G) \to \Z. \]
This is an $\ep$-symmetric form over the ring $\Z$ with trivial involution,  and coincides with $h^G$ for $h = e_L$ from \cref{def:fixTate}. We write $e_L^G \in \Sym_{\ep}((L^*)^G \oplus L^G)$.
\end{definition}

\begin{definition}\label{def:ehatG}
By functoriality of $\wH^0(G; -)$, and the isomorphism $(L^*)\hh \cong (\wh L)^*$, there is an induced bilinear form:
\[ \wh e_L\colon ((L^*)\hh \oplus \wh L) \times ((L^*)\hh \oplus \wh L) \to \Z/|G|.  \]
This is a non-singular $\ep$-symmetric form of $\Z/|G|$-modules. We write $\wh e_L \in \Sym_{\ep}(\wh{L}^* \oplus \wh L)$. 
\end{definition}

 The form $\wh e_L$ in the definition above coincides with $\wh h$ for $h = e_L$ from \cref{def:fixTate}. This is a consequence of \cref{prop:two-hat-forms}.

The following will be useful later (see \cref{def:fixTate}).
 
\begin{proposition} \label{prop:adapted} 
Let $\phi \in \Sym_\ep(L)$ be an $\ep$-symmetric form such that $\phi^G = 0$. If $h= \Met_\ep(L, \phi)$ is the corresponding metabolic form, then $h^G\cong e_L^G$ and $\wh h\cong \wh e_L$.
\end{proposition}

\begin{proof} 
The first statement is immediate from the assumption that $\phi^G =0$.  It follows from 
\cref{def:fixTate} that $\wh h$ is the reduction of $h^G$, and so we also have $\wh h\cong \wh e_L$.
\end{proof}

 The following summary is the basis for the explicit models for $e_L^G$ and $\wh e_L$ we give in \cref{prop:sevenfour}.
 
\begin{lemma} \label{lemma:evaluation-induced=evaluation}
Under the inclusion map $(L^*)^G \subseteq  (L^G)^*$ induced by restriction, $e_L^G$ is the restriction of the evaluation form
\[ e_{L^G}: ((L^G)^* \oplus L^G) \times ((L^G)^* \oplus L^G) \to \Z, \quad (\varphi_1,m_1), (\varphi_2,m_2) \mapsto \varphi_1(m_2)+\varphi_2(m_1). \]

Under the canonical identification $(L^*)\hh \oplus \wh L \cong (\wh L)^* \oplus \wh L$, 
the form $\wh e_L$ corresponds to the evaluation form
	\[ e_{\wh L}\colon((\wh L)^* \oplus \wh L) \times ((\wh L)^* \oplus \wh L) \to \Z, \quad (\varphi_1,m_1), (\varphi_2,m_2) \mapsto \varphi_1(m_2)+\varphi_2(m_1). \]
\end{lemma}

\begin{proof} This follows directly from the definitions above and \cref{prop:two-hat-forms}.
\end{proof}

\subsection{Matrix representations for the  evaluation forms}
Let $G$ be a finite group, let $L$ be a $\Z G$-lattice and let $e = e_{L}$ be the non-singular evaluation form.
The following is a consequence of \cite[Propositions III.1 \& III.2]{Kreck:1984}.

\begin{lemma} \label{lemma:G*-subtle}
Suppose $L^G \cong \Z^d$ for some $d \ge 0$ and that $\wh L$ has invariant factors $n_d \mid \cdots \mid n_1$ (possibly with some $n_i=1$), i.e. $\wh L \cong \Z/n_1 \times \cdots \times \Z/n_d$. Then there is an exact sequence of abelian groups:
\[ 0 \to (L^*)^G \xrightarrow[]{\rest} (L^G)^* \to \Z/\beta_1 \times \cdots \times \Z/\beta_d \to 0 \]
where $\rest\colon \varphi \mapsto \varphi \mid_{L^G}$ is the restriction map and $\beta_i\colon = |G|/n_i$ for $1 \le i \le d$ so that $\beta_1 \mid \cdots \mid \beta_d$.
\end{lemma}

 By combining this with \cref{lemma:evaluation-induced=evaluation}, we now have the following (which can be found in the discussion on \cite[p24]{Kreck:1984}). 

\begin{proposition}\label{prop:sevenfour}
Let $n_1, \cdots, n_d$ and $\beta_1, \cdots, \beta_d$ be as in \cref{lemma:G*-subtle}. Then, by choosing bases for $(L^*)^G$ and $(L^G)^*$ such that $(L^*)^G = \beta_1 \Z \oplus \cdots \beta_d \Z \subseteq \Z^d = (L^G)^*$, the bilinear form $e^G$ has the matrix representation: \[ 
e^G\colon \underbrace{(\Z^d \oplus \Z^d)}_{\cong (L^*)^G \oplus L^G} \times \underbrace{(\Z^d \oplus \Z^d)}_{\cong (L^*)^G \oplus L^G} \to \Z , \quad
e^G = 
\left(\begin{array}{c|c} 
	0 & \begin{smallmatrix} \beta_1 & & \\ & \ddots & \\ & & \beta_d \end{smallmatrix} 
 \\ 
 \vspace{-5mm}
 \\
	\hline 	
\vspace{-5mm}	
\\
\begin{smallmatrix} \beta_1 & & \\ & \ddots & \\ & & \beta_d \end{smallmatrix}
 & 0 
\end{array}\right)
\]

We can choose bases for $(L^*)\hh$ and $\wh L$ such that $\wh e$ has matrix representation: 
\smallskip
\[
 \wh e\colon \underbrace{\textstyle (\bigoplus_{i=1}^d \Z/n_i  \oplus \bigoplus_{i=1}^d \Z/n_i)}_{\cong (L^*)\hh \oplus \wh L} \times \underbrace{\textstyle (\bigoplus_{i=1}^d \Z/n_i  \oplus \bigoplus_{i=1}^d \Z/n_i)}_{\cong (L^*)\hh \oplus \wh L} \to \Z/|G| , 
 \]
\smallskip
 \[
\wh e = 
\left(\begin{array}{c|c} 
	0 & \begin{smallmatrix} \beta_1 & & \\ & \ddots & \\ & & \beta_d \end{smallmatrix} 
 \\ 
 \vspace{-5mm}
 \\
	\hline 	
\vspace{-5mm}	
\\
\begin{smallmatrix} \beta_1 & & \\ & \ddots & \\ & & \beta_d \end{smallmatrix}
 & 0 
\end{array}\right).
\vspace{2mm}
\]
\end{proposition}

\begin{remark}
(a)
That $\wh e$ is non-singular follows from \cite[Theorem III.1]{Kreck:1984} and the fact that $e$ is non-singular. However, it also follows directly from \cref{lemma:evaluation-induced=evaluation}. 

(b)
In the matrix representation for $\wh e$, the non-zero entries $\beta_i$ correspond to maps $\Z/n_i \times \Z/n_i \to \Z/|G|$, $(x,y) \mapsto \beta_ixy$. This is well-defined since $\beta_i = |G|/n_i$ and so, if $x \equiv x' \mod n_i$ and $y\equiv y' \mod n_i$, then $\beta_ixy \equiv \beta_ix'y' \mod |G|$.
\end{remark}

\subsection{Isometries of evaluation forms} \label{ss:isom(e)}

The following definitions are motivated by the definitions made in \cref{ss:bias-algebraic}. The first is analogous to \cref{def:psi-auto}. Even more generally, we could make both definitions for homomorphisms instead of automorphisms. Many of the results we need extend to this setting (see, for example, \cite[Lemma 3 (\S 2)]{Schafer:1996}). 

\begin{definition} \label{def:psi-isom}
For $i=1,2$, let $\L_i$ be a $\Z G$-lattice, let $e_i = e_{\L_i}$ be its evaluation form and let $\psi_i\colon \L_i^G \twoheadrightarrow \wh \L_i$ and $\psi_i'\colon (\L_i^*)^G \twoheadrightarrow (\L_i^*)\hh$ denote the canonical reduction maps. Let $\Psi_i = \psi_i' \oplus \psi_i$.

We say an isometry $\varphi \in \Isom(e_1^G,e_2^G)$ is a \textit{$(\Psi_1,\Psi_2)$-isometry} if $\varphi(\ker(\Psi_1)) = \ker(\Psi_2)$.
Let $\Isom_{\Psi_1,\Psi_2}(e_1^G,e_2^G) \subseteq \Isom(e_1^G,e_2^G)$ denote the subset consisting of $(\Psi_1,\Psi_2)$-isometries. There is an induced function
\[ (\Psi_1,\Psi_2)_*\colon \Isom_{\Psi_1,\Psi_2}(e_1^G,e_2^G)  \to \Isom(\wh e_1, \wh e_2), \quad \varphi \mapsto (x \mapsto \Psi_2(\varphi(\wt x))) \]
where $\wt x \in (\L_1^*)^G \oplus \L_1^G$ is any lift of $x \in (\L_1^*)\hh \oplus \wh \L_1$, i.e. $\Psi_1(\wt x) = x$.

In the case where $\L_1=\L_2=:\L$, write $e=e_i$, $\psi=\psi_i$ and $\Psi=\Psi_i$ for $i=1,2$. A $(\Psi,\Psi)$-isometry will be referred to as a \textit{$\Psi$-isometry}. The set of $\Psi$-isometries forms a subgroup $\Isom_{\Psi}(e^G) \le \Isom(e^G)$ and the induced function $(\Psi,\Psi)_*$ is a group homomorphism
$\Psi_*\colon \Isom_{\Psi}(e^G) \to \Isom(\wh e)$.
\end{definition}

Before making the next definition, we will start by establishing the following.

\begin{lemma} \label{lemma:G-map-restriction}
For $i=1,2$, let $\L_i$ be a $\Z G$-lattice and let $\psi_i\colon \L_i^G \twoheadrightarrow \wh \L_i$ and $\psi_i'\colon (\L_i^*)^G \twoheadrightarrow (\L_i^*)\hh$ denote the canonical reduction maps. 
 Let $f \in \Hom_{\psi_1,\psi_2}(\L_1^G,\L_2^G)$.
Then there exists a unique map $g \in \Hom_{\psi_1',\psi_2'}((\L_2^*)^G,(\L_1^*)^G)$ such that we have a commutative diagram

\[
\begin{tikzcd}
(\L_2^*)^G	\ar[d, "g"'] \ar[r,"\rest"] & (\L_2^G)^* \ar[d,"f^*"'] \\
(\L_1^*)^G	\ar[r,"\rest"] & (\L_1^G)^*
\end{tikzcd}		
\]
where $\rest$ denotes the restriction maps as used in \cref{lemma:G*-subtle}.   If $f \in \Iso_{\psi_1,\psi_2}(\L_1^G,\L_2^G)$, then 
$g \in \Iso_{\psi_1',\psi_2'}((\L_2^*)^G,(\L_1^*)^G)$.
\end{lemma}

\begin{proof}
This is a consequence of the proof of \cite[Proposition III.6]{Kreck:1984}. 
Simply note that dualising the commutative diagram given there gives the commutative diagram we require. This works since all modules involved are $\Z G$-lattices and so double dualising returns the original module.
\end{proof}

From now on, let $e_i$, $\Psi_i$ be as defined above for $i=1,2$.
The following is defined in \cite[p30-31]{Kreck:1984}. 

\begin{definition} \label{def:DiagIsom}
An isometry $\rho \in \Isom(\wh e_1, \wh e_2)$ is called \textit{diagonal} if there exists an isomorphism of abelian groups $f\colon \wh \L_1 \to \wh \L_2$ such that $\rho = (f^*)^{-1} \oplus f$. Here $f^*\colon (\wh \L_2)^* \to (\wh \L_1)^*$ is viewed as a map $(\L_2^*)\hh \to (\L_1^*)\hh$ via the canonical identifications $(\L_i^*)\hh \cong (\wh \L_i)^*$.
We write $\DiagIsom(\wh e_1, \wh e_2)$ for the set of diagonal isometries from $\wh e_1$ to $\wh e_2$, which we will view as a subset of $\Isom(\wh e_1,\wh e_2)$.
When $e_1=e_2=:e$, this defines a subgroup $\DiagIsom(\wh e) \le \Isom(\wh e)$.

An isometry $\rho \in \Isom(e_1^G,e_2^G)$ is called \textit{diagonal} if there exists an isomorphism of abelian groups $f\colon \L_1^G \to \L_2^G$ such that $\rho = (f^* \mid_{(\L_2^*)^G})^{-1} \oplus f$. Here we note that $f^*$ restricts to an isomorphism $f^* \mid_{(\L_2^*)^G}\colon (\L_2^*)^G \to (\L_1^*)^G$ by \cref{lemma:G-map-restriction}. We write $\DiagIsom(e_1^G,e_2^G)$ for the set of diagonal isometries from $e_1^G$ to $e_2^G$, which we will view as a subset of $\Isom(e_1^G,e_2^G)$.
When $e_1=e_2=:e$, this defines a subgroup $\DiagIsom(e^G) \le \Isom(e^G)$.

In the notation of \cref{def:psi-isom}, we write $\DiagIsom_{\Psi_1,\Psi_2}(e_1^G,e_2^G)$ to denote 
$$\DiagIsom(e_1^G,e_2^G) \cap \Isom_{\Psi_1,\Psi_2}(e_1^G,e_2^G).$$
The map $(\Psi_1,\Psi_2)_*$ defined in \cref{def:psi-isom} restricts to a function
\[ (\Psi_1,\Psi_2)_*\colon \DiagIsom_{\Psi_1,\Psi_2}(e_1^G,e_2^G)  \to \DiagIsom(\wh e_1, \wh e_2). \]

\begin{remark}\label{rem:check}
As a check: if  $\rho =  (f^* \mid_{(\L_2^*)^G})^{-1} \oplus f \in \DiagIsom(e_1^G,e_2^G)$, then $\rho \in 
\Isom_{\Psi_1,\Psi_2}(e_1^G,e_2^G)$ if and only if $f \in \Iso_{\psi_1,\psi_2}(\L_1^G,\L_2^G)$, by
\cref{lemma:G-map-restriction}. In addition, the commutative diagram in \cref{lemma:G-map-restriction} shows that the Tate reductions of
$f^* \mid_{(\L_2^*)^G}$ and $f^*$ agree,  via the canonical identifications $(\L_i^*)\hh \cong (\wh \L_i)^*$.
Hence under the restriction map $(\Psi_1,\Psi_2)_*(\rho) \in \DiagIsom(\wh e_1, \wh e_2)$. 
\end{remark}   

In the case where $e_1=e_2=:e$, we write this as $\DiagIsom_{\Psi}(e^G)$. The map  $\Psi_*$ restricts to a group homomorphism
\[ \Psi_*\colon \DiagIsom_{\Psi}(e^G) \to \DiagIsom(\wh e). \]
\end{definition}

 Note that we actually gave a definition of $\DiagIsom(\wh e_1, \wh e_2)$ which is a priori stronger than the one given in \cite[p30-31]{Kreck:1984}.
 We will  now show that these definitions are equivalent subject to the hypothesis that the $\Z G$-lattices $L_1$ and $L_2$ are \emph{$|G|$-locally isomorphic}, meaning that $\bZ_{(p)} \otimes_{\bZ} L_1 \cong \bZ_{(p)} \otimes_{\bZ} L_2$ as $\bZ_{(p)} G$-modules for all primes $p \mid |G|$, where $\bZ_{(p)}$ is the localisation of $\Z$ at the prime ideal $(p)$.

 \begin{proposition} \label{prop:diag-def-equiv} 
 Let $\rho \in \Isom(\wh e_1,\wh e_2)$. If $L_1$ and $L_2$ are $|G|$-locally isomorphic, then $\rho \in \DiagIsom(\wh e_1,\wh e_2)$ if and only if 
there exist maps $f \in \Hom_{\psi_1,\psi_2}(\L_1^G,\L_2^G)$ and $g \in \Hom_{\psi_2,\psi_1}(\L_2^G,\L_1^G)$ such that $\rho = \ol{(g^*\mid_{(\L_1^*)^G})} \oplus \ol{f}$, under the Tate reduction maps $\ol{\cdot} = (\psi_1', \psi_2')_*(\cdot)$ and $\ol{\cdot} = (\psi_1,\psi_2)_*(\cdot)$ in the two cases respectively.
\end{proposition}

\begin{proof}
Firstly note that, since $L_1$ and $L_2$ are $|G|$-locally isomorphic, it follows directly that $\rank_{\bZ} L_1^G = \rank_{\bZ} L_2^G$ and $\wh L_1 \cong \wh L_2$ (see, for example, the proof of \cite[Proposition III.5]{Kreck:1984}).

If $\rho \in \DiagIsom(\wh e_1,\wh e_2)$, then  $\rho = (h^*)^{-1} \oplus h$ for some $h \in \Iso(\wh \L_1,\wh \L_2)$. 
By \cref{lemma:G*-subtle}, for $i = 1,2$ we have two presentations
\[ 0 \to (L_i^*)^G \xrightarrow[]{\rest} (L_i^G)^* \to  Z/\beta_1 \times \cdots \times \Z/\beta_d \to 0 \]
of the form $0 \to \bZ^d \to \bZ^d  \to A \to 0$, where $\rank_{\bZ} L_i^G = d$ and $A = \wh L_1 \cong \wh L_2$ is a finite abelian group.
 By the structure theorem for  finitely generated abelian groups, any isomorphism $h \colon \wh L_1 \to \wh L_2$ can be lifted to a map $f \in \Hom_{\psi_1,\psi_2}(\L_1^G,\L_2^G)$ (see \cref{remark:wlog-psi-diagonal}). 
Similarly, we have a lift 
$s\colon (\L_1^*)^G \to (\L_2^*)^G$ such that
  $\ol{s} = (f^*)^{-1}$. 
By \cref{lemma:G-map-restriction}, there exists a homomorphism $g\colon \L_2^G \to \L_1^G$ for which $g^*\mid_{(\L_1^*)^G} = s$ (compare the last paragraph of the proof of \cite[Proposition III.5]{Kreck:1984}).
Hence $\rho = \ol{(g^*\mid_{(\L_1^*)^G})} \oplus \ol{f}$.

Conversely, suppose $f \in \Hom_{\psi_1,\psi_2}(\L_1^G,\L_2^G)$ and $g \in \Hom_{\psi_2,\psi_1}(\L_2^G,\L_1^G)$ are such that $\rho = \ol{(g^*\mid_{(\L_1^*)^G})} \oplus \ol{f} \in \Isom(\wh e_1,\wh e_2)$. Then $h = \ol{f} \in \Iso(\wh L_1,\wh L_2)$ and $t = \ol{(g^*\mid_{(\L_1^*)^G})} \in \Iso((\wh L_1)^*,(\wh L_2)^*)$.
It follows from \cite[Proposition III.6]{Kreck:1984} that there exists a unique $h' \in \Iso((\wh L_1)^*,(\wh L_2)^*)$ for which $h' \oplus h \in \Isom(\wh e_1,\wh e_2)$. Since $(h^*)^{-1} \in \Iso((\wh L_1)^*,(\wh L_2)^*)$ and we have $(h^*)^{-1} \oplus h \in \Isom(\wh e_1,\wh e_2)$, it follows that $t = (h^*)^{-1}$, and so $\rho \in \DiagIsom(\wh e_1,\wh e_2)$.
\end{proof}

The following is immediate from \cref{def:DiagIsom}.

\begin{proposition}\label{prop:sevenfifteen}
There are bijections
\begin{align*}
A\colon & \DiagIsom(\wh e_1,\wh e_2) \to \Iso(\wh \L_1,\wh \L_2), 	\quad (f^*)^{-1} \oplus f \mapsto f \\
B\colon & \DiagIsom(e_1^G,e_2^G) \to \Iso(\L_1^G,\L_2^G), 	\quad (f^* \mid_{(\L_2^*)^G})^{-1} \oplus f \mapsto f
\end{align*}
such that
\begin{clist}{(i)}
\item
$B$ restricts to a bijection	 $B'\colon \DiagIsom_{\Psi_1,\Psi_2}(e_1^G,e_2^G) \to \Iso_{\psi_1,\psi_2}(\L_1^G,\L_2^G)$.
\item 
There is a commutative diagram
\[
\begin{tikzcd}[column sep=15mm]
	\DiagIsom_{\Psi_1,\Psi_2}(e_1^G,e_2^G) \ar[d,"B'"'] \ar[r,"\text{$(\Psi_1,\Psi_2)_*$}"] & \DiagIsom(\wh e_1, \wh e_2) \ar[d,"A"'] \\
	\Iso_{\psi_1,\psi_2}(\L_1^G,\L_2^G) \ar[r,"\text{$(\psi_1,\psi_2)_*$}"] & \Iso(\wh \L_1,\wh \L_2).
\end{tikzcd}
\]
\item
In the case where $e_1=e_2$, $A$, $B$ and $B'$ are isomorphisms of abelian groups.
\end{clist}
\end{proposition}

In particular, this shows that $\IM(\DiagIsom_{\Psi}(e^G)) \unlhd \DiagIsom(\wh e)$ is a normal subgroup and there are isomorphisms of abelian groups
\eqncount
\begin{equation}\label{eq:diagquotient}
\frac{\DiagIsom(\wh e)}{\DiagIsom_{\Psi}(e^G)} \xrightarrow[\cong]{A} \frac{\Aut(\wh \L)}{\Aut_{\psi}(\L^G)} \xrightarrow[\cong]{\rho_*} (\Z/m)^\times/\{\pm 1\} 
\end{equation}
where $\rho_*$ is the map defined in \cref{lemma:webb} and $m=m_{(G,n)}$.
We will also need to consider non-diagonal isometries. The following special cases will suffice.

\begin{definition}
An isometry $\rho \in \Isom(\wh e)$ is called \textit{(upper) triangular} if there exists a homomorphism $f\colon \wh \L \to (\wh \L)^*$ such that 
$\rho = \left(\begin{smallmatrix} \id & f \\ 0 & \id \end{smallmatrix}\right)\colon  (\wh L)^* \oplus \wh L \to (\wh L)^* \oplus \wh L$.		
The set of triangular isometries defines a subgroup $\TriIsom(\wh e) \le \Isom(\wh e)$.	
\end{definition}

Note that $\TriIsom(\wh e)$ is abelian since $ \left(\begin{smallmatrix} \id & f \\ 0 & \id \end{smallmatrix}\right) \left(\begin{smallmatrix} \id & g \\ 0 & \id \end{smallmatrix}\right) =  \left(\begin{smallmatrix} \id & f+g \\ 0 & \id \end{smallmatrix}\right)$ for homomorphisms $f,g \colon \wh \L \to (\wh \L)^*$.

\section{The doubling construction} \label{s:doubling-construction}

We will begin by defining the doubling construction for finite $(G,n)$-complexes. Details can be found in \cite[Section II]{Kreck:1984}. 

Let $n \ge 2$ and let $G$ be a group of type $F_n$. If $X$ is a finite $(G,n)$-complex, then there exists an embedding $i\colon  X \hookrightarrow \R^{2n+1}$. Let $N(X) \subseteq \R^{2n+1}$ be a smooth regular neighbourhood of this embedding (unique up to concordance, by \cite[p76]{Wall:1966}) and define $M(X) := \bd N(X)$. This is a closed oriented smooth stably parallelisable $2n$-manifold (see, for example, \cite[p15]{Kreck:1984}).
If $X$ is well defined up to homotopy equivalence, then $M(X)$ is well-defined up to $h$-cobordism and, in particular, does not depend on the choice of embedding or smooth regular neighbourhood.  A polarisation $\pi_1(X) \cong G$ induces a polarisation $\pi_1(M(X)) \cong G$ and, from now on, we will assume that the manifolds $M(X)$ come equipped with a polarisation of this form.

We will refer to $M(X)$ as the \textit{double} of $X$ and the \textit{doubling construction} as the function
\[ \mathscr{D}\colon  \HT(G,n) \to \{\text{closed oriented smooth $2n$-manifolds}\} / \simeq \,\, , \quad X \mapsto M(X). \]
We refer to the manifolds arising via this construction as \textit{doubled $(G,n)$-complexes} and denote the set of all such manifolds up to homotopy equivalence by
$\dHT(G,n) := \IM(\mathscr{D})$. 
The doubles of minimal finite $(G,n)$-complexes could then be denoted by $\dHT_{\min}(G,n)$, but instead we will write $\sM_{2n}(G) := \mathscr{D}(\HT_{\min}(G,n))$ to simplify the notation, as in the Introduction. 

This map is referred to as the doubling construction since it has the following equivalent form. By \cite[Proposition II.2]{Kreck:1984}, there exists a $2n$-thickening $L(X)$ of $X$ (which need not embed in $\R^{2n}$) such that
\[ M(X) \cong_{\hCob} \bd(L(X) \times [0,1]) \cong L(X) \cup -L(X) \] 
where $\cong_{\hCob}$ denotes $h$-cobordism and $L(X) \cup -L(X)$ is the double of $L(X)$ along its boundary. 
Note that more general notions of double exist (see, for example, \cite{Hambleton:2009,NNP24}); our notion is sometimes referred to as a trivial double elsewhere in the literature.

The following is \cite[Proposition I.1]{Kreck:1984}. Since it is illuminating, we include a short proof below.

\begin{proposition} \label{prop:stable} 
If $X$ and $Y$ are  finite $(G,n)$-complexes such that $\chi(X) = \chi(Y)$, then $M(X)$ and $M(Y)$ are stably diffeomorphic.
\end{proposition}

\begin{proof} Since $\chi(X) = \chi(Y)$, there exists $r \geq 0$ such that $X' = X \vee rS^n \simeq Y' = Y \vee rS^n$ are simple homotopy equivalent (see \cite{Whitehead:1939}). Then by \cite[Corollary 2.1]{Wall:1966}, $N(Y')$ embeds in $N(X')$, and the region $W = N(X') - N(Y')$ is an $s$-cobordism between $M(X')$ and $M(Y')$,  which implies that $M(X')$ and $M(Y')$ are stably diffeomorphic (or diffeomorphic if $n > 2$).
 But $M(X') \cong M(X)\# r(S^n \times S^n)$ and $M(Y') \cong M(Y) \# r(S^n \times S^n)$ by construction, and so $M(X)$ and $M(Y)$ are stably diffeomorphic.
\end{proof}

We will now restrict to the case where $G$ is a finite group. 
In this case, we can identify the equivariant intersection form $S_{M(X)}$ as follows.

\begin{proposition}
\label{prop:intdouble}
Let $G$ be a finite group and $\ep = (-1)^n$. If $X$ is a finite $(G,n)$-complex, then 
$$\pi_n(M(X)) \cong \pi_n(X)^* \oplus \pi_n(X)$$ and $S_{M(X)} \cong \Met_\ep(\pi_n(X), \phi)$  are isometric for some $\phi\in \Sym_\ep(\pi_n(X))$ with $\phi^G = 0$.
\end{proposition}

This is a consequence of \cite[Proposition II.2]{Kreck:1984}. The property that $\phi^G=0$ is not stated explicitly and so we include further details on this below. Recall that $\Sym_\ep(L)$ denotes the $\ep$-symmetric forms on a $\ZG$-module $L$, which admit elements $g \in G$ as isometries (see \cref{def:fourtwo}).

\begin{proof} By \cite[Proposition II.4]{Kreck:1984}, the equivariant intersection form 
$$S_{M(X)}\cong \Met(\pi_n(X), \phi),$$
restricted to the  summand $0 \oplus \pi_n (X)$, defines a form $\phi \in \Sym_\ep(\pi_n(X))$. It remains to show that $\phi^G = 0$.
The  $2n$-thickening  $L(X)$ can be constructed explicitly as follows (see the proof of \cite[Proposition II.2 (i)]{Kreck:1984}).  Let $N(X^{(n-1)}) \subseteq \R ^{2n}$ be a $2n$-thickening of the $(n-1)$-skeleton of $X$. For each $n$-cell of $X$, the attaching map $f_i\colon  S^{n-1} \to X^{(n-1)}$ can be shown to induce an embedding $g_i\colon  S^{n-1} \times D^n \hookrightarrow \bd N(X^{(n-1)})$. Then form $L(X)$ by attaching $n$-handles to $N(X^{(n-1)})$ along the embeddings $g_i$ for each $n$-cell of $X$. It follows that $N(X) = L(X) \times I$ and $M(X) = L(X) \cup - L(X)$  (compare \cite[Lemma 3.3]{Hambleton:2009} for  uniqueness of the thickenings).

By an appropriate choice of these embeddings, one can ensure that the ordinary intersection form
$$S_X \colon H_n(L(X); \bZ) \otimes H_n(L(X); \bZ) \to \bZ$$
is zero.  This part of the proof is carefully explained in 
 \cite[Proposition II.4(ii), pp17-18]{Kreck:1984}. By the transfer map on rational homology, we have
$$H_n(M(X); \Z G)^G\otimes \bQ = H_n(\widetilde{M(X)};\bQ)^G = H_n(M(X); \bQ) = H_n(L(X); \bQ) \oplus H_n(-L(X); \bQ). $$
Since the intersection forms respect this splitting, the restriction $\phi$ of $S_{M(X)}$ has the property that $\phi^G \otimes \bQ = S_X \otimes \bQ = 0$. It follows that $\phi^G=0$, as required.
\end{proof}

The following will lead to our algebraic model for doubled $(G,n)$-complexes. A similar model was given by Kreck-Schafer in \cite[Section III.1]{Kreck:1984}.

\begin{definition} For $n \geq 2$ and $\ep = (-1)^n$, let $\mathfrak{C}^{\alg}_n(G)$ denote the category whose objects consist of pairs $(C, \phi)$, where $C \in \Alg(G,n)$  and $\phi \in \Sym_\ep(H_n(C))$, such that $\phi^G = 0$. A morphism $(C, \phi) \to (C', \phi')$ is a pair of 
chain maps $f\colon C \to C'$ and $g\colon C' \to C$ inducing the identity on $H_0$.

The objects of $\mathfrak{C}^{\alg}_n(G)$ admit an $\Aut(G)$-action where, for $\theta \in \Aut(G)$ and $(C,\phi) \in \mathfrak{C}^{\alg}_n(G)$, we define $(C,\phi)_\theta = (C_\theta, \phi_\theta)$ where $C_\theta$ and $\phi_\theta$ are as defined in Sections \ref{s:preliminaries} and \ref{ss:Herm-ZG} respectively.
\end{definition}

\begin{definition}  For $n \geq 2$ and $\ep = (-1)^n$, let $\mathfrak{M}^{\alg}_{2n}(G)$ denote the category  whose objects are pairs $(D,\Phi)$ where $D = (D_*,\bd_*)$ is a chain complex of (finitely generated) free $\ZG$-modules $D_*$ equipped with choices of $\Z G$-module isomorphisms $H_0(D) \cong \Z$ and $H_{2n}(D) \cong \Z$ such that
\begin{clist}{(i)}
\item
$D_i = 0$ for $i < 0$ or $i > 2n$.
\item
$H_i(D)=0$ for $0 < i < n$ and $n < i < 2n$.
\item 
$\Phi \in \Sym_\ep(H_n(D))$ is a non-singular $\ep$-symmetric form.
\end{clist}
 A morphism $(D, \Phi) \to (D', \Phi')$ is a chain map $h \colon D \to D'$ such that $H_0(h) = \id_{\Z}$ and $H_{2n}(h)=\id_{\Z}$, and $H_n(h)$ induces an isometry $H_n(h)\hh \colon \wh \Phi \to \wh \Phi'$ of the induced Tate forms $(\wH^0(G; H_n(D)), \hat\Phi)$ and  $(\wH^0(G; H_n(D')), \hat\Phi')$.

Two objects $(D,\Phi)$, $(D',\Phi') \in \mathfrak{M}^{\alg}_{2n}(G)$ are said to be \textit{homotopy equivalent} if there exists a chain homotopy equivalence $h\colon D \to D'$ such that $H_0(h) = \id_{\Z}$ and $H_{2n}(h)=\id_{\Z}$, and $H_n(h)$ defines an isometry $\Phi \to \Phi'$.

The objects in $\mathfrak{M}^{\alg}_{2n}(G)$  admit an $\Aut(G)$-action where, for $\theta \in \Aut(G)$ and $(D,\Phi) \in \mathfrak{M}^{\alg}_{2n}(G)$, we define
$(D,\Phi)_\theta = (D_\theta, \Phi_\theta)$. Here, if $D = (D_*,\bd_*)$, then $D_\theta = ((D_*)_\theta, \bd_*)$ similarly to the $\Aut(G)$-action on $\Alg(G,n)$. The form $\Phi_\theta$ is as defined in \cref{ss:Herm-ZG}.
\end{definition}
	
\begin{definition}\label{def:algdouble}
For an object $(C, \phi) \in \mathfrak{C}^{\alg}_n(G)$, we define a chain complex
$D_* = M(C_*)$, called the \textit{algebraic $2n$-double}, as follows
\begin{enumerate}
\item $D_i = C_i$ for $0 \leq i \leq n-1$, $D_i = C_{2n-i}^*$ for $n+1 \le i \le 2n$, and $D_n= C_n^* \oplus C_n$.
\item 
$\bd^D_i = \bd^C_i$ for $0 \leq i \leq n-1$, $\bd^D_i = \bd^*_{2n-i+1}$ for $n+2 \leq i \leq 2n$, \\ $\bd^D_{n} = \binom{0}{\pm \bd^C_n} \colon   C_n^* \oplus C_n \to  C_{n-1}$, and $\bd^D_{n+1} = (\pm \bd^*_{n}, 0) \colon C^*_{n-1} \to C_n^* \oplus C_n$.
\item
The identifications $H_0(D) \cong \Z$ and $H_{2n}(D) \cong \Z$ are induced by the identification $H_0(C) \cong \Z$.
\end{enumerate}

For $n \geq 2$ and $\ep = (-1)^n$, we define
a functor $M \colon \mathfrak{C}^{\alg}_n(G) \to \mathfrak{M}^{\alg}_{2n}(G)$ given on objects by $M(C, \phi) := (D, \Phi)$, where $D_* = M(C_*)$ and $\Phi = \Met_\ep( H_n(C), \phi)$.
If $f \colon C \to C'$ and $g\colon C' \to C$ are morphisms in $\mathfrak{C}^{\alg}_n(G)$, then $M(f,g) := h$ where $h_i = f_i$, $0 \leq i <n$; $h_i = g^*_i$, $n+1 \leq i \leq n$ and $h_n = g_n^* \oplus f_n$. 
See \cref{prop:chain-map}(i) for the proof that this defines a morphism in $\mathfrak{M}^{\alg}_{2n}(G)$.
\end{definition}

\begin{definition} \label{def:Dalg} 
Define $\dAlg(G,n)$ to be the set of homotopy types of algebraic $2n$-doubles $M(C, \phi)$ for $(C, \phi) \in \mathfrak{C}^{\alg}_n(G)$. Define the \textit{algebraic doubling construction} to be the map
\[ \dalg\colon\Alg(G,n) \to \dAlg(G,n), \quad C \mapsto M(C,0) = \big (M(C),H_\ep(H_n(C))\big ) \]
where $H_\ep(H_n(C))$ denotes the hyperbolic form on the module $H_n(C)^* \oplus H_n(C)$.

We also define $\Malg_{2n}(G)$ to be the set of homotopy types of algebraic $2n$-doubles $M(C, \phi)$ for $(C, \phi) \in \mathfrak{C}^{\alg}_n(G)$ such that $(-1)^n\chi(C) = \chi_{\min}(G,n)$; that is, $C \in \Alg_{\min}(G,n)$. Note that we could alternatively write this as $\dAlg_{\min}(G,n)$, but we use $\Malg_{2n}(G)$ to simplify notation.
\end{definition}

We will now introduce two new equivalence relations on the objects in $\mathfrak{M}^{\alg}_{2n}(G)$. Both equivalence relations refine homotopy equivalence, so induce a priori weaker equivalence relations on the set of algebraic $2n$-doubles $\dAlg(G,n)$.

\begin{definition} \label{def:Z-isometry}
Let $(D,\Phi)$, $(D',\Phi') \in \mathfrak{M}^{\alg}_{2n}(G)$ be two objects. A morphism $\Phi \to \Phi'$ is said to be an \textit{integral isometry} if it induces an isometry $\Phi^G \to (\Phi')^G$, and a \textit{Tate isometry} if it induces an isometry $\hat\Phi \to \hat\Phi'$. Note that integral isometries are Tate isometries.

If there exists a chain homotopy equivalence $h \colon  D \to D'$ such that $H_0(h) = \id_{\Z}$, $H_{2n}(h)=\id_{\Z}$, and $H_n(h)$ defines a morphism $\Phi \to \Phi'$ which is an integral isometry (resp. Tate isometry), then we write $(D,\Phi)\simeq_{\Z} (D,'\Phi')$ (resp. $(D,\Phi)\simeq_{\wh \Z} (D,'\Phi')$).
\end{definition} 

For $(C,\phi) \in \mathfrak{C}^{\alg}_n(G)$, it can be shown that $M((C,\phi)_\theta) = M(C,\phi)_\theta$. Note that $(M_\theta)^* \cong (M^*)_\theta$  if $M$ is a $\Z G$-lattice (see, for example, \cite[Section 6.1]{Ni20-I}). In particular, $\dalg$ induces a map on the orbits under the $\Aut(G)$-actions
\[ \dalg\colon\Alg(G,n)/\Aut(G) \to \dAlg(G,n)/\Aut(G). \]

Recall that, by \cref{prop:d2-problem}, there is an injective map $\mathscr{C} \colon \HT(G,n) \hookrightarrow \Alg(G,n)/\Aut(G)$ given by $X \mapsto C_*(\wt X)$. Similarly, by \cite[Proposition II.3]{Kreck:1984}, we have a map
\[ d(\mathscr{C}) \colon \dHT(G,n) \to \dAlg(G,n)/\Aut(G), \quad M \mapsto (C_*(\wt M), S_M).\]
The action of $\Aut(G)$ on $\dAlg(G,n)$ induces an action on
$\dAlg(G,n)/\simeq_{\Z}\,$,  so there is a bijection
$$(\dAlg(G,n)/\Aut(G))/\simeq_{\Z} \,\,\,\,\to \,\, (\dAlg(G,n)/\simeq_{\Z}\,)/\Aut(G).$$
The following observation will be crucial in our definition of the quadratic bias in \cref{s:bias-manifolds}.

\begin{proposition} \label{prop:Dalg-commutes}
Let $G$ be a finite group.
If $X$ is a finite $(G,n)$-complex, then 
 $$(C_*(\wt{M(X)}), S_{M(X)}) \simeq_{\Z} (C_*(\wt{M(X)}),H_\ep(\pi_n(X))).$$
 That is, there is a commutative diagram:
\[
\begin{tikzcd}
\HT(G,n) \ar[d,"\mathscr{D}"] \ar[r,"\mathscr{C}"] & \Alg(G,n)/\Aut(G) \ar[d,"\dalg"] \\
\dHT(G,n) \ar[r,"d(\mathscr{C})"] & (\dAlg(G,n)/\Aut(G))/\simeq_{\Z}.
\end{tikzcd}
\]
\end{proposition}

\begin{proof}
For $X \in \HT(G,n)$, we have $(d(\mathscr{C}) \circ \mathscr{D})(X) = (C_*(\wt{M(X)}), S_{M(X)})$. By \cref{prop:intdouble}, there is an isometry $S_{M(X)} \cong \Met_\ep(\pi_n(X), \phi)$ for some $\phi\in \Sym_\ep(\pi_n(X))$ with $\phi^G = 0$.
Next, we have $(\dalg \circ \mathscr{C})(X) = (M(C_*(\wt X)),H_\ep(\pi_n(X)))$. By \cite[Proposition II.3]{Kreck:1984}, we have that $C_*(\wt{M(X)}) \cong M(C_*(\wt X))$ are chain isomorphic. 
Let $\Phi = \Met_\ep(\pi_n(X), \phi)$ and $L = \pi_n(X)$, so that $H_\ep(\pi_n(X)) \cong e_L$. Since $\phi^G = 0$, \cref{prop:adapted} implies that $\Phi^G \cong e_L^G$. Hence  $(C_*(\wt{M(X)}),\Phi) \simeq_{\Z} (C_*(\wt{M(X)}),e_L)$ via the identity map on $C_*(\wt{M(X)})$, as required.
\end{proof}

\section{The quadratic bias invariant} \label{s:bias-manifolds}

Throughout this section, we will fix $n \ge 2$ and a finite group $G$. 
We will now introduce the quadratic bias invariant, which is a homotopy invariant for the class of doubles $M(X)$ for $X$ a finite $(G,n)$-complex.

In \cref{s:bias-complexes}, we defined the bias invariant for an arbitrary finite $(G,n)$-complex $X$ (see \cref{def:bias-original}). It followed from \cref{prop:PB(Gnl)} that the bias invariant vanishes if $X$ is non-minimal; in fact, the obstruction group $B(G,n,\chi)$ is trivial in this case. 
The construction of the quadratic bias invariant factors through the bias invariant and so, in the general setting, has an obstruction group $B_Q(G,n,\chi)$ which is a quotient of $B(G,n,\chi)$. This consequently vanishes for $\chi > \chi_{\min}(G,n)$. For notational simplicity, we will restrict to \emph{minimal} finite $(G,n)$-complexes from now on (see \cref{def:minimal-Gn}), and define the quadratic bias invariant only in this case.

We will begin with a polarised version of the invariant  in the more general setting of algebraic $2n$-doubles $M(C, \phi) \in \Malg_{2n}(G)$. In \cref{ss:qbias_manifolds}, we return to manifolds and prove \cref{thmx:main-general}.

\subsection{The quadratic bias for algebraic $2n$-doubles}

For a $\Z G$-module $L$, recall that there is a canonical identification $(L^*)\hh \cong (\wh L)^*$. 
For $(C, \phi) \in \mathfrak{C}^{\alg}_n(G)$, we let $L= H_n(C)$ and denote the evaluation form on $L^*\oplus L$ by $e: = e_{L}$.
Recall that, if $(D,\Phi) = M(C,\phi)$, then \cref{prop:adapted} implies that $\Phi^G \cong e^G$ and $\wh \Phi \cong \wh e$. For the remainder of this section, we will use the identifications $(L^*)\hh \cong (\wh L)^*$, $\Phi^G \cong e^G$ and $\wh \Phi \cong \wh e$ without further mention.

The following is a slight extension of \cite[Propositions III.3 \& III.4]{Kreck:1984}. We can view it as the analogue of \cref{prop:bias-tate}. The notation $\DiagIsom(\wh e_1,\wh e_2)$ for the set of  \emph{diagonal isometries} is given in \cref{def:DiagIsom}.

\begin{proposition} \label{prop:chain-map} \label{def:I(D1D2)}
For $i = 1,2$, let $(D_i, \Phi_i) =M(C_i, \phi_i)$, for $(C_i, \phi_i) \in \mathfrak{C}^{\alg}_n(G)$ such that $\chi(C_1)= \chi(C_2)$.
Let $L_i = H_n(C_i)$ and $e_i = e_{L_i}$, so that $\wh \Phi_i \cong \wh e_i$ for $i=1, 2$.
Then:
\begin{clist}{(i)}
\item
There exists a chain map $h\colon D_1 \to D_2$ such that $H_0(h) = \id_\Z$, $H_{2n}(h)=\id_\Z$ and 
\[ H_n(h)\hh\colon (\L_1^*)\hh \oplus \wh \L_1 \to (\L_2^*)\hh \oplus \wh \L_2\]
is a diagonal isometry from $\wh e_1$ to $\wh e_2$. Furthermore, we can take $h = M(f,g)$ where $f \colon C_1 \to C_2$ and $g \colon  C_2 \to C_1$ are any chain maps such that $H_0(f) = \id_{\Z}$ and $H_0(g) = \id_{\Z}$.
\item
Let $h\colon D_1 \to D_2$ be a chain map such that $H_0(h) = \id_\Z$, $H_{2n}(h)=\id_\Z$ and 
$H_n(h)\hh \in \DiagIsom(\wh e_1,\wh e_2)$. Then $H_n(h)\hh$ is independent of the choice of $h\colon D_1 \to D_2$.
We will write this as $I(D_1,D_2) = (\nu(D_1,D_2)^*)^{-1} \oplus \nu(D_1,D_2) \in \DiagIsom(\wh e_1,\wh e_2)$ where $\nu(D_1,D_2) \in \Iso(\wh \L_1,\wh \L_2)$.
\item
More generally, let $h\colon D_1 \to D_2$ be a chain map such that $H_0(h) = \id_\Z$, $H_{2n}(h)=\id_\Z$ and $H_n(h)\hh \in \Isom(\wh e_1,\wh e_2)$. Then 
\[ H_n(h)\hh = I(D_1,D_2) \circ \left(\begin{smallmatrix}\id & \alpha \\ 0 & \id\end{smallmatrix}\right)\]
for some $\alpha\colon  \wh L_1 \to (\wh L_1)^*$, where $\left(\begin{smallmatrix}\id & \alpha \\ 0 & \id\end{smallmatrix}\right) : (\L_1^*)\hh \oplus \wh \L_1 \to (\L_1^*)\hh \oplus \wh \L_1$, $(x,y) \mapsto (x+\alpha(y),y)$.
\end{clist}
\end{proposition}

\begin{proof}
(i) We combine parts of the arguments in \cite[Propositions III.3 \& III.4]{Kreck:1984} to verify these statements. Since $D_i = M(C_i,\phi_i)$, there exists a chain map $f\colon C_1\to C_2$ such that
$H_0(f) = \id_\Z$, which induces an isomorphism $f_*\colon H_n(C_1)\hh \to H_n(C_2)\hh$. Similarly, there exists a chain map $g\colon C_2\to C_1$ such that
$H_0(f) = \id_\Z$, which induces an isomorphism $g_*\colon H_n(C_1)\hh \to H_n(C_2)\hh$. The chain map $h = M(f,g)$ has the required properties. In particular, $H_n(h)\hh \in \DiagIsom(\wh e_1,\wh e_2)$, so that $h\colon (D_1, \Phi_1) \to (D_2, \Phi_2)$ gives a morphism in the category $\mathfrak{M}^{\alg}_{2n}(G)$.

(ii)/(iii) Let $h\colon D_1 \to D_2$ be a chain map such that $H_0(h) = \id_\Z$, $H_{2n}(h)=\id_\Z$ and $H_n(h)\hh \in \Isom(\wh e_1,\wh e_2)$. By (i), there exists a chain map $h_0\colon D_1 \to D_2$ such that $H_0(h_0)=\id_{\Z}$, $H_{2n}(H_0)=\id_{\Z}$ and $I := H_n(h_0)\hh \in \DiagIsom(\wh e_1, \wh e_2)$.
It follows from the arguments given in \cite[Proposition III.4]{Kreck:1984} that $H_n(h)\hh = I \circ \left(\begin{smallmatrix}\id & \alpha \\ 0 & \id\end{smallmatrix}\right)$ for some $\alpha\colon  \wh L_1 \to (\wh L_1)^*$.
The argument applies since the equivariant intersection form $\Phi_i =\Met_\ep(H_n(C_i), \phi_i)$ has $\phi_i^G=0$, implying that the Tate forms $\wh \Phi_i$ are hyperbolic, i.e. $\wh \Phi_i = \wh e_i$.

To prove (ii), suppose that $H_n(h)\hh \in \DiagIsom(\wh e_1,\wh e_2)$. If $I = (\nu^*)^{-1} \oplus \nu = \left(\begin{smallmatrix}(\nu^*)^{-1} & 0 \\ 0 & \nu \end{smallmatrix}\right)$ for some $\nu \in \Iso(\wh L_1,\wh L_2)$, then $H_n(h)\hh = \left(\begin{smallmatrix}(\nu^*)^{-1} & (\nu^*)^{-1} \circ \alpha \\ 0 & \nu \end{smallmatrix}\right)$. Since this is diagonal, we have $(\nu^*)^{-1} \circ \alpha = 0$ and so $\alpha = 0$. Hence $H_n(h)\hh = I$ and so $I = I(D_1,D_2)$ is independent of the choice of $h$.
(iii) now follows immediately by returning to the general case.
\end{proof}

 The following is a consequence of \cite[Proposition III.5]{Kreck:1984}. Note that the definition of $\DiagIsom(\wh e_1,\wh e_2)$ given in \cref{ss:isom(e)} is a priori different to the one given in in \cite{Kreck:1984}.
 However, since the lattices $L_1$ and $L_2$ are $|G|$-locally isomorphic (see, for example, the proof of \cite[Proposition III.5]{Kreck:1984}), it follows from \cref{prop:diag-def-equiv} that these definitions are equivalent.

\begin{lemma} \label{lemma:diag-inducing-diag} Under the assumptions in \cref{prop:chain-map}, 
there exists a diagonal isometry $\varphi\colon e_1^G \to e_2^G$ inducing a diagonal isometry $\wh \varphi\colon \wh e_1 \to \wh e_2$. That is, $\IM(\DiagIsom_{\Psi_1,\Psi_2}(e_1^G,e_2^G)) \cap \DiagIsom(\wh e_1, \wh e_2) \ne \emptyset$.
\end{lemma}

We will now restrict to the case of minimal complexes.
Fix $X \in \HT_{\min}(G,n)$, $L = \pi_n(X)$ and $\mathscr{D}^{\alg}(C_*(\wt X)) = (D,e)$ where $D = C_*(\wt M(X))$ and $e = e_L \cong H_\ep(L)$. We will refer to $(D,e) \in \Malg_{2n}(G)$ as the \textit{reference minimal algebraic $2n$-double}. See \cref{def:Dalg} for the definition of $\Malg_{2n}(G)$. 
We  often write this as $(D,L,e)$ when $L$ is not clear from the context.

For $i = 1,2$, let $(D_i,L_i,e_i)$ be as in \cref{prop:chain-map} and such that $\chi(D_1) = \chi(D_2) = \chi(D)$. By \cref{lemma:diag-inducing-diag}, there exists $\ol{\tau}_{D_i} \in \DiagIsom(e^G,e_i^G)$ inducing $\tau_{D_i} \in \DiagIsom(\wh e, \wh e_i)$.
Fix these reference isometries once and for all.

We are now ready to define the quadratic bias invariant, first in the setting of algebraic $2n$-doubles.
Recall from \cref{ss:isom(e)} that $\DiagIsom(\wh e)\le \Isom(\wh e)$ denotes the set of diagonal isometries and  $\TriIsom(\wh e) \le \Isom(\wh e)$ the set of triangular isometries. For subgroups $A, B \le C$, we define $A \cdot B = \{ab : a\in A, b \in B\} \subseteq C$.
 
\begin{definition}[Quadratic bias invariant for algebraic $2n$-doubles] \label{def:bias-quad}
Fix a reference minimal doubled complex $(D,e)$. 
Define the \textit{polarised quadratic bias obstruction group} to be
\[ PB_Q(G,n):= \frac{\DiagIsom(\wh e)}{[\IM(\Isom_\Psi(e^G)) \cdot \TriIsom(\wh e)] \cap \DiagIsom(\wh e)}. \]	
When $n = 2$, we write $P_Q(G):=P_Q(G,2)$.

For $i=1,2$, let $(D_i,e_i) \in \Malg_{2n}(G)$ such that $e_i = e_{L_i}$ for some module $L_i$.
Define the \textit{quadratic bias invariant} to be:
\[ \beta_{Q}((D_1,e_1),(D_2,e_2)) := [\tau_{D_2}^{-1} \circ I(D_1,D_2) \circ \tau_{D_1}] \in PB_Q(G,n) \]
where the $\tau_{D_i}$ are as defined above and $[\,\cdot\,]\colon  \DiagIsom(\wh e) \twoheadrightarrow PB_Q(G,n)$ is the quotient map.
\end{definition}

We will now establish the following two propositions.

\begin{proposition} \label{prop:PB-well-defined}
$[\IM(\Isom_\Psi(e^G)) \cdot \TriIsom(\wh e)] \cap \DiagIsom(\wh e)$ is a normal subgroup of $\DiagIsom(\wh e)$. In particular, $PB_Q(G,n)$ is a well-defined abelian group.
\end{proposition}

\begin{proposition} \label{prop:PB-an-invariant}
$\beta_{Q}((D_1,e_1),(D_2,e_2)) \in PB_Q(G,n)$ does not depend on the choice of representatives $(D_i,e_i) \in \Malg_{2n}(G)$ and isometries $\tau_{D_i} \in \DiagIsom(\wh e, \wh e_i)$. In particular, if $(\ol{D},\ol{e})$ is a reference minimal algebraic $2n$-double, then the quadratic bias invariant defines a map
\[ \beta_Q\colon\{(D,e) \in \Malg_{2n}(G) : \chi(D)=\chi(\ol{D}) \} \to PB_Q(G,n), \quad (D,e) \mapsto \beta_Q((D,e),(\ol{D},\ol{e})). \]

Furthermore, $\beta_Q$ is an invariant of algebraic $2n$-doubles up to the equivalence relation $\simeq_{\Z}\,$. That is, if $(D_1,e_1) \simeq_{\Z} (D_2,e_2)$, then $\beta_Q((D_1,e_1)) = \beta_Q((D_2,e_2)) \in PB_Q(G,n)$.
\end{proposition}

We will begin with the proof of \cref{prop:PB-well-defined}.
Recall that two subgroups $A, B \le C$ \textit{commute} if $A \cdot B = B \cdot A$. If so, then it follows that $A \cdot B$ is a subgroup of $C$.

\begin{lemma} \label{lemma:denominator-group}
In the notation above, we have:
\[ [\IM(\Isom_\Psi(e^G)) \cdot \TriIsom(\wh e)] \cap \DiagIsom(\wh e) = [((\DiagIsom(\wh e) \cdot \TriIsom(\wh e)) \cap \IM(\Isom_\Psi(e^G))) \cdot \TriIsom(\wh e)] \cap \DiagIsom(\wh e).\]
\end{lemma}

\begin{proof}
The inclusion $\supseteq$ is clear and so it suffices to prove $\subseteq$. Let $\varphi \in [\IM(\Isom_\Psi(e^G)) \cdot \TriIsom(\wh e)] \cap \DiagIsom(\wh e)$. Then $\varphi \in \DiagIsom(\wh e)$ and $\varphi = \rho_1 \circ \rho_2$ for some $\rho_1 \in \IM(\Isom_\Psi(e^G))$ and $\rho_2 \in \TriIsom(\wh e)$. We have $\rho_1 = \varphi \circ \rho_2^{-1} \in \DiagIsom(\wh e) \cdot \TriIsom(\wh e)$ which implies that $\rho_1 \in (\DiagIsom(\wh e) \cdot \TriIsom(\wh e)) \cap \IM(\Isom_\Psi(e^G))$, and so $\varphi \in ((\DiagIsom(\wh e) \cdot \TriIsom(\wh e)) \cap \IM(\Isom_\Psi(e^G))) \cdot \TriIsom(\wh e)$, which completes the proof. 
\end{proof}

In order to prove this equivalent form is a subgroup of $\DiagIsom(\wh e)$, we will use:

\begin{lemma} \label{lemma:Diag-Tri-commute}
If $\varphi \in \DiagIsom(\wh e)$, then $\varphi \cdot \TriIsom(\wh e) = \TriIsom(\wh e) \cdot \varphi$. It follows that:
\begin{clist}{(i)}
\item
$\DiagIsom(\wh e)$ normalises $\TriIsom(\wh e)$ in $\Isom(\wh e)$, and so $\DiagIsom(\wh e) \cdot \TriIsom(\wh e) \le \Isom(\wh e)$ is a subgroup. 
\item 
If $H \le \DiagIsom(\wh e) \cdot \TriIsom(\wh e)$ is a subgroup, then $H \cdot \TriIsom(\wh e) \le \Isom(\wh e)$ is a subgroup.
\end{clist}
\end{lemma}

\begin{proof}
Let $\varphi \in \DiagIsom(\wh e)$ and $\varphi_T \in \TriIsom(\wh e)$. Then there exists an isomorphism $f \colon \wh L \to \wh L$ and a homomorphism  $h \colon  \wh L \to (\wh L)^*$ such that $\varphi = \left(\begin{smallmatrix} (f^*)^{-1} & 0 \\ 0 & f \end{smallmatrix}\right)$ and  $\varphi_T = \left(\begin{smallmatrix} \id & h \\ 0 & \id \end{smallmatrix}\right)$. 
Then we have 
\[ \varphi \circ \varphi_T = \left(\begin{smallmatrix} (f^*)^{-1} & 0 \\ 0 & f \end{smallmatrix}\right)\left(\begin{smallmatrix} \id & h \\ 0 & \id \end{smallmatrix}\right) 
= \left(\begin{smallmatrix} \id & (f^*)^{-1} \circ h \circ f^{-1} \\ 0 & \id \end{smallmatrix}\right)
\left(\begin{smallmatrix} (f^*)^{-1} & 0 \\ 0 & f \end{smallmatrix}\right) = \varphi_T' \circ \varphi \]
where $\varphi_T' \in \TriIsom(\wh e)$. 
Hence $\varphi \cdot \TriIsom(\wh e) \cdot \varphi^{-1} \subseteq \TriIsom(\wh e)$, and equality follows by applying this to $\varphi^{-1}$.

Part (i) now follows immediately. To see part (ii), consider the normaliser subgroup
\[ N_{\Isom(\wh e)}(\TriIsom(\wh e)) = \{ \varphi \in \Isom(\wh e) : \varphi \cdot \TriIsom(\wh e) = \TriIsom(\wh e) \cdot \varphi\}.\]
Since $\DiagIsom(\wh e), \TriIsom(\wh e) \le N_{\Isom(\wh e)}(\TriIsom(\wh e))$, we have $\DiagIsom(\wh e) \cdot \TriIsom(\wh e) \le N_{\Isom(\wh e)}(\TriIsom(\wh e))$.
\end{proof}

\begin{proof}[Proof of \cref{prop:PB-well-defined}]
We will start by showing that $[\IM(\Isom_\Psi(e^G)) \cdot \TriIsom(\wh e)] \cap \DiagIsom(\wh e)$ is a subgroup of $\DiagIsom(\wh e)$.
By \cref{lemma:Diag-Tri-commute} (i), $\DiagIsom(\wh e) \cdot \TriIsom(\wh e)$ is a subgroup. Since $\IM(\Isom_\Psi(e^G)))$ is a subgroup, this gives that $H := (\DiagIsom(\wh e) \cdot \TriIsom(\wh e)) \cap \IM(\Isom_\Psi(e^G))$ is a subgroup. Since $H \le \DiagIsom(\wh e) \cdot \TriIsom(\wh e)$ is a subgroup, \cref{lemma:Diag-Tri-commute} (ii) now implies that $H \cdot \TriIsom(\wh e)$ is a subgroup and so $(H \cdot \TriIsom(\wh e))\cap \DiagIsom(\wh e)$ is a subgroup.
The result now follows by combining with \cref{lemma:denominator-group}.

To see that $K:=[\IM(\Isom_\Psi(e^G)) \cdot \TriIsom(\wh e)] \cap \DiagIsom(\wh e)$ is a normal subgroup of $\DiagIsom(\wh e)$, note that $\IM(\DiagIsom_\Psi(e^G)) \le K$. By combining \cref{lemma:webb} and \cref{prop:sevenfifteen} (see the discussion following \cref{prop:sevenfifteen}), we have that $\IM(\DiagIsom_\Psi(e^G)) \unlhd \DiagIsom(\wh e)$ is a normal subgroup and $\DiagIsom(\wh e)/\IM(\DiagIsom_\Psi(e^G)) \cong (\Z/m)^\times/\{\pm 1\}$ is abelian, for some $m \ge 1$. This implies there is a quotient map $f \colon \DiagIsom(\wh e) \twoheadrightarrow (\Z/m)^\times/\{\pm 1\}$ with $\ker(f) = \IM(\DiagIsom_\Psi(e^G))$. 

Let $K' = f(K)$. Since $\ker(f) = \IM(\DiagIsom_\Psi(e^G)) \le K$, we have that $f^{-1}(K') = K \cdot \ker(f) =K$. Since $(\Z/m)^\times/\{\pm 1\}$ is abelian, $K' \unlhd (\Z/m)^\times/\{\pm 1\}$ is a normal subgroup. The preimage of a normal subgroup is normal, and so $K \unlhd \DiagIsom(\wh e)$ is normal.
Thus, $PB_Q(G,n)$ is a well-defined group. It is a quotient of $(\Z/m)^\times/\{\pm 1\}$ and so is abelian.
\end{proof}

\begin{proof}[Proof of \cref{prop:PB-an-invariant}]
We begin by noting that, given $(D_1,e_1)$ and $(D_2,e_2)$, the class $[\tau_{D_2}^{-1} \circ I(D_1,D_2) \circ \tau_{D_1}] \in PB_Q(G,n)$ is independent of the choice of $\tau_{D_1}$ and $\tau_{D_2}$. This follows from the fact that, if $\tau_{D_1}'$ and $\tau_{D_2}'$ are other choices, then the two classes would differ by multiplication by $(\tau_{D_2}')^{-1} \circ \tau_{D_2}$ and $\tau_{D_1}^{-1} \circ \tau_{D_1}'$, which are both in $\Isom_{\Psi}(e^G) \cap \DiagIsom(\wh e)$.

Next suppose that, for $i =1,2$, there is a chain homotopy equivalence $h_i\colon D_i \to D_i'$ which induces an integral isometry. Thus $H_n(h_i)\hh \in \Isom(\wh e_i,  (e_i')\hh)$ is an isometry which lifts to an isometry $H_n(h_i)^G \in \Isom(e_i^G,  (e_i')^G)$. 
By \cref{prop:chain-map} (iii), we have that
\[ H_n(h_i)\hh = I(D_i,D_i') \circ \left(\begin{smallmatrix}\id & \alpha_i \\ 0 & \id\end{smallmatrix}\right)\]
for some $\alpha_i \colon  \wh L_i \to (\wh L_i')^*$ and where $I(D_i,D_i')$ is as defined in \cref{def:I(D1D2)}. This implies that
\[ A_i := \tau_{D_i'}^{-1} \circ I(D_i,D_i') \circ \tau_{D_i} = (\tau_{D_i'}^{-1} \circ H_n(h_i)\hh \circ \tau_{D_i}) \circ (\circ \tau_{D_i}^{-1} \circ \left(\begin{smallmatrix}\id & -\alpha_i \\ 0 & \id\end{smallmatrix}\right) \circ \circ \tau_{D_i}).  \]
We have $\tau_{D_i'}^{-1} \circ H_n(h_i)\hh \circ \tau_{D_i} = \Psi_*(\ol{\tau}_{D_i'}^{-1} \circ H_n(h_i)^G \circ \ol{\tau}_{D_i}) \in \IM(\Isom_\Psi(e^G))$. Since $\tau_{D_i} \in \DiagIsom(\wh e_i, \wh e)$, we have that $\tau_{D_i}^{-1} \circ \left(\begin{smallmatrix}\id & -\alpha_i \\ 0 & \id\end{smallmatrix}\right) \circ \tau_{D_i} \in \TriIsom(\wh e)$ by the same argument as \cref{lemma:Diag-Tri-commute}. In particular, we have that $A_i \in [\IM(\Isom_\Psi(e^G))) \cdot \TriIsom(\wh e)] \cap \DiagIsom(\wh e)$ for $i=1,2$.

It follows from \cref{prop:chain-map} that $I(D_1,D_2) = I(D_2,D_2')^{-1} \circ I(D_1',D_2') \circ I(D_1,D_1')$.
Hence we have
\[\tau_{D_2}^{-1} \circ I(D_1,D_2) \circ \tau_{D_1} = A_2^{-1} \circ (\tau_{D_2'}^{-1} \circ I(D_1',D_2') \circ \tau_{D_1'}) \circ A_1 \]
and so $[\tau_{D_2}^{-1} \circ I(D_1,D_2) \circ \tau_{D_1}] = [\tau_{D_2'}^{-1} \circ I(D_1',D_2') \circ \tau_{D_1'}] \in PB_Q(G,n)$, as required.
\end{proof}

Define $\wh{\mathscr{D}}$ to be the surjective abelian group homomorphism given by the composition
\[ \wh{\mathscr{D}} \colon \frac{\Aut(\wh \L)}{\Aut_{\psi}(\L^G)} \xrightarrow[\cong]{A^{-1}} \frac{\DiagIsom(\wh e)}{\DiagIsom_{\Psi}(e^G)}   \twoheadrightarrow PB_Q(G,n) \]	
where $A$ is as defined in the discussion following \cref{prop:sevenfifteen}.
In particular, for $f \in \Aut(\wh L)$, we have that $\wh{\mathscr{D}}([f]) = [(f^*)^{-1} \oplus f]$ where $(f^*)^{-1} \oplus f = \left(\begin{smallmatrix} (f^*)^{-1} & 0 \\ 0 & f \end{smallmatrix}\right) \in \DiagIsom(\wh e)$.

We conclude this section by establishing the following relationship between the bias invariant for an algebraic $n$-complex $C$ over $\Z G$ with $(-1)^n\chi(C) = \chi_{\min}(G,n)$ and the quadratic bias invariant for the corresponding algebraic $2n$-double $(D,e)$.
The following was asserted on \cite[p34]{Kreck:1984} though no argument was given.

\begin{proposition}
\label{prop:bias-vs-qbias}	
Fix $C \in \Alg_{\min}(G,n)$, let $L = \pi_n(C)$ and let $(D,e) = \dalg(C) \in \Malg_{2n}(G)$ denote the corresponding algebraic $2n$-double. Then there is a commutative diagram of sets
\[
\begin{tikzcd}
\Alg_{\min}(G,n) \ar[r,"\mathscr{D}^{\alg}"] \ar[d,"\text{$\beta(\cdot,C)$}"'] & \Malg_{2n}(G)
 \ar[d,"\text{$\beta_Q(\cdot, (D,e))$}"] \\
\displaystyle \frac{\Aut(\wh \L)}{\Aut_{\psi}(\L^G)} \ar[r,twoheadrightarrow,"\text{$\wh{\mathscr{D}}$}"] & \displaystyle PB_Q(G,n).
\end{tikzcd}
\]	
\end{proposition}

\begin{proof}
Let $C' \in \Alg_{\min}(G,n)$. We start by evaluating the bottom left composition.
By the discussion in \cref{ss:bias-algebraic}, there exists an isomorphism $\ol{\tau}_C\colon H_n(C)^G \to H_n(C')^G$ inducing an isomorphism $\tau_C\colon H_n(C)\hh \to H_n(C')\hh$, and similarly for $\tau_{C'}$, and there exists a chain map $f\colon C \to C'$ such that $H_0(f) = \id_{\Z}$. By \cref{prop:bias-tate}, $H_n(f)\hh$ is an isomorphism and coincides with $\sigma(C,C')$. By definition, we have $\beta(C',C) = [F]$ where $F =  \tau_{C}^{-1} \circ H_n(f)\hh \circ \tau_{C'}$.
This implies that $\wh{\mathscr{D}}(\beta(C',C)) = [(F^*)^{-1} \oplus F]$.

To evaluate the top right composition, first let $(D',e') = \dalg(C')$ where $e' = e_{L'}$ for $L' = H_n(C')$.
Let $\tau_D = (\tau_C^*)^{-1} \oplus \tau_C$, $\tau_{D'} = (\tau_{C'}^*)^{-1} \oplus \tau_{C'}$, and let $h = M(f,g)$ where $g\colon C' \to C$ is a chain map such that $H_0(g) = \id_{\Z}$. By \cref{prop:chain-map}(i), $H_n(h)\hh$ is a diagonal isometry which coincides with $I(D,D')$. By definition, we have $\beta_Q((D',e'),(D,e)) = [\tau_D^{-1} \circ H_n(h)^\hh \circ \tau_{D'}]$. We have $H_n(h)\hh = (H_n(g)^*)\hh \oplus H_n(f)\hh$.
It is shown in \cite[Lemma 1]{Schafer:1996} that, if $f' \colon C_1 \to C_2$ is a chain map with $H_0(f')=\id_{\Z}$, then $H_n(f')\hh$ depends only on $C_1$ and $C_2$ and not on the map $f'$. It follows that, since $g \circ f \colon C \to C$ has $H_0(g \circ f) = \id_{\Z}$, we have that $H_n(g \circ f)\hh = H_n(\id_C)\hh = \id$ and so $H_n(g)\hh = (H_n(f)\hh)^{-1}$, and similarly we can obtain $(H_n(g)^*)\hh = ((H_n(f)^*)\hh)^{-1}$.
Hence $H_n(h)\hh = ((H_n(f)^*)\hh)^{-1} \oplus H_n(f)\hh$ and so $\beta_Q((D',e'),(D,e)) =[(F^*)^{-1} \oplus F]$, as required.
\end{proof}

\subsection{The quadratic bias invariant for manifolds}\label{ss:qbias_manifolds}

Fix an integer $n \ge 2$ and a finite group $G$.
The proof of \cref{thmx:main-general} follows from \cref{thm:main-high} and  \cref{prop:beta-to-betaQ} in this section.
Recall from \cref{s:doubling-construction} that $\sM_{2n}(G)$ denotes the set of homotopy types of doubles $M(X)$ for $X$ a minimal finite $(G,n)$-complex.  
    
We will now establish the analogue of \cref{prop:bias-Aut(G)-action} for the quadratic bias invariant. Recall the definitions of $\rho$ and $\varphi_{(G,n)}$ from \cref{lemma:webb} and \cref{prop:bias-Aut(G)-action} respectively. 

\begin{proposition} \label{prop:quadratic-bias-Aut(G)-action}
Let	$n \ge 2$ and let $G$ be a finite group. If $X$ is a minimal finite $(G,n)$-complex and $(D,e) = \dalg(C_*(\wt X))$, then the map
\[ \Psi_{G,n}\colon  \Aut(G) \to PB_Q(G,n), \quad \theta \mapsto \beta_Q((D,e),(D_\theta,e_\theta))  \]
is a group homomorphism and is independent of the choice of $X$. Furthermore, we have that $\Psi_{G,n} = \wh{\mathscr{D}} \circ (\rho^{-1})_* \circ  \varphi_{(G,n)}$.
\end{proposition}

\begin{proof}
The equality $\Psi_{G,n} = \wh{\mathscr{D}} \circ (\rho^{-1})_* \circ  \varphi_{(G,n)}$ follows directly from \cref{prop:bias-vs-qbias} and the fact that, if $\theta \in \Aut(G)$ and $C \in \Alg(G,n)$ has $(D,e) = \mathscr{D}^{\alg}(C)$, then $\mathscr{D}^{\alg}(C_\theta) = (D_\theta,e_\theta)$. Since, by \cref{prop:bias-Aut(G)-action}, $\varphi_{(G,n)}$ is a group homomorphism and is independent of the choice of $X$, the same therefore holds for $\Psi_{G,n}$.
\end{proof}

\begin{definition}[Quadratic bias invariant for doubled $(G,n)$-complexes]
\label{def:quadbias-high}
Let $n \ge 2$ and let $G$ be a finite group. Define $D_Q(G,n)$ to be the image of the map $\Psi_{G,n}$ given in \cref{prop:quadratic-bias-Aut(G)-action}.
Fix a reference minimal $(G,n)$-complex $\ol{X}$ and let $(D,e) = \mathscr{D}^{\alg}(C_*(\ol{X}))$. Define the \textit{quadratic bias obstruction group} for doubled minimal $(G,n)$-complexes to be 
\[ B_Q(G,n)\colon = \frac{PB_Q(G,n)}{D_Q(G,n)} = \frac{\DiagIsom(\wh e)}{([\IM(\Isom_\Psi(e^G))) \cdot \TriIsom(\wh e)] \cap \DiagIsom(\wh e)) \cdot D_Q(G,n)}. \]
When $n = 2$, we write $B_Q(G):=B_Q(G,2)$.

Let $X_1$, $X_2$ be minimal finite $(G,n)$-complexes, let $M(X_1)$, $M(X_2)$ be the doubles, and let $(D_1,e_1) = \dalg(C_*(\wt X_1))$, $(D_2,e_2) = \dalg(C_*(\wt X_2))$.
Define the \textit{quadratic bias invariant} to be
\[ \beta_Q(M(X_1),M(X_2))\colon = [\beta_Q((D_1,e_1),(D_2,e_2))] \in B_Q(G,n) \]
where $[\,\cdot\,]\colon  PB_Q(G,n) \twoheadrightarrow B_Q(G,n)$ is the quotient map and $\beta_Q((D_1,e_1),(D_2,e_2))$ denotes the quadratic bias invariant defined in the case of algebraic $2n$-doubles in \cref{def:bias-quad}.
\end{definition}

The quotient $PB_Q(G,n)/D_Q(G,n)$ is well defined since $PB_Q(G,n)$ is an abelian group.

\begin{remark} \label{remark:Aut(G)-is-easy}
(i)
It follows from the definition, as well as \cref{prop:quadratic-bias-Aut(G)-action}, that $D_Q(G,n)$ is the image of $D(G,n)$ under $\wh{\mathscr{D}} \circ (\rho^{-1})_*$. In particular, there is a commutative diagram
\[
\begin{tikzcd}
(\Z/m)^\times/\{\pm 1\} \ar[r,"(\rho^{-1})_*","\cong"'] \ar[d,twoheadrightarrow] & \displaystyle \frac{\Aut(\wh \L)}{\Aut_{\psi}(\L^G)} \ar[r,twoheadrightarrow,"\wh{\mathscr{D}}"] & PB_Q(G,n) \ar[d,twoheadrightarrow] \\
	\displaystyle B(G,n) \ar[rr,twoheadrightarrow,"q"] & & B_Q(G,n).
\end{tikzcd}
\]
where $m=m_{(G,n)}$, $B(G,n) = (\Z/m)^\times/\pm D(G,n)$ and $q$ is the natural quotient map.

(ii) In the Introduction, we considered (the $n=2$ case of) the group 
\[ N(G,n) := \ker(q \colon B(G,n) \twoheadrightarrow B_Q(G,n)).\]
This is the image of 
$[\IM(\Isom_\Psi(e^G))) \cdot \TriIsom(\wh e)] \cap \DiagIsom(\wh e) \le \DiagIsom(\wh e)$ under the surjection
 $$\DiagIsom(\wh e) \twoheadrightarrow \frac{\DiagIsom(\wh e)}{\DiagIsom_{\Psi}(e^G)} \cong \frac{\Aut(\wh \L)}{\Aut_{\psi}(\L^G)} \cong  (\Z/m)^\times/\{\pm 1\} $$ provided by \eqref{eq:diagquotient}.
\end{remark}

Note that the quadratic bias invariant $\beta_Q(M(X_1),M(X_2))$ was defined using the algebraic $2n$-doubles $\dalg(C_*(\wt X_i)) = (C_*(\wt{M(X_i)}), H_\ep( \pi_2(X_i)))$. The latter form does not necessarily coincide with the equivariant intersection form $S_{M(X_i)}$ and so is not clearly a function of the manifold $M(X_i)$ itself.
In \cref{prop:Dalg-commutes}, we showed that there is an equivalence
\[ d(\mathscr{C})(M(X)) = (C_*(\wt{M(X_i)}),S_{M(X_i)}) \simeq_{\Z} (C_*(\wt{M(X_i)}),H_\ep( \pi_2(X_i)). \]
Since the quadratic bias invariant $\beta_Q$ is a $\simeq_{\Z}$ invariant for algebraic $2n$-doubles (\cref{def:Z-isometry} and \cref{prop:PB-an-invariant}), it follows that $\beta_Q(M(X_1),M(X_2))$ depends only on $M(X_1)$ and $M(X_2)$ up to homotopy equivalence (or, more generally, up to $\simeq_{\Z}\,$).

In particular we have now established the following result which, alongside \cref{prop:beta-to-betaQ} below, implies \cref{thmx:main-general}.

\begin{theorem}\label{thm:main-high}
$\beta_Q(M(X_1),M(X_2)) \in B_Q(G,n)$ depends only on the manifolds $M(X_1)$ and $M(X_2)$ up to homotopy equivalence. In particular, if $\ol{X}$ is a reference minimal $(G,n)$-complex, then the quadratic bias invariant defines a map
\[ \beta_Q\colon \sM_{2n}(G) \to B_Q(G,n), \quad X \mapsto \beta_Q(X,\ol{X}). \]
Thus, $\beta_Q$ is an invariant of doubled minimal $(G,n)$-complexes up to homotopy equivalence.
\end{theorem}

In fact, \cref{prop:PB-an-invariant} actually proves the following stronger statement. 
Note that $\simeq_{\Z}\,$ can be viewed as an equivalence relation on manifolds via the map $d(\mathscr{C})$ defined in \cref{s:doubling-construction}.

\begin{proposition}
Let $X_1, X_2 \in \HT_{\min}(G,n)$. If $M(X_1) \simeq_{\Z} M(X_2)$, then $\beta_Q(M(X_1)) = \beta_Q(M(X_2))$. In particular, the quadratic bias invariant defines a map
\[ \beta_Q\colon\sM_{2n}(G)/\simeq_{\Z} \,\,\, \to \, B_Q(G,n). \]
\end{proposition}
The following is a consequence of \cref{prop:bias-vs-qbias}. 
Here $q\colon B(G,n) \to B_Q(G,n)$ is the map defined in \cref{remark:Aut(G)-is-easy} (i).

\begin{proposition}\label{prop:beta-to-betaQ}
There is a commutative diagram of sets
\[
\begin{tikzcd}
\HT_{\min}(G,n)
\ar[r,twoheadrightarrow,"\mathscr{D}"] \ar[d,"\text{$\beta$}"'] & \sM_{2n}(G)
 \ar[d,"\text{$\beta_Q$}"] \\
\displaystyle B(G,n) \ar[r,twoheadrightarrow,"q"] & \displaystyle B_Q(G,n)
\end{tikzcd}
\]
In particular, if $\beta\colon  \HT_{\min}(G,n) \to B(G,n)$ is surjective, then $\beta_Q\colon  \sM_{2n}(G) \to B_Q(G,n)$ is surjective. 
\end{proposition}

\begin{remark}
If $(G,n)$ does not satisfy the strong minimality hypothesis then, by \cref{remark:bias-minimality}, the bias invariant is zero and hence the quadratic bias invariant is zero.
\end{remark}

\subsection{Relationship between the quadratic 2-type and the quadratic bias} \label{ss:Q2type}

The quadratic $2$-type of a closed oriented (smooth or topological) $4$-manifold $M$ is the quadruple 
\[ Q(M) = [\pi_1(M), \pi_2(M), k_M, S_M], \]
where $k_M \in H^3(\pi_1(M);\pi_2(M))$ denotes the $k$-invariant and $S_M\colon  \pi_2(M) \times \pi_2(M) \to \Z[\pi_1(M)]$ denotes the equivariant intersection form. An isometry of two such quadruples is an isomorphism of pairs $\pi_1$, $\pi_2$ respecting the $k$-invariant and inducing an isometry on $S$ (similar definitions apply for $X$ a closed oriented Poincar\'e $4$-complex).

The data $[\pi_1(M), \pi_2(M), k_M]$ determine the \emph{algebraic $2$-type} $B = B(M)$, which is the total space of a $2$-stage Postnikov fibration $K(\pi_2(M), 2) \to B \to K(\pi_1(M), 1)$. 
A $B$-polarised oriented finite Poincar\'e $4$-complex is a $3$-equivalence $f\colon X \to B$. Let  $\mathscr{S}^{PD}_4(B)$ denote the set of  $B$-polarised homotopy types over $B$ (see \cite[\S 1]{Hambleton:1988} for more details).

We will now prove the following result from the Introduction (see \cref{thm:main-Q2type}).

\begin{theorem} \label{thm:main-Q2type-body}
The quadratic $2$-type determines the quadratic bias invariant.
More specifically, let $G$ be a finite group and let $M_1, M_2 \in \sM_4(G)$. If $Q(M_1) \cong Q(M_2)$, then $\beta_Q(M_1) = \beta_Q(M_2)$.
\end{theorem}

The proof of \cref{thm:main-Q2type-body} will be based on the following observation which is a direct consequence from the definition of $\beta_Q$. For convenience, we will work in the $B$-polarised setting. 

\begin{lemma}\label{lem:Ctobeta}
Let $G$ be a finite group and let $M_1, M_2 \in \sM_4(G)$. If there is a $B$-polarised equivalence $(C_*(\wt M_1), S_{M_1}) \simeq (C_*(\wt M_2), S_{M_2})$, then $\beta_Q(M_1) = \beta_Q(M_2)$.
\end{lemma}

In particular, it now suffices to prove the following, where $B = B(X)$ for some fixed  finite Poincar\'e $4$-complex $X$  with $\pi_1(X) = G$ as reference.

\begin{proposition}\label{prop:QtoC}
For $G$ a finite group,  let $X_1, X_2$ be oriented finite $B$-polarised Poincar\'{e} $4$-complexes with $Q(X_1) \cong Q(X_2)$. 
Then $(C_*(\wt M_1), S_{M_1})$ and $(C_*(\wt M_2), S_{M_2})$ are $B$-polarised homotopy equivalent.
\end{proposition}

\begin{proof}
First note that we can write $X_1 \simeq K \cup_g D^4$ where $K$ is a finite $3$-complex and the attaching map $g\colon  S^3 \to K$ is an element of $\pi_3(K)$.
Since $G$ is finite and $Q(X_1) \cong Q(X_2)$, it follows from \cite[Theorem 1.5]{KT21} that $X_1$ and $X_2$ differ by the action of an element $\alpha \in \Tors(\Z \otimes_{\Z G} \Gamma(\pi_2(B)))$ where $B$ is a 3-coconnected CW-complex which is 3-equivalent to both $X_1$ and $X_2$.
This action is described on \cite[p89-90]{Hambleton:1988}. We have that $\Gamma(\pi_2(B)) \cong H_4(\wt B)$. The action implies that $X_2 \simeq K \cup_{g+\alpha'} D^4$ where $\alpha' \in \pi_3(K)$ is obtained by choosing a preimage of 
$\alpha \in \Tors(H_4(\widetilde B) \otimes _\l \Z)$ in $H_4(\wt B)$ and mapping it under the composition $H_4(\wt B) \to H_4(\wt B, \wt K) \cong \pi_4(B, K) \to \pi_3(K)$. 

It now suffices to show that there is an equivalence $(C_*(\wt X_1), S_{X_1}) \simeq (C_*(\wt X_2), S_{X_2})$. However, by construction the image of $\alpha' \in \pi_3(K)$ is a torsion element which vanishes in $H_3(\widetilde K)$, and therefore does not affect the chain complexes. 
\end{proof}

\begin{proof}[The proof of \cref{thm:main-Q2type-body}]
For a closed oriented 4-manifold $M$, consider the pair of invariants $(C_*(\wt M), S_M)$.
For closed oriented 4-manifolds $M$ and $N$, an isometry $Q(M) \cong Q(N)$ implies that $B(M) \simeq B(N)$, and we can work in the $B$-polarised setting with $B:= B(M)$. We say that these pairs are $B$-polarised equivalent, which we write as $(C_*(\wt M), S_M) \simeq (C_*(\wt N), S_N)$, if there exists a chain homotopy equivalence $f\colon  C_*(\wt M) \to C_*(\wt N)$ over $B$ such that the induced map $f_*\colon  H_2(\wt M) \to H_2(\wt N)$ is an isometry $S_M \to S_N$ under the identifications $H_2(\wt M) \cong \pi_2(M)$ and $H_2(\wt N) \cong \pi_2(N)$.
This is a $B$-polarised homotopy invariant of $4$-manifolds in 
$\mathscr{S}^{PD}_4(B)$. The proof now follows from \cref{lem:Ctobeta} and \cref{prop:QtoC}.
\end{proof}

\section{Evaluation of the quadratic bias obstruction group}
\label{s:qbias_Lgroups}

Let $n \ge 2$ and let $G$ be a finite group. Recall from \cref{ss:bias-algebraic} that the invariant rank is defined to be $r_{(G,n)} = \rank_{\Z}(L^G)$ where $L = \pi_n(X)$ for $X$ any minimal $(G,n)$-complex. By \cref{prop:min-hyp}, $(G,n)$ satisfies the strong minimality hypothesis if and only if $r_{(G,n)} = d(H_n(G))$.

The main result of this section will be the following, which implies \cref{thmx:main-B_Q(G)}.

\begin{theorem}\label{thm:BQ-main}
Let $n \ge 2$ and let $G$ be a finite group such that $H_n(G) \cong (\Z/m)^d$ for some $m \ge 1$ and $d = r_{(G,n)}\geq 3$. Then:
\begin{clist}{(i)}
\item If $n$ is even, there are isomorphisms 
\[ PB_Q(G,n) \cong 
\frac{(\Z/m)^\times}{\pm (\cy m)^{\times 2}}
 \quad \text{\ and\ } \quad  B_Q(G,n) \cong  
 \frac{(\Z/m)^\times}{\pm (\cy m)^{\times 2}\cdot D(G,n)} \cdot
 \]
 \item If $n$ is odd, then $PB_Q(G,n) = B_Q(G,n) =0$.
 \end{clist}
\end{theorem}

\begin{remark}
If $d \ne r_{(G,n)}$, then $(G,n)$ does not satisfy the strong minimality hypothesis and \cref{remark:bias-minimality} implies that $B(G,n)=0$ and so $B_Q(G,n)=0$. Hence the above computes $B_Q(G,n)$ for all finite groups $G$ such that $H_n(G) \cong (\Z/m)^d$ for some $m \ge 1, d \ge 3$.
\end{remark}

This section will be structured as follows. In \cref{ss:PB-to-U}, we will begin by giving a formulation of $PB_Q(G,n)$ in the case where $H_n(G) \cong (\Z/m)^d$ in terms of unitary groups (\cref{thm:PB-H_2-elementary}).
In \cref{ss:proof-n-even}, we will prove \cref{thm:BQ-main} in the case where $n$ is even. The strategy will be to use the formulation in terms of unitary groups to establish a group homomorphism
\[ \phi\colon PB_Q(G,n) \longrightarrow \frac{(\Z/m)^\times}{\pm (\cy m)^{\times 2}} \]
by using the connection between unitary groups and algebraic $L$-theory.
Evaluating the $L$-groups and showing that $\phi$ is an isomorphism then leads to the result. The proof of \cref{thm:BQ-main} in the case where $n$ is odd is carried out in \cref{ss:proof-n-odd}. 

Detailed background on unitary groups and algebraic $L$-theory can be found in \cref{ss:L-odd} and the $L$-theory calculations are carried out in \cref{ss:L-theory-computations}.

\subsection{Formulation in terms of unitary groups} \label{ss:PB-to-U}

Let $(\l,-)$ be a ring with involution, let $\varepsilon =(-1)^n$ for some $n$, and let $d \ge 1$. The \textit{unitary group} $\U_{2d}(\La)$ is the subgroup of $GL_{2d}(\l)$ consisting of block matrices of the form
$\sigma = \left(\begin{smallmatrix} \alpha & \beta \\ \gamma & \delta \end{smallmatrix}\right)$
where $\alpha,\beta,\gamma, \delta \in M_d(\l)$ are such that
\begin{enumerate}
\item $\alpha \delta^* + (-1)^n\beta \gamma^* = I$
\item $\alpha\beta^*$ and $\gamma \delta^*$ each have the form $\theta - (-1)^n\theta^*$ for some  $d\times d$ matrix $\theta$.
\end{enumerate}
where, if $\alpha =(\alpha_{ij}) \in GL_{d}(\La)$, then $\alpha^* = (\wbar{\alpha}_{ji}) $ denotes the conjugate transpose matrix.
By \cref{remark:unitary-definitions} the unitary group $\U_{2d}(\La)$ is a subgroup of the \textit{hermitian unitary group} $\Isom(H_\ep(\l^d))$.

Define $D_{2d}(\l)$ to be the subgroup of $GL_{2d}(\l)$ consisting of matrices of the form 
\[ \left(\begin{smallmatrix} Q & 0 \\ 0 & (Q^*)^{-1} \end{smallmatrix}\right) \text{ where $Q = \left(\begin{smallmatrix} a &  & &  \\ &  1 & &  \\ &  & \ddots &  \\ &  & & 1 \end{smallmatrix}\right) \in GL_d(\l)$ for $a \in \La^\times$}. \]
This is a subgroup of $\U_{2d}(\La)$ for both choices of $\varepsilon \in \{\pm 1\}$.
Next recall that, in \cref{def:bias-quad}, we defined
\[ PB_Q(G,n) = \frac{\DiagIsom(\wh e)}{\left [\IM(\Isom_\Psi(e^G)) \cdot \TriIsom(\wh e)\right ] \cap \DiagIsom(\wh e)} \]	
where $e=e_L$ is the evaluation form and $L = \pi_n(X)$ for $X$ a reference minimal $(G,n)$-complex.
By hypothesis, we have that $L^G \cong \Z^d$ and $\wh L \cong H_n(G) \cong (\Z/m)^d$.
For the forms $e^G$ and $\wh e$, this implies that $\beta_1 = \cdots = \beta_d = m^{d-1}$ and so $e^G = m^{d-1} \cdot H_\ep(\Z^d)$ and $\wh e = m^{d-1} \cdot H_\ep((\Z/m)^d)$ are just scaled $\varepsilon$-hyperbolic forms for $\varepsilon=(-1)^n$ by \cref{prop:sevenfour}. In particular, there are isometries $e^G \cong H_\ep(\Z^d)$ and $\wh e \cong H_\ep((\Z/m)^d)$.

 By \cref{remark:wlog-psi-diagonal}, we can choose identifications $\L^G \cong \Z^d$ and $\wh \L \cong (\Z/m)^d$ so that $\psi\colon  L^G \twoheadrightarrow \wh L$ is reduction mod $m$.
 Since $\ker(\psi) \subseteq \Z^d$ is a characteristic subgroup, this implies that
\[ \Isom_{\Psi}(e^G) = \Isom(e^G) \cong \Isom(H_\ep(\Z^d)).\] 
In the case where $\l = \Z$ with the trivial involution, we have $\Isom(H_\ep(\Z^d)) = \U_{2d}(\Z)$ for $n$ even (\cref{example:unitary}). However, note that $\Isom(H_\ep(\Z^d)) \ne \U_{2d}(\Z)$ for $n$ odd and any $d \ge 1$.

The remainder of this section will be devoted to establishing the following.

\begin{proposition} \label{thm:PB-H_2-elementary}
Let $n \ge 2$, let $\varepsilon=(-1)^n$ and let $G$ be a finite group such that $H_n(G) \cong (\Z/m)^d$ where $m \ge 1$ and $d = r_{(G,n)}$. Then:
\begin{clist}{(i)}
\item If $n$ is even, there is 
an isomorphism
\[ PB_Q(G,n) \cong \frac{D_{2d}(\Z/m)}{\IM{(\U_{2d}(\Z))} \cap D_{2d}(\Z/m)} \cdot \]
\item If $n$ is odd,  there is a surjection 
$$\frac{D_{2d}(\Z/m)}{\IM(\U_{2d}(\Z))) \cap D_{2d}(\Z/m)} \twoheadrightarrow PB_Q(G,n).$$ 
\end{clist}
\end{proposition}

To prove this, first note that the discussion above implies that
\[ PB_Q(G,n) \cong \frac{\DiagIsom(\wh e)}{[\IM(\Isom(H_\ep(\Z^d))) \cdot \TriIsom(\wh e)] \cap \DiagIsom(\wh e)} \]
where $\wh e \cong H_\ep((\Z/m)^d)$.
It follows by the same argument as given in \cref{prop:PB-well-defined} that 
$[\IM(\U_{2d}(\Z)) \cdot \TriIsom(\wh e)] \cap \DiagIsom(\wh e)$ is a normal subgroup of $\DiagIsom(\wh e)$, where $\IM(\cdot)$ denotes the image under the reduction map $\U_{2d}(\Z) \to \U_{2d}(\Z/m)$. 
Thus we can define a group
\[ \wt{PB}_Q(G,n) := \frac{\DiagIsom(\wh e)}{[\IM(\U_{2d}(\Z)) \cdot \TriIsom(\wh e)] \cap \DiagIsom(\wh e)}.\]
There is a surjection $\pi: \wt{PB}_Q(G,n) \twoheadrightarrow PB_Q(G,n)$ induced by inclusion $\U_{2d}(\Z) \le \Isom(H_\ep(\Z^d))$. Note that, if $n$ is even, then $\U_{2d}(\Z) = \Isom(H_\ep(\Z^d))$ and so $\pi$ is an isomorphism.

The key technical step in the proof is the following, which shows that the subgroup of triangular isometries $\TriIsom(\wh e)$ does not contribute to $\wt{PB}_Q(G,n)$.

\begin{lemma}\label{lem:tri}
$[\IM(\U_{2d}(\Z)) \cdot \TriIsom(\wh e)] \cap \DiagIsom(\wh e) = \IM(\U_{2d}(\Z)) \cap \DiagIsom(\wh e)$. 
\end{lemma}

\begin{proof}
First note that
\begin{align*}
\DiagIsom(\wh e) &= \left\{ \left(\begin{smallmatrix} Q & 0 \\ 0 & (Q^*)^{-1} \end{smallmatrix}\right) : Q \in GL_d(\Z/m) \right\} \le GL_{2d}(\Z/m) \\
\TriIsom(\wh e) &= \left\{ \left(\begin{smallmatrix} I & P \\ 0 & I \end{smallmatrix}\right) : P \in M_d(\Z/m), \, P^* = -(-1)^nP \right\} \le GL_{2d}(\Z/m).
\end{align*} 
We claim that, if $A \in \DiagIsom(\wh e)$ and $B \in \TriIsom(\wh e)$, then $AB \in \IM{(\U_{2d}(\Z))}$ if and only if $A, B \in \IM{(\U_{2d}(\Z))}$.
Let $A = \left(\begin{smallmatrix} Q & 0 \\ 0 & (Q^*)^{-1} \end{smallmatrix}\right)$ and $B = \left(\begin{smallmatrix} I & P \\ 0 & I \end{smallmatrix}\right)$ where $Q \in GL_d(\Z/m)$ and $P \in M_d(\Z/m)$ is such that $P^* = - (-1)^nP$. If $AB \in \IM{(\U_{2d}(\Z))}$, then $AB \in \U_{2d}(\Z/m)$ and so
\[ \left(\begin{smallmatrix} Q & QP \\ 0 & (Q^*)^{-1} \end{smallmatrix}\right) = AB \in \U_{2d}(\Z/m). \]
It follows from the definition of $\U_{2d}(\Z/m)$ that $Q(QP)^* = E - (-1)^n E^*$ for some $E \in M_d(\Z/m)$ and so $P = F - (-1)^n F^*$ where $F = Q^{-1}E^* (Q^*)^{-1} \in M_d(\Z/m)$. 
Let $\wt F \in M_d(\Z)$ be any integral lift of $F$ and let $\wt P = \wt F - (-1)^n (\wt F)^*$. Then $\left(\begin{smallmatrix} I & \wt P \\ 0 & I \end{smallmatrix}\right) \in \U_{2d}(\Z)$ and so 
\[ B = \left(\begin{smallmatrix} I & P \\ 0 & I \end{smallmatrix}\right) = \psi_*\left[\left(\begin{smallmatrix} I & \wt P \\ 0 & I \end{smallmatrix}\right)\right] \in \IM{(\U_{2d}(\Z))}. \]
It follows that $A \in \IM{(\U_{2d}(\Z))}$. The converse is clear. This completes the proof of the claim.

Let $N = \IM{(\U_{2d}(\Z))} \cdot \TriIsom(\wh e)$.
By the claim, we have $N \cap \DiagIsom(\wh e) = \IM{(\U_{2d}(\Z))} \cap \DiagIsom(\wh e)$. In particular, suppose $AB \in N \cap \DiagIsom(\wh e)$ for some $A \in  \IM{(\U_{2d}(\Z))}$ and $B \in \TriIsom(\wh e)$. Since $AB \in \DiagIsom(\wh e)$, $B^{-1} \in \TriIsom(\wh e)$ and $(AB) B^{-1} \in  \IM{(\U_{2d}(\Z))}$, the claim implies that $AB \in  \IM{(\U_{2d}(\Z))}$ and so $AB \in \IM{(\U_{2d}(\Z))} \cap \DiagIsom(\wh e)$.
\end{proof}

\begin{proof}[Proof of \cref{thm:PB-H_2-elementary}]
By \cref{lem:tri}, we now have that
\[ \wt{PB}_Q(G,n) \cong \frac{\DiagIsom(\wh e)}{\IM(\U_{2d}(\Z)) \cap \DiagIsom(\wh e)}. \]
By \cref{prop:sevenfifteen} and the discussion that followed, we have that
\[ \frac{\DiagIsom(\wh e)}{\DiagIsom_{\Psi}(e^G)} \cong (\Z/m)^\times /\{\pm 1\} \]
and we can see directly that $\DiagIsom_{\Psi}(e^G) = \DiagIsom(H_\ep(\Z^d)) \le \U_{2d}(\Z)$.
The subgroup $D_{2d}(\Z/m) \le \DiagIsom(\wh e)$, which is isomorphic to $(\Z/m)^\times$, maps surjectively onto $(\Z/m)^\times/\{\pm 1\}$. This implies that the inclusion map induces an isomorphism
\[ \wt{PB}_Q(G,n) \cong \frac{D_{2d}(\Z/m)}{\IM(\U_{2d}(\Z)) \cap D_{2d}(\Z/m)}.\]  
This completes the proof since there is a surjection $\pi\colon \wt{PB}_Q(G,n) \twoheadrightarrow PB_Q(G,n)$ which is an isomorphism provided $n$ is even.
\end{proof}

\subsection{Proof of \cref{thm:BQ-main} for $n$ even} \label{ss:proof-n-even}

Let $n \ge 2$ be even. We will now evaluate $PB_Q(G,n)$ using the form established in \cref{thm:PB-H_2-elementary} (i). 

\begin{proposition} \label{prop:hom-to-Lgroups}
Let $n \ge 2$ be even, let $G$ be a finite group, and suppose that $H_n(G) \cong (\Z/m)^d$ where $m \ge 3$ and $d = r_{(G,n)}$.
Then the quotient map $ D_{2d}(\Z/m) \to (\Z/m)^\times/\{\pm 1\} $ induces a homomorphism 
\[\phi\colon PB_Q(G,n) \cong
 \frac{D_{2d}(\Z/m)}{\IM(U_{2d}(\Z)) \cap D_{2d}(\Z/m)} \longrightarrow
  \frac{(\Z/m)^\times}{\pm (\cy m)^{\times 2}}.  \]
\end{proposition}

\begin{proof} 
To check that $\phi$ is well-defined, we must show that 
\[ \phi(\IM{(U_{2d}(\Z))} \cap D_{2d}(\Z/m)) \subseteq  \{\pm (\cy m)^{\times 2}\}. \]
Since $m \ge 3$, we have that
\[ \IM{(U_{2d}(\Z))} \cap D_{2d}(\Z/m) = \IM{(SU_{2d}(\Z))} \cap D_{2d}(\Z/m). \]
To see this note that, if $A \in GL_{2d}(\Z)$ and the reduction $\ol{A} \in GL_{2d}(\Z/m)$ has $\det(\ol{A}) = 1$, then $\det(A) \in \{\pm 1\}$ and $\det(A) \equiv \det(\ol{A}) \equiv 1 \mod m$ which implies that $\det(A)=1$ since $m \ge 3$.

Let $\rho_m\colon  SU_{2d}(\Z) \to SU_{2d}(\Z/m)$ denote reduction mod $m$.
Then we have a commutative diagram
\[
\begin{tikzcd}
\rho_m^{-1}(D_{2d}(\Z/m)) \ar[r,hookrightarrow] \ar[d,"\rho_m"] & SU_{2d}(\Z) \ar[d,"\rho_m"] \ar[r] & SU(\bZ)/ RU(\bZ)  = L_1^s(\bZ) \ar[d,"\rho_m"] \\
	D_{2d}(\Z/m) \ar[r,hookrightarrow] & SU_{2d}(\Z/m) \ar[r] & SU(\cy m)/RU(\cy m) =L_1^s(\cy m).
\end{tikzcd}
\]
The image of the subgroup 
$$N:= \IM(U_{2d}(\Z)) \cap D_{2d}(\Z/m) \le \IM(SU_{2d}(\Z) \to SU_{2d}(\cy m))$$ after stabilisation gives a subgroup 
$$\overline{N} \subseteq \IM(L_{1}^s(\bZ) \to L_{1}^s(\cy m)).$$
The calculations of Wall \cite[\S 1]{Wall:1976} show that $L^s_1(\bZ) = L^h_1(\bZ) = 0$, so that
$$0 = \overline{N} \subseteq \ker(L_{1}^s(\cy m) \to L_{1}^h(\cy m)) \cong
  (\Z/m)^\times/\pm (\cy m)^{\times 2}$$
by naturality and \cref{prop:Lgroups-main}  (based on the calculations of \cref{lemma:L_1(Z),lemma:L_1(Z/m)}).
Moreover,  the composite under stabilisation 
$$D_{2d}(\cy m) \to SU_{2d}(\cy m) \to L_1^s(\cy m) \cong 
(\Z/m)^\times/\pm (\cy m)^{\times 2}$$
is just the quotient map used to define $\phi$ (see \cref{cor:exceptcor}).
Therefore  $\phi(\IM(U_{2d}(\Z)) \cap D_{2d}(\Z/m))\subseteq \{\pm (\Z/m)^{\times 2}\}$, and $\phi$ is well-defined.
\end{proof}

We now claim that the map $\phi$ is an isomorphism provided that, in addition, we have $d \ge 3$.
We will begin by establishing the following lifting result for squares, where $(D_{2r}(\Z/m))^2$ denotes the subgroup $\{a^2 : a \in D_{2d}(\Z/m)\}$.

\begin{proposition}
\label{lemma:lifting-squares}
Let $m \ge 1$, $d \ge 3$,  and let $\varepsilon \in \{\pm 1\}$.
Then 
\[(D_{2d}(\Z/m))^2 \subseteq \IM(\EU_{2d}(\Z)) \subseteq \IM(\U_{2d}(\Z))  \] 
where $\IM(\cdot)$ denotes the image under the mod $m$ reduction  map $\U_{2d}(\Z) \to \U_{2d}(\Z/m)$. 
\end{proposition}

The proof will be based on the following well-known identities. For $r \ge 1$, let $I_r$ denote the $n \times n$ identity matrix and $\oplus$ denote the orthogonal block sum of matrices.

\begin{lemma} \label{lemma:identities}
Let $\l$ be a ring and set $I = I_r$ for some $r \geq 1$.
\begin{clist}{(i)}
\item
If $A, B \in GL_r(\l)$, then
\[
A^{-1}B^{-1}AB \oplus I = ((BA)^{-1} \oplus BA)(A \oplus A^{-1})(B\oplus B^{-1}) \in GL_{2r}(\l). \]
\item
If $A \in GL_r(\l)$, then
\[ A \oplus A^{-1} = \left(\begin{smallmatrix} I & A \\ 0 & I \end{smallmatrix}\right) \left(\begin{smallmatrix}I & 0 \\ I-A^{-1} & I \end{smallmatrix}\right)\left(\begin{smallmatrix} I & -I \\ 0 & I \end{smallmatrix}\right)\left(\begin{smallmatrix} I & 0 \\ I-A & I\end{smallmatrix}\right) \in GL_{2r}(\l).\]
\end{clist}
\end{lemma}

The two identities can be verified directly by multiplication. The former appears in \cite[Proof of Theorem 6.3]{Wall:1970} and the latter is the Whitehead identity.

\begin{lemma}\label{lem:3step}
 Let $\l$ be a ring with involution. If $Q \in GL_r(\La)$, then $Q \oplus Q^{-1}\oplus I_r \in EU^\ep_{3r}(\La)$.
\end{lemma}
\begin{proof} By matrix multiplication, we have 
$$Q \oplus Q^{-1}\oplus I_r =  \left ( Q \oplus I_r\oplus (Q^*)^{-1} \right ) \left (I_r \oplus Q\oplus (Q^*)^{-1} \right )^{-1}\in U^\ep_{3r}(\La)$$
and, since $Q \in GL_r(\La)$, we have $Q \oplus (Q^*)^{-1}  \in EU^\ep_{2r}(\La)$ by definition of the elementary unitary relations. The result follows by stabilisation.
\end{proof}

\begin{proof}[Proof of \cref{lemma:lifting-squares}]
It suffices to treat the case $d=3$ since, after reordering bases, an arbitrary element of $(D_{2d}(\Z/m))^2$ has the form $\left(\begin{smallmatrix}a^2 & 0 \\ 0 & a^{-2} \end{smallmatrix}\right) \oplus I_{2d-2}$ for some $a \in (\Z/m)^\times$. 

We start by applying \cref{lemma:identities} (i)  to the matrices 
$$A = \left(\begin{smallmatrix}a^{-1} & 0 \\ 0\sx & a \end{smallmatrix}\right),  \quad B = \left(\begin{smallmatrix} 0 & 1 \\ 1 & 0 \end{smallmatrix}\right) \quad \text{and\ } \quad (BA)^{-1} = \left(\begin{smallmatrix}0 \sx& a \\ a^{-1} & 0 \end{smallmatrix}\right)$$  
in $GL_2(\Z/m)$. Since
$A^{-1}B^{-1}AB  = \left(\begin{smallmatrix}a^2 & 0 \\ 0 & a^{-2} \end{smallmatrix}\right)$, \cref{lemma:identities} (i) with $r=2$ gives that
$$ \left(\begin{smallmatrix}a^2 & 0 \\ 0 & a^{-2} \end{smallmatrix}\right) \oplus I_2 = A^{-1}B^{-1}AB \oplus I_2 =
((BA)^{-1} \oplus BA)(A \oplus A^{-1})(B\oplus B^{-1})  \in GL_4(\cy m).$$
It remains to show that the three matrices on the right-hand side are, after one stabilisation, in the image of  $EU^\ep_{6}(\bZ)$ under the map induced by reduction mod $m$. 

First, it is clear that $B$ lifts to $\wt B = \left(\begin{smallmatrix} 0 & 1 \\ 1 & 0 \end{smallmatrix}\right) \in GL_2(\Z)$. Since $B$ is symmetric and $\Z$ has trivial involution, $B \oplus B^{-1} = B \oplus (B^*)^{-1}$ lifts to $\wt B \oplus (\wt B^*)^{-1} \in \EU_4(\Z)$.
Secondly, by \cref{lemma:identities} (ii), we have that
\[ A = \left(\begin{smallmatrix} a^{-1} & 0 \\ 0 \sx & a \end{smallmatrix}\right) = \left(\begin{smallmatrix} 1 & a^{-1} \\ 0 & 1\sx \end{smallmatrix}\right)\left(\begin{smallmatrix} 1 & 0 \\ 1-a & 1 \end{smallmatrix}\right)\left(\begin{smallmatrix} 1 & -1 \\ 0 & \sx 1 \end{smallmatrix}\right)\left(\begin{smallmatrix} 1 & 0 \\ 1-a^{-1} & 1 \end{smallmatrix}\right) \]
and each of the matrices on the right hand side lifts to $ GL_2(\Z)$.  Hence $A$ lifts to $\tilde A \in GL_2(\bZ)$. Since $A$ is symmetric and $\cy m$ has trivial involution, $A \oplus A^{-1} = A \oplus (A^*)^{-1}$ lifts to $\tilde A \oplus ((\tilde A)^*)^{-1} \in EU_4(\bZ)$.
Finally, $BA$ lifts to $\wt B \wt A \in GL_2(\Z)$ and so $(BA)^{-1} \oplus BA \oplus I_2$ lifts to $(\wt B \wt A)^{-1} \oplus (\wt B \wt A) \oplus I_2$. By \cref{lem:3step}, this is contained in $EU^\ep_6(\cy m)$.
This completes the proof.
\end{proof}

\begin{proof}[Proof of \cref{thm:BQ-main} {\text{\normalfont (i)}}]
By combining \cref{thm:PB-H_2-elementary} (i) and \cref{prop:hom-to-Lgroups}, we have a map
\[ PB_Q(G,n) \cong \frac{D_{2d}(\Z/m)}{\IM(U_{2d}^\varepsilon(\Z)) \cap D_{2d}(\Z/m)} 
\xrightarrow{\ \quad \phi\quad } \frac{(\Z/m)^\times}{\pm (\cy m)^{\times 2}}\,\cdot \]
We claim that $\phi$ is bijective (it is clearly surjective by definition).
The result holds for $m \le 2$ since both sides are trivial. For example, if $m =2$, then we have $D_{2d}(\Z/2) \cong (\Z/2)^{\times} = \{1\}$. From now on, we will assume that $m \ge 3$. 

To see that $\phi$ is injective, suppose that $A=(a) \oplus (a^{-1}) \oplus I_{2d-2} \in D_{2d}(\Z/m)$ has $\phi(A) = 0$. This implies that $a\in \pm (\cy m)^{\times 2}$ and so $A \in \pm (D_{2d}(\Z/m))^{\times 2}$. By \cref{prop:sevenfifteen}, we have that $-I_{2d} \in \rho_m^{-1}(D_{2d}(\Z/m))$ where $\rho_m\colon  SU_{2d}(\Z) \to SU_{2d}(\Z/m)$ denote reduction mod $m$. By \cref{lemma:lifting-squares}, we have that $(D_{2d}(\Z/m))^{\times 2} \subseteq \rho_m^{-1}(D_{2d}(\Z/m))$. Hence the image $[A] = 0 \in PB_Q(G,n)$, and $\phi$ is injective.
\end{proof}

\subsection{Proof of \cref{thm:BQ-main} for $n$ odd} \label{ss:proof-n-odd}

Let $n > 2$ be odd so that $\ep = -1$. In order to show that $PB_Q(G,n)=0$, it is enough to check that every element in $D_{2d}(\Z/m)\subseteq \SU_{2d}(\cy m)$ can be lifted to $\U_{2d}(\bZ)$ (see \cref{thm:PB-H_2-elementary} (ii)).
From  \cref{prop:except}, we have $L_3^s(\cy m)=\SU(\cy m)/\RU(\cy m) =0$, and hence the image of 
$D_{2d}(\Z/m)$ after stabilisation is zero in $L_3^s(\cy m)$. 
In other words, $D_{2d}(\Z/m) \subseteq  \RU(\cy m)$.

By stability results for unitary groups (see \cite{Bak:2003}, \cite[Chapter VI]{Knus:1991}), and the fact that the ring $\La = \cy m$ has stable range 1, it follows that $RU_{2d}^\ep(\cy m) = \RU(\cy m)$ provided  that $d \geq 3$. Hence any element in $RU_{2d}^\ep(\cy m)$ can be expressed as a product of elementary special unitary matrices (see the list in \cref{ss:L-odd}, with the additional assumption $\det Q = 1$ in item (i)).  We apply this observation to each  element $A = \left(\begin{smallmatrix}a & 0 \\ 0 & a^{-1} \end{smallmatrix}\right) \in D_{2d}(\Z/m)$. Each of the terms in this product can be lifted to $\U_{2d}(\bZ)$, by the methods used in the proof of \cref{lemma:lifting-squares}, and hence $PB_Q(G,n) = 0$ for $n$ odd. 

\section{Examples of homotopy inequivalent $(G,n)$-complexes} \label{s:examples-complexes}

The aim of this section will be to survey examples of finite $(G,n)$-complexes $X$ and $Y$ with $\pi_1(X) \cong \pi_1(Y)$ and $\chi(X)=\chi(Y)$ but which are not homotopy equivalent. The 4-manifolds which we construct in \cref{s:examples-manifolds} in order to prove \cref{thmx:main-examples}  will all be constructed by applying the doubling construction to the examples here.

Let $\beta$ denote the bias invariant defined in \cref{s:bias-complexes}. For a finite abelian group $G$ and $r \in (\Z/m_{(G,n)})^\times$, define $X^r_{G,n}$ to be the finite $(G,n)$-complexes defined in \cite[Proof of Proposition 6]{Sieradski:1979}. For example, when $n=2$ and $G = \Z/{m_1} \times \cdots \times \Z/{m_d}$ with $m_i \mid m_{i+1}$ for all $i$ and $m_1 > 1$ (and so $m_1 = m_G$), we have that $X^r_{G,2}$ coincides with the presentation complex for the presentation:
\[ \mathcal{P}_r = \langle x_1, \cdots, x_d \mid x_1^{m_1}, \cdots, x_n^{m_d}, [x_1^r,x_2], \{[x_i,x_j]\colon i<j, \, (i,j) \ne (1,2)\} \rangle. \]

We will now point out the following. This is a consequence of work of Metzler \cite{Metzler:1976}, Sieradski \cite{Sieradski:1977}, Sieradski-Dyer \cite{Sieradski:1979}, Browning \cite{Browning:1979c,Browning:1979} and Linnell \cite{Linnell:1993}. However, as far as we know, this observation has not previously appeared in the literature for all $n \ge 2$.

\begin{theorem} \label{thm:bias-surj-abelian}
Let $n \ge 2$, let $G$ be a finite abelian group and let $m = m_{(G,n)}$. Then:
\begin{clist}{(i)}
\item
The bias invariant gives a bijection
\[ \beta\colon \HT_{\min}(G,n) \to B(G,n).\]
Furthermore, $\beta(X^r_{G,n}) = [r]$ for all $r \in (\Z/m)^\times$ and so every minimal finite $(G,n)$-complex $X$ has $X \simeq X^r_{G,n}$ for some $r \in (\Z/m)^\times$.
\item
If $X$, $Y$ are finite $(G,n)$-complexes with $(-1)^n\chi(X)=(-1)^n\chi(Y)>\chi_{\min}(G,n)$, then $X \simeq Y$.
\end{clist}
\end{theorem}

This is achieved using the Browning invariant, as introduced by Browning in \cite{Browning:1979c}. For a finite group $G$ and finite $(G,n)$-complexes $X$ and $Y$, this is an invariant $B(X,Y) \in \Br(G,n)$ where $\Br(G,n)$ is the Browning obstruction group (see also \cite[Section 2]{Latiolais:1993}). If $G$ is a finite abelian group, then this is a complete homotopy invariant.

\begin{proof}
(i) We begin by noting that, if $X$ is a minimal finite $(G,n)$-complex, then $X \simeq X^r_{G,n}$ for some $r \in (\Z/m)^\times$. The case $n=2$ is \cite[Theorem 1.7]{Browning:1979c}. For the general case, the result follows from the fact that, as noted by Linnell in \cite[p318]{Linnell:1993}, the complexes $X^r_{G,n}$ realise all the elements of the Browning obstruction group $\Br(G,n)$.

Note that $\beta$ is surjective since $\beta(X^r_{G,n}) = [r]$ for all $r \in (\Z/m)^\times$ \cite[Proposition 6]{Sieradski:1979}. To show that $\beta$ is injective, it remains to prove that the complexes $X^r_{G,n}$ are determined up to homotopy equivalence by $\beta$. The case $n=2$ was proven by Sieradski \cite{Sieradski:1977} by constructing explicit homotopy equivalences.
The general case is a consequence of the comparison between \cite[Proposition 8]{Sieradski:1979}, which gives a lower bound on $|\HT_{\min}(G,n)|$ by computing $\IM(\beta)$, and \cite[Theorem 1.3]{Linnell:1993} which shows that $|\HT_{\min}(G,n)|$ is equal to this lower bound (see also the discussion on \cite[p307]{Linnell:1993}).

(ii) This follows from a result of Browning \cite[Theorem 5.4]{Browning:1979} (see also \cite[Theorem 1.1]{Linnell:1993}).
\end{proof}

Define $\gamma(G,n) = |\HT_{\min}(G,n)|$. The above shows that, for $n \ge 2$ and $G$ finite abelian, we have $\gamma(G,n) = |B(G,n)|$. Recall from \cref{ss:bias-complexes} that $B(G,n) = (\Z/m)^\times/\pm D(G,n)$ where $D(G,n) = \IM(\varphi_{(G,n)} \colon \Aut(G) \to (\Z/m)^\times/\{\pm 1\})$.

The following was shown by Browning \cite{Browning:1979} in the case $n=2$, and Sieradski-Dyer \cite[Proposition 8]{Sieradski:1979} and Linnell \cite[Theorem 1.3 \& Corollary 1.5]{Linnell:1993} in the general case.

\begin{proposition} \label{prop:compute-D(Gn)}
Let $n \ge 2$, let $G$ be a finite abelian group, let $m=m_{(G,n)}$ and $d = d(G)$.
Then:
\[ D(G,n) = ((\Z/m)^\times)^{e(d,n)}\] 
where 
\vspace{-2mm}
\[ 
e(d,n) = 
\begin{cases}
\sum_{i=0}^{\frac{n}{2}-1} \frac{n-2i}{2} \binom{d+2i-1}{d-2}, & \text{if $n$ is even and $d \ge 2$} \\
\sum_{i=0}^{\frac{n-1}{2}} \frac{n+1-2i}{2} \binom{d+2i-2}{d-2}, & \text{if $n$ is odd and $d \ge 2$}
\end{cases}
\]
and $e(1,n)=1$ for $n$ even, $e(1,n) = \frac{1}{2}(n+1)$ for $n$ odd.

In particular, we have $\gamma(G,n) = |(\Z/m)^\times / \pm ((\Z/m)^\times)^{e(d,n)}|$.
\end{proposition} 

\begin{remark} \label{remark:KS-error}
In \cite[Proposition 8]{Sieradski:1979}, Sieradski-Dyer proved that the number of minimal finite $(G,n)$-complexes up to homotopy equivalence was at least the bound given above. However, as pointed out in \cite[p305]{Linnell:1993}, the formula for $e(d,n)$ given in \cite[p210-211]{Sieradski:1979} is incorrect when $n \ge 3$. 
\end{remark}

For example, $e(d,2) = d-1$ for $d \ge 2$. Using elementary number theory, it is possible to evaluate $\gamma(G,n)$ precisely (see \cite[p137]{Sieradski:1977}).

We will now consider the homotopy classification of finite $(G,n)$-complexes $X$ for which $\pi_n(X)$ is a fixed $\Z G$-module. For a minimal finite $(G,n)$-complex $X$, define
\[ \gamma'(G,n) = |\{\, Y \in \HT_{\min}(G,n)\colon \pi_n(X) \cong_{\Aut(G)} \pi_n(Y) \}|.\] 
Let $N = \sum_{g \in G} g \in \Z G$ denote the group norm and let $(N,r) = \Z G \cdot N + \Z G \cdot r \le \Z G$ denote the Swan module corresponding to $r \in \Z$.
In what follows, we write $[\,\cdot\,]$ to denote the quotient map $(\Z/m)^\times/\{\pm 1\} \twoheadrightarrow (\Z/m)^\times/\pm ((\Z/m)^\times)^{e(d,n)} \cong B(G,n)$.

\begin{theorem} \label{thm:gamma'}
Let $n \ge 2$, let $p$ be a prime, let $G$ be a non-cyclic abelian $p$-group, let $d = d(G)$ and $m = m_{(G,n)}$. Then:
\begin{clist}{(i)}
\item
 For a minimal finite $(G,n)$-complex $X$, we have
\[ \beta(\{\, Y \in \HT_{\min}(G,n) : \pi_n(X) \cong_{\Aut(G)} \pi_n(Y)\}) = \beta(X) \cdot [V^{s(d,n)}] \subseteq B(G,n) \]
where $s(d,n) = \sum_{i=0}^{n-1} (-1)^{n+i+1}\binom{d+i}{d-1}$ and
\[ V = \{ [r] \in (\Z/m)^\times : r \in \Z, \, (r,|G|)=1 \text{ and } (N,r) \cong \Z G \}. \]
\item
$\gamma'(G,n)$ is equal to the number of elements in $[V^{s(d,n)}] \subseteq B(G,n)$. In particular, it depends only on $G$ and $n$, and not on the choice of $X$.
\item If $p =2$, then $\gamma'(G,n) = 1$. If $p$ is odd, then
\[ 
\gamma'(G,n) = \frac{\gcd(e(d,n),\frac{1}{2}(p-1))}{\gcd(e(d,n),s(d,n),\frac{1}{2}(p-1))} \,\cdot
\]
\end{clist}
\end{theorem}

\begin{proof}
This is proven by combining \cref{thm:bias-surj-abelian} with results of Linnell \cite{Linnell:1993}. More specifically, (i) follows from \cite[Theorem 8.4 (iii)]{Linnell:1993}, (ii) follows from (i), and (iii) follows from \cite[Theorems 1.2 (4)]{Linnell:1993}.
\end{proof}

We will now use this to show the following.

\begin{theorem} \label{thm:e-s-ex}
Let $n \ge 2$. Then there exists $d \ge 3$ such that, if $p \equiv 1 \mod 4$ is prime and $G = C_{p^{m_1}} \times \cdots \times C_{p^{m_d}}$ with $m_i \le m_{i+1}$ for all $i$ and $m_1 \ge 1$, then $\gamma'(G,n) > 1$.	 
In particular, there exist finite $(G,n)$-complexes $X$ and $Y$ such that $\pi_n(X) \cong \pi_n(Y)$ as $\Z G$-modules but $X \not \simeq Y$.
\end{theorem}

This will be a consequence of combining \cref{thm:gamma'} with the following result concerning the parity of $e(d,n)$ and $s(d,n)$. 
 For the remainder of this section, let $\equiv$ denote equivalence mod $2$.

\begin{proposition} \label{prop:e-s-parity}
Let $n \ge 2$. Then there exists $d \ge 3$ such that $e(d,n) \equiv 0$ and $s(d,n) \equiv 1$.
\end{proposition}

The proof is based on detailed calculations presented in Appendix \ref{s:funct}. 

\begin{proof}
We refer to Appendix \ref{s:funct} for the relevant properties of the functions $e(d,n)$ and $s(d,n)$. Here is a summary of the results:

\begin{itemize}
\item If $n = 4k$, then $e(4t,4k) \equiv 0$ for any $t \ge 1$ by Proposition \ref{prop:e(d,n)-eval}. By Proposition \ref{prop:s(d,n)-eval}, we have $s(4t,4k) \equiv 1 + \binom{k+t}{t}$. Let $k = 2^m r$ where $m \ge 0$ and $r \ge 1$ is odd. If we take $t = 2^m$, then $s(2^{m+2},4k) \equiv 1+\binom{2^mr+2^m}{2^m} \equiv 1 + \binom{r+1}{1} \equiv r \equiv 1$ using \cref{lemma:mod2-useful}.

\item If $n = 4k+1$, then $s(4t+1,4k+1) \equiv 1$ for any $t \ge 1$ by Proposition \ref{prop:s(d,n)-eval}. By Proposition \ref{prop:e(d,n)-eval}, we have $e(4t+1,4k+1) \equiv \binom{k+t}{t}$. Similarly to the $n =4k$ case, let $k = 2^m r$ where $m \ge 0$ and $r \ge 1$ is odd. If we take $t = 2^m$, then $e(2^{m+2}+1,4k+1) = \binom{2^mr+2^m}{2^m} \equiv r+1 \equiv 0$.

\item If $n = 4k+2$, then $e(d,4k+2) \equiv 0$ and $s(d,4k+2) \equiv 1$ whenever $d \equiv 3 \mod 4$, by Propositions \ref{prop:e(d,n)-eval} and \ref{prop:s(d,n)-eval}. For example, we can take $d = 3$.

\item If $n = 4k+3$, then	$e(d,4k+3) \equiv 0$ and $s(d,4k+3) \equiv 0$ whenever $d \equiv 1 \mod 4$, by Propositions \ref{prop:e(d,n)-eval} and \ref{prop:s(d,n)-eval}. For example, we can take $d = 5$. \qedhere
\end{itemize}
\end{proof}

\begin{proof}[Proof of \cref{thm:e-s-ex}]
By \cref{prop:e-s-parity}, there exists $d \ge 3$ such that $e(d,n) \equiv 0$ and $s(d,n) \equiv 1$. Let $p$ be a prime such that $p \equiv 1 \mod 4$ and $G = C_{p^{m_1}} \times \cdots \times C_{p^{m_d}}$ with $m_i \le m_{i+1}$ for all $i$ and $m_1 \ge 1$.
By \cref{thm:gamma'}, we have that
\[ \gamma'(G,n) = \frac{\gcd(e(d,n),\frac{1}{2}(p-1))}{\gcd(e(d,n),s(d,n),\frac{1}{2}(p-1))}.
 \]

Since $e(d,n) \equiv 0$,  and $p \equiv 1 \mod 4$ implies that $\frac{1}{2}(p-1) \equiv 0$,  we have $2 \mid \gcd(e(d,n),\frac{1}{2}(p-1))$. Conversely, since $s(d,n) \equiv 1$ and $2 \nmid \gcd(e(d,n),s(d,n),\frac{1}{2}(p-1))$, we have $2 \mid \gamma'(G,n)$.	
\end{proof}

\section{Examples of homotopy inequivalent doubled $(G,n)$-complexes}
\label{s:examples-manifolds}

The aim of this section will be to use the quadratic bias invariant to distinguish certain doubled $(G,n)$-complexes $M(X)$ up to homotopy equivalence. 

The examples must take as input a finite group $G$ and a pair of finite $(G,n)$-complexes $X$, $Y$ with $\chi(X)=\chi(Y)$ but which are not homotopy equivalent. Such examples have previously only been known to exist when $G$ is either a finite abelian group \cite{Metzler:1976,Sieradski:1979} or a group with periodic cohomology \cite{Dy76,Ni20-II,Ni21}. 
In light of \cref{thm:BQ-main} (ii), we must restrict to the case where $n \ge 2$ is even.
If $G$ has periodic cohomology, then $H_n(G)=0$ for all $n \ge 2$ even \cite[Corollary 2]{Sw71} and so the quadratic bias invariant contains no information.

For this reason, we will begin by restricting to examples over finite abelian fundamental groups (\cref{ss:examples-abelian}). In \cref{ss:examples-non-abelian}, we will demonstrate that the quadratic bias invariant is computable more generally by constructing examples over the non-abelian group $Q_8 \times (\Z/17)^3$. This also gives the first example of a non-abelian finite group $G$ which does not have periodic cohomology such that there exists homotopically distinct finite $(G,n)$-complexes $X$, $Y$ with $\chi(X)=\chi(Y)$.

\subsection{Examples with abelian fundamental groups} \label{ss:examples-abelian}

The aim of this section will be to establish the following two results, which imply \cref{thmx:main-examples} and \cref{thm:main-higher-dim} from the Introduction.

\begin{theorem} \label{thm:final-examples}
Let $n \geq 2$ be even and let $k \geq 2$.  Then there exist closed smooth $2n$-manifolds $M_1, M_2, \dots, M_k$ which are all stably diffeomorphic but not pairwise homotopy equivalent. 

Furthermore, if $m$ is an integer with at least $1+\log_2(k)$ prime factors, then the examples can be taken to have fundamental group $(\Z/m)^d$ for some $d \ge 3$ or $(\Z/m)^{d} \times \Z/t$ for some $d \ge 4$ and any integer $t > 1$ such that $t \mid m$ and $m$, $m/t$ have the same prime factors.
\end{theorem}

\begin{theorem} \label{thm:final-examples-eqforms}
Let $n \geq 2$ be even. Then there exist closed smooth $2n$-manifolds $M$ and $N$ which are stably diffeomorphic, have isometric hyperbolic equivariant intersection forms, but are not homotopy equivalent.
Furthermore, for any prime $p$ such that $p \equiv 1 \mod 4$, the examples can be taken to have fundamental group $(\Z/p)^d$ for some $d \ge 3$.
\end{theorem}

The proofs will rely on \cref{thm:BQ-main} which computes the quadratic bias obstruction group $B_Q(G,n)$ in the case where $G$ is a finite group such that $H_n(G) \cong (\Z/m)^r$ where $r=r_{(G,n)}$.
We start by giving a collection of finite abelian groups which satisfy this hypothesis.

\begin{lemma} \label{lemma:H_n(G)}
Let $n \ge 2$ and let $G = (\Z / m)^d \times \Z /t $ where $d \geq 3$ and $m, t \ge 1$ are such that $t \mid m$ and $m$, $m/t$ have the same prime factors.
Then
\[
H_n(G) \cong \begin{cases} (\Z/m)^r, & \text{if $n$ is even} \\
 (\Z/m)^r \times \Z/t, & \text{if $n$ is odd}	
 \end{cases}
\]
where $r = r_{(G,n)} \geq 3$.
\end{lemma}

In particular, the groups $G = (\Z/m)^d$ for $d \ge 3$ satisfy the hypothesis for all $n \ge 2$. If $n \ge 2$ is even, then the hypothesis is satisfied by the larger class of groups $G = (\Z/m)^d \times \Z/t$ where $d \ge 3$, $t \mid m$ and $m$, $m/t$ have the same prime factors.

\begin{proof}
Let $G = \prod_{i=1}^s G_i$ where $G_i$ is the $p_i$-primary component for some prime $p_i$. We have $H_n(G) \cong \prod_{i=1}^s H_n(G_i)$.
The hypothesis on $m$ and $t$ imply that $G_i \cong (\Z/p_i^{a_i})^d \times \Z/p_i^{b_i}$ where $a_i > b_i \ge 1$.
By \cite[Theorem 4.3]{DP17}, we have that
\[ H_n(G_i) \cong (\Z/p_i^{a_i})^r \times (\Z/p_i^{b_i})^{\nu(n,1)-(-1)^n} \]
where $r = \sum_{k = 2}^d (\nu(n,k)-(-1)^n)$ and $\nu(n,k) = \sum_{j=0}^n (-1)^{n+j} \binom{k+j-1}{j}$. We can verify directly that, if $d \ge 3$, then $r \ge 3$. Furthermore, $\nu(n,1) = 1$ if $n$ is even and $\nu(n,1)=0$ if $n$ is odd. This gives the required form for $H_n(G)$ by recombining the $p_i$-primary componeents.

Finally, the fact that $r= r_{(G,n)}$ follows since $G$ satisfies the strong minimality hypothesis (see \cref{ss:min-hyp} and \cite[Proposition 5]{Sieradski:1979}).
\end{proof}

Fix $n \ge 2$ even, $m \ge 2$, $d \ge 3$ and $t \mid m$ such that $m$, $m/t$ have the same prime factors. Let $G = (\Z/m)^d \times \Z/t$. By \cref{thm:BQ-main} and \cref{lemma:H_n(G)}, there is an isomorphism
\[ B_Q(G,n) \cong \frac{(\Z/ m)^\times}{\pm (\Z/ m)^{\times 2} \cdot D(G,n)}.\]
Let $\delta(t) = 1$ if $t > 1$ and $\delta(t) = 0$ if $t=1$. We will write $\delta = \delta(t)$, so that $d(G) = d+\delta$.

Since $G$ is abelian, \cref{prop:compute-D(Gn)} implies that $D(G,n) = ((\Z/m)^\times)^{e(d+\delta,n)}$. It follows that
\[ B_Q(G,n) \cong \begin{cases} \frac{(\Z/ m)^\times}{\pm (\Z/ m)^{\times 2}} , & \text{if $e(d+\delta,n)$ is even} \\ 
 0, & \text{if $e(d+\delta,n)$ is odd}.	
 \end{cases}
 \]
By \cref{thm:bias-surj-abelian}, we also have that $\beta\colon\HT_{\min}(G,n) \to B(G,n)$ is a bijection. By \cref{prop:beta-to-betaQ}, this implies that
 $\beta_Q\colon  \sM_{2n}(G) \to B_Q(G,n)$ is surjective. 

\begin{proof}[Proof of \cref{thm:final-examples}]
By \cref{prop:stable}, all the manifolds in $\sM_{2n}(G)$ are stably diffeomorphic. 
It therefore suffices to show that, for all $k \ge 2$, there exists $m \ge 2$, $t \ge 1$ and $d \ge 3$ such that $|\sM_{2n}(G)| \ge k$ when $G = (\Z/m)^d \times \Z/t$. 

By \cref{prop:e-s-parity}, there exists $s \ge 3$ such that $e(s,n)$ is even. We now choose either $d=s$ and $t=1$, or $d=s-1$ and any $t >1$ such that $t \mid m$ and $m$, $m/t$ have the same prime factors.
By the above discussion, this implies that 
\[ |\sM_{2n}(G)| \ge \left| \frac{(\Z/ m)^\times}{\pm (\Z/ m)^{\times 2}} \right|. \]

Let $r \ge 1$. Then $|(\Z/2^r)^\times/((\Z/2^r)^\times)^2| = 4$ and, if $p$ is an odd prime, then we have $|(\Z/p^r)^\times/((\Z/p^r)^\times)^2| = 2$. If $m = p_1^{\alpha_1} \cdots p_t^{\alpha_t}$ for distinct primes $p_1, \cdots, p_t$ and $\alpha_i \ge 1$, then
\[ |\sM_{2n}(G)| \ge \frac{1}{2}\left| \frac{(\Z/ m)^\times}{(\Z/ m)^{\times 2}} \right| = \frac{1}{2} \prod_{i=1}^t\left| \frac{(\Z/ p_i^{\alpha_i})^\times}{((\Z/ p_i^{\alpha_i})^\times)^2} \right| \ge 2^{t-1}. \]
Thus, if $m$ has at least $1+\log_2(k)$ distinct prime factors, then $|\sM_{2n}(G)| \ge k$.
\end{proof}

In order to prove \cref{thm:final-examples-eqforms}, we will need the following which is a slight extension of the examples constructed in \cref{thm:e-s-ex}.

\begin{lemma} \label{lemma:e-s-ex}
Let $n \ge 2$, let $d \ge 3$, let $p$ be an odd prime and let $G = (\Z/p)^d$. Let $X$ be a reference minimal $(G,n)$-complex. Then:
\[ \beta(\{\, Y \in \HT_{\min}(G,n) : \pi_n(X) \cong_{\Aut(G)} \pi_n(Y)\}) = \frac{\pm ((\Z/p)^\times)^{\gcd(e(d,n),s(d,n))}}{\pm ((\Z/p)^\times)^{e(d,n)}}  \]
\end{lemma}

\begin{proof}
By \cref{thm:gamma'}, we have that the image is $\pm V^{s(d,n)} \cdot ((\Z/m)^\times)^{e(d,n)} / \pm ((\Z/m)^\times)^{e(d,n)}$ where $V = \{ [r] \in (\Z/p)^\times : r \in \Z, \, (r,|G|)=1 \text{ and } (N,r) \cong \Z G \}$.
It is well known that that $(N,r)$ is a projective $\Z G$-module. Since $G$ is abelian and $\Z G$ has projective cancellation (see \cite[Lemma 3.2]{Linnell:1993}), we have that $V = \ker((\Z/p)^\times \to \wt K_0(\Z G))$ is the kernel of the Swan map. It is now a consequence of a result of Taylor \cite[Theorem 3]{Ta78} (see also \cite[54.15]{CR87}) that, since $G$ is a non-cyclic $p$-group for $p$ odd, we have $|V| = \frac{1}{2}(p-1)$ (see \cite[Theorem 2.1]{Linnell:1993}). Since $V \le (\Z/p)^\times$, we therefore have $V = (\Z/p)^\times$.  The result follows.
\end{proof}

\begin{proof}[Proof of \cref{thm:final-examples-eqforms}]
By \cref{prop:e-s-parity}, there exists $d \ge 3$ be such that $e(d,n)$ is even and $s(d,n)$ is odd. Let $p$ be a prime such that $p \equiv 1 \mod 4$ and let $G = (\Z/p)^d$. Since $p \equiv 1 \mod 4$, $-1$ is a square mod $p$ and so $\pm (\Z/p)^\times = (\Z/p)^\times$. Since $e(d,n)$ is even, we therefore have that $B_Q(G,n) \cong (\Z/p)^\times/(\Z/p)^{\times 2}$ which has order two.
Let $q\colon B(G,n) \twoheadrightarrow B_Q(G,n)$ denote the natural quotient map defined in \cref{prop:beta-to-betaQ}. 

By \cref{lemma:e-s-ex}, we now have that
\[ q(\beta(\{\, Y \in \HT_{\min}(G,n) : \pi_n(X) \cong_{\Aut(G)} \pi_n(Y)\})) = \frac{((\Z/p)^\times)^{\gcd(e(d,n),s(d,n))}}{(\Z/p)^{\times 2}} = \frac{(\Z/p)^\times}{(\Z/p)^{\times 2}}  \]
since the fact that $s(d,n)$ is odd implies that $\gcd(e(d,n),s(d,n)) \equiv 1 \mod 2$. By \cref{prop:beta-to-betaQ}, we have that $q \circ \beta = \beta_Q \circ \mathscr{D}$. Hence there exist $X, Y \in \HT_{\min}(G,n)$ such that $\pi_n(X) \cong_{\Aut(G)} \pi_n(Y)$ and $\beta_Q(M(X)) \ne \beta_Q(M(Y))$, which implies that $M(X) \not \simeq M(Y)$.
By \cref{prop:stable}, $M(X)$ and $M(Y)$ are stably diffeomorphic. 

It remains to show that $M(X)$ and $M(Y)$ have isometric equivariant intersection forms. For simplicity, we first rechoose the identification $\pi_1(Y) \cong G$ so that $L := \pi_n(X) \cong \pi_n(Y)$ are isomorphic as $\Z G$-modules.
By \cref{prop:intdouble} (proven in \cite[Proposition II.2]{Kreck:1984}), we have that there are isometries $S_{M(X)} \cong \Met(L^* \oplus L,\phi_X)$ and $S_{M(Y)} \cong \Met(L^* \oplus L,\phi_Y)$ for some $\phi_X,\phi_Y \in \Sym(L)$ with $\phi_X^G = \phi_Y^G = 0$.
 Since $M(X)$ is stably parallelisable, $\wt{M(X)}$ is spin and so \cref{example:intform} implies that $S_{M(X)}$ is weakly even. Since $|G| = p^d$ is odd, weakly even metabolic forms over $\Z G$ are hyperbolic (see \cref{prop:oddmetabolic}).
Hence $S_{M(X)} \cong S_{M(Y)} \cong H(L)$.
\end{proof}

\subsection{Examples with non-abelian fundamental groups} \label{ss:examples-non-abelian}

We will now establish the following. For simplicity, we will restrict to the case of 4-manifolds and to a small range of fundamental groups. A similar result can be obtained for manifolds of arbitrary dimension $2n \ge 4$, and for a much wider range of fundamental groups. 

\begin{theorem} \label{thm:Q8xCp^3}
Let $G = Q_8 \times (\Z/p)^3$ where $p$ is a prime such that $p \equiv 1 \mod 8$. Then:
\begin{clist}{(i)}
\item
There exist minimal finite $2$-complexes $X$, $Y$ with fundamental group $G$ which are homotopically distinct.
\item
There exist closed smooth $4$-manifolds $M$, $N$ with fundamental group $G$ which are stably diffeomorphic but not homotopy equivalent. More specifically, we can take $M = M(X)$ and $N = M(Y)$ for some minimal finite $2$-complexes $X$, $Y$ with fundamental group $G$.
\end{clist}
\end{theorem}

The proof will be broken into the following sequence of lemmas. In what follows, we will fix identifications $Q_8 = \langle x,y \mid x^2y^{-2},yxy^{-1}x\rangle$ and $(\Z/p)^3 = \langle a,b,c \mid a^p,b^p,c^p,[a,b],[a,c],[b,c]\rangle$.

\begin{lemma} \label{lemma:Q8-presentations}
Let $G = Q_8 \times (\Z/p)^3$ where $p$ is a prime such that $p \equiv 1 \mod 4$. Then:
\begin{clist}{(i)}
\item
For each $r \in \Z$ with $(r,p)=1$, we have that
\[ \mathcal{P}_r = \langle A, B, C \mid A^{2p}B^{-2p}, BAB^{-1}A^{2p-1}, C^p, [A,B^{p-1}], [A,C^r], [B,C]\rangle\]
is a presentation for $G$. Furthermore, we can identify $A=xa$, $B=yb$ and $C=c$.
\item
$G$ satisfies the conditions of \cref{thmx:main-B_Q(G)}. More specifically, $G$ satisfies the strong minimality hypothesis, $H_2(G) \cong (\Z/p)^3$, $m_G = p$ and $X_r := X_{\mathcal{P}_r} \in \HT_{\min}(G)$ for each $r$.
\end{clist}
\end{lemma}

\begin{proof}
(i)
Let $(\Z/p)^2 =\langle a, b \rangle$ and $H = Q_8 \times (\Z/p)^2$. Let $A=xa$, $B=yb$. Then $A^p=x$, $A^{1-p}=a$, $B^p=y$ and $B^{1-p}=b$, and so $H = \langle A, B \rangle$. By combining the given presentations for $Q_8$ and $(\Z/p)^2$ together, and adding commutators, we obtain a presentation
\[ \langle A, B \mid A^{2p}B^{-2p}, B^pA^pB^{-p}A^p, A^{p(p-1)}, B^{p(p-1)}, [A^{p-1},B^{p-1}],[A^{p-1},B^p], [A^p,B^{p-1}]\rangle  \]
where we can omit the relators $[A^p,A^{p-1}]$ and $[B^p,B^{b-1}]$ since they are trivial in $F(A,B)$.

The relators $A^{2p}B^{-2p}$ and $B^pA^pB^{-p}A^p$ imply $A^{4p}$ and $B^{4p}$. Since $4 \mid p-1$, this implies $A^{p(p-1)}$ and $B^{p(p-1)}$, so these two relators can be omitted.

Using $[A^{p-1},B^{p-1}]$, we can replace $[A^{p-1},B^p]$ with $[A^{p-1},B]$ and $[A^p,B^{p-1}]$ with $[A,B^{p-1}]$. Either of these relators then implies $[A^{p-1},B^{p-1}]$ and so it can be omitted. We can also use these relators to replace $B^pA^pB^{-p}A^p$ with $BAB^{-1}A^{2p-1}$.
We now have a presentation:
\[ \langle A, B \mid A^{2p}B^{-2p}, BAB^{-1}A^{2p-1}, [A^{p-1},B], [A,B^{p-1}]\rangle.  \]

We now claim that $[A^{p-1},B]$ is a consequence of the other three relators, and so can be omitted. First  note that $[A,B^{p-1}]=1$ and $BAB^{-1}=A^{1-2p}$ implies that $B^{p-1} = AB^{p-1}A^{-1} = (AB^{-1}A^{-1})^{1-p}=(B^{-1}A^{2p})^{1-p}$. Since $A^{2p}=B^{2p}$ is central, we can rewrite this as $B^{p-1} = B^{p-1}A^{2p(1-p)}$ and so $A^{2p(p-1)}=1$. Using $BAB^{-1}=A^{1-2p}$ and $A^{2p(p-1)}=1$, we now obtain $BA^{p-1}B^{-1} = A^{(1-2p)(p-1)} = A^{p-1} A^{2p(1-p)} = A^{p-1}$, as required.
Hence we have that 
\[ \langle A, B \mid A^{2p}B^{-2p}, BAB^{-1}A^{2p-1}, [A,B^{p-1}]\rangle \]
is a presentation for $H$. By adding a generator $C$ and relators $C^p$, $[A,C]$ and $[B,C]$, we clearly obtain a presentation for $G = H \times \Z/p$. It is straightforward to see that, if $r \in \Z$ with $(r,p)=1$, then replacing $[A,C]$ with $[A,C^r]=1$ does not change the group. Thus $\mathcal{P}_r$ presents $G$.

(ii) By the K\"{u}nneth formula, we have that
\[ H_2(G) \cong H_2(Q_8) \oplus H_2((\Z/p)^3) \oplus (H_1(Q_8) \otimes_{\Z} H_1((\Z/p)^3)) \cong H_2((\Z/p)^3) \cong (\Z/p)^3 \]
by \cref{lemma:H_n(G)} and the fact that $H_2(Q_8)=0$, and $H_1(Q_8) \cong (\Z/2)^2$ and $H_1((\Z/p)^3) \cong (\Z/p)^3$ have coprime orders. Since $\chi(X_r) = 1-3+6=4=1+d(H_2(G))$, this implies both that $G$ has the strong minimality hypothesis and that $X_r$ is a minimal finite $2$-complex. Since $H_2(G) \cong (\Z/p)^3$, it now follows that $m_G = p$.
\end{proof}

In what follows, we will let $X_1$ be the reference minimal finite $2$-complex for $G = Q_8 \times (\Z/p)^3$. To simplify the calculation of $D(G)$, we will now restrict to the case where $p \equiv 1 \mod 8$. We will use the formulation for the bias invariant established in \cref{prop:bias-versions}.

\begin{lemma} \label{lemma:Q8-bias-surj}
Let $G = Q_8 \times (\Z/p)^3$ where $p$ is a prime such that $p \equiv 1 \mod 8$. Then:
\begin{clist}{(i)}
\item
The bias invariant gives a surjection 
\[ \beta\colon\HT_{\min}(G) \to B(G).\] 
Furthermore, $\beta(X_r) = [r]$ for all $r \in \Z$ with $(r,p)=1$.
\item
$D(G) = (\Z/p)^{\times 2}/\{\pm 1\}$ and $B(G) = (\Z/p)^\times/\pm(\Z/p)^{\times 2} \cong \Z/2$.
\end{clist}
\end{lemma}

\begin{proof}
(i) Using Fox differentiation (see \cite[Section 3.2]{Si93}), we get that
\[ C_*(\wt X_r) = \Bigl (\Z G^6 
\xrightarrow[d_2(\wt X_r)]{
\cdot \left(\begin{smallmatrix} 
\Sigma_{2p}(xa) & -\Sigma_{2p}(xa) & 0 \\
yb+x^{-1}a\Sigma_{2p-1}(xa) & 1-x^{-1}a & 0 \\
0 & 0 & \Sigma_p(c) \\
1-b^{-1} & (xa-1)\Sigma_{p-1}(yb) & 0 \\
1-c^r & 0 & (xa-1)\Sigma_r(c) \\
0 & 1-c & yb-1
\end{smallmatrix}\right)
} 
\Z G^3 \xrightarrow[d_1(\wt X_r)]{\cdot \left(\begin{smallmatrix} xa-1 \\ yb-1 \\ c-1 \end{smallmatrix}\right)} \Z G \Bigr ) \]
where, if $u \in \Z G$ and $m \ge 1$, we write $\Sigma_m(u) := 1+u+\cdots+u^{m-1}$.
Note that the only terms involving $r$ are the entries $1-c^r$ and $(xa-1)\Sigma_r(c)$ in the $5$th row of $d_2(\wt X_r)$.

There is a chain map $f=(f_2,f_1,f_0) : C_*(\wt X_r) \to C_*(\wt X_1)$ where $f_0 = \id_{\Z G}$, $f_1 = \id_{\Z G^3}$ and 
\[ f_2 = \cdot \left(\begin{smallmatrix} 
1 & 0 & 0 & 0 & 0 & 0 \\
0 & 1 & 0 & 0 & 0 & 0 \\
0 & 0 & 1 & 0 & 0 & 0 \\
0 & 0 & 0 & 1 & 0 & 0 \\
0 & 0 & 0 & 0 & \Sigma_r(c) & 0 \\
0 & 0 & 0 & 0 & 0 & 1
\end{smallmatrix}\right) : \Z G^6 \to \Z G^6. \]

Let $\varepsilon\colon\Z G \to \Z$ denote the augmentation map. If $d \in M_m(\Z G)$ is a matrix, $\varepsilon(d) \in M_m(\Z)$ will denote the matrix $d$ with $\varepsilon$ applied to each entry.
We have 
\[ H_2(X_r) \cong \ker(\varepsilon(d_2(\wt X_r)):\Z^6 \to \Z^3) = \ker \left\{ \cdot \left(\begin{smallmatrix} 
2p & -2p & 0 \\
2p & 0 & 0 \\
0 & 0 & p \\
0 & 0 & 0 \\
0 & 0 & 0 \\
0 & 0 & 0
\end{smallmatrix}\right) \right\} = \left\langle
\left(\begin{smallmatrix}
	 0 \\ 0 \\ 0 \\ 1 \\ 0 \\ 0
\end{smallmatrix} \right),
\left(\begin{smallmatrix}
	 0 \\ 0 \\ 0 \\ 0 \\ 1 \\ 0
\end{smallmatrix} \right),
\left(\begin{smallmatrix}
	 0 \\ 0 \\ 0 \\ 0 \\ 0 \\ 1
\end{smallmatrix} \right)
\right\rangle \cong \Z^3.
\] 
With respect to this identification, the induced map $(f_2)_*\colon H_2(X_r) \to H_2(X_r)$ is given by $(f_2)_* = \left(\begin{smallmatrix}
	1 & 0 & 0 \\ 0 & r & 0 \\ 0 & 0 & 1
\end{smallmatrix}\right)$.
Hence we have $\beta(X_r) = \det((f_2)_*) = [r] \in (\Z/p)^\times/\pm D(G)$.

(ii) We have $D(G) = \IM(\varphi_G\colon\Aut(G) \to (\Z/p)^\times/\{\pm 1\})$. Let $q\colon(\Z/p)^\times/\{\pm 1\} \twoheadrightarrow (\Z/p)^\times/\pm(\Z/p)^{\times 2}$ denote the natural quotient map.
Since $p \equiv 1 \mod 8$, we can write $p = 1+8m$ for some $m$. This implies that $(\Z/p)^\times/\{\pm 1\} \cong \Z/4m$ and so we can view $\varphi_G$ as a map $\varphi_G\colon\Aut(G) \to \Z/4m$. Since $p \equiv 1 \mod 4$, $-1 \in (\Z/p)^{\times 2}$ and so $(\Z/p)^\times/\pm(\Z/p)^{\times 2} \cong \Z/2$. Thus we can view $q$ as a map $q\colon\Z/4m \twoheadrightarrow \Z/2$. This is reduction mod 2 since there is a unique surjection.

Since $Q_8$ and $(\Z/p)^3$ have coprime order, we have $\Aut(G) \cong \Aut(Q_8) \times \Aut((\Z/p)^3)$. We will now deal with the image of $\Aut(Q_8)$ and $\Aut((\Z/p)^3)$ under $\varphi_G$ separately.

By \cite[Lemma IV.6.9]{AM04}, $\Aut(Q_8) \cong S_4$. By the fact that $(S_4)^{\text{ab}} \cong \Z/2$ and $\Z/4m$ is abelian, we have that $\varphi_G(\Aut(Q_8)) \cong 0$ or $\Z/2$, and so $\varphi_G(\Aut(Q_8)) \subseteq \{0,2m\}$. Since $q$ is reduction mod 2, this implies that $(q \circ \varphi_G)(\Aut(Q_8)) = \{0\}$ and so $\varphi_G(\Aut(Q_8)) \le (\Z/p)^{\times 2}/\{\pm 1\}$.

We have $\Aut((\Z/p)^3) \cong GL_3(\Z/p)$. Since $\Z/p$ is a field, it follows from a 1901 theorem of Dickson \cite{Di01} that $SL_3(\Z/p)$ is the commutator subgroup of $GL_3(\Z/p)$, and so there is an isomorphism $GL_3(\Z/p)^{\text{ab}} \cong (\Z/p)^\times$ induced by the determinant map. 
Let $\theta \in \Aut((\Z/p)^3)$ be given by $a \mapsto a, b \mapsto b, c \mapsto c^r$ for some generator $r \in (\Z/p)^\times$. Then $\det(\theta) = r$ and so $\theta$ generates the abelianisation. Since $\Z/4m$ is abelian, $\varphi_G \mid_{\Aut((\Z/p)^3)}\colon\Aut((\Z/p)^3) \to \Z/4m$ factors through the abelianisation and so $\varphi_G(\Aut((\Z/p)^3)) = \langle \varphi_G(\theta) \rangle$ is generated by $\varphi_G(\theta)$ as a subgroup of $\Z/4m$.

Now, using the isomorphisms $\Z G_{\theta^{-1}} \cong \Z G$, we obtain a chain isomorphism
\[ C_*(\wt X_1)_{\theta^{-1}} \cong \Bigl (\Z G^6 
\xrightarrow[\theta_*(d_2(\wt X_1))]{
\cdot \left(\begin{smallmatrix} 
\Sigma_{2p}(xa) & -\Sigma_{2p}(xa) & 0 \\
yb+x^{-1}a\Sigma_{2p-1}(xa) & 1-x^{-1}a & 0 \\
0 & 0 & \Sigma_p(c^r) \\
1-b^{-1} & (xa-1)\Sigma_{p-1}(yb) & 0 \\
1-c^r & 0 & xa-1 \\
0 & 1-c^r & yb-1
\end{smallmatrix}\right)
} 
\Z G^3 \xrightarrow[\theta_*(d_1(\wt X_1))]{\cdot \left(\begin{smallmatrix} xa-1 \\ yb-1 \\ c^r-1 \end{smallmatrix}\right)} \Z G \Bigr ) \]
where, for $i =1,2$, $\theta_*(d_i(\wt X_1))$ denotes the matrix $d_i(\wt X_1)$ with the induced map $\theta_*\colon\Z G \to \Z G$ applied to each entry.
Since $r \in (\Z/p)^\times$ and $c$ has order $p$, we have $\Sigma_p(c^r) = \Sigma_p(c)$. In particular, the top four rows of $C_*(\wt X_1)_{\theta^{-1}}$ and $C_*(\wt X_1)$ coincide.

There is a chain map $g = (g_2,g_1,g_0) : C_*(\wt X_1)_{\theta^{-1}} \to C_*(\wt X_1)$ where $g_0 = \id_{\Z G}$ and we have 
\[ g_1 = \cdot \left(\begin{smallmatrix}
	1 & 0 & 0 \\
	0 & 1 & 0 \\
	0 & 0 & \Sigma_r(c)
\end{smallmatrix}\right) : \Z G^3 \to \Z G^3, \quad 
g_2 = \cdot \left(\begin{smallmatrix}
1 & 0 & 0 & 0 & 0 & 0 \\
0 & 1 & 0 & 0 & 0 & 0 \\
0 & 0 & \Sigma_r(c) & 0 & 0 & 0 \\
0 & 0 & 0 & 1 & 0 & 0 \\
0 & 0 & 0 & 0 & \Sigma_r(c) & 0 \\
0 & 0 & 0 & 0 & 0 & \Sigma_r(c)
\end{smallmatrix}\right)  : \Z G^6 \to \Z G^6.  \]

As in (i), we have that $H_2(X_r) \cong \Z^3$ and $H_2(\Z \otimes_{\Z G} C_*(\wt X_1)_{\theta^{-1}}) \cong \Z^3$ are generated by the bottom three copies of $\Z$. With respect to these identifications, the induced map $(g_2)_* : H_2(\Z \otimes_{\Z G} C_*(\wt X_1)_{\theta^{-1}}) \to H_2(X_r)$ is given by $(g_2)_* = \left(\begin{smallmatrix}
	1 & 0 & 0 \\ 0 & r & 0 \\ 0 & 0 & r
\end{smallmatrix}\right)$. This gives that
\[ \varphi_G(\theta) = \beta(C_*(\wt X_1),C_*(\wt X_1)_{\theta}) = \beta(C_*(\wt X_1)_{\theta^{-1}},C_*(\wt X_1)) = \det((g_2)_*) = [r^2] \in (\Z/p)^\times/\{\pm 1\}. \]

Hence $\varphi_G(\Aut((\Z/p)^3)) = \langle \varphi_G(\theta) \rangle = \langle r^2 \rangle = (\Z/p)^{\times 2}/\{\pm 1\}$, since $r \in (\Z/p)^\times$ was chosen to be a generator. We established previously that $\varphi_G(\Aut(Q_8)) \le (\Z/p)^{\times 2}/\{\pm 1\}$, and thus we have $D(G) = \IM(\varphi_G) = (\Z/p)^{\times 2}/\{\pm 1\}$ and $B(G) = (\Z/p)^\times/\pm(\Z/p)^{\times 2} \cong \Z/2$.
\end{proof}

\begin{proof}[Proof of \cref{thm:Q8xCp^3}]
To prove (i), note that \cref{lemma:Q8-bias-surj} (i) implies that there is a surjection 
\[ \beta\colon\HT_{\min}(G) \twoheadrightarrow B(G) \cong \Z/2.\] 
This implies that $|\HT_{\min}(G)| > 1$, as required.

To prove (ii), note that \cref{lemma:Q8-presentations} (ii) implies that $G$ satisfies the conditions of \cref{thmx:main-B_Q(G)}. Since $m_G = p$, we have 
\[ B_Q(G) \cong \frac{(\Z/p)^\times}{\pm(\Z/p)^{\times 2} \cdot D(G)}. \] 
It follows from \cref{lemma:Q8-bias-surj} (ii) that $D(G) = (\Z/p)^\times)^2/\{\pm 1\}$, and so we obtain
\[ B_Q(G) \cong (\Z/p)^\times / \pm(\Z/p)^{\times 2} \cong \Z/2. \]
Since $\beta\colon\HT_{\min}(G) \twoheadrightarrow B(G)$ is surjective, \cref{prop:beta-to-betaQ} implies that 
\[ \beta_Q\colon\sM_{4}(G) \twoheadrightarrow B_Q(G) \cong \Z/2\] 
is surjective, i.e. since $\beta_Q(M(X)) = q(\beta(X))$ is the image of the bias under a surjection $q$.
This implies that $|\sM_{4}(G)| \ge 2$, which completes the proof.
\end{proof}

\appendix

\section{Unitary groups and odd-dimensional $L$-theory}
\label{ss:L-odd}

\renewcommand{\thesubsection}{\thesection(\roman{subsection})}
\renewcommand{\thetheorem}{\thesection\arabic{theorem}}

Let $\l$ be a ring with involution. For $n \in \Z$, we will define the
surgery obstruction groups $L_{2n+1}^h(\l)$ and $L_{2n+1}^s(\l)$ for $\La$ a ring with involution. 
These are abelian groups which only depend on the value of $n \mod 2$. In particular, it suffices to define the groups in the cases $2n+1 \equiv  1, 3 \mod 4$. The original definition for $\La = \ZG$ arose from analysing the obstructions to geometric surgery problems with fundamental group $G$
(see \cite[Chapter 6]{Wall:1970}).

We will need to recall some of the details  of the definition of   $L^h_{2n+1}(\La)$ for use in \cref{ss:PB-to-U}.  For $r \ge 1$, let $M_r(\l)$ denote the ring of $r \times r$ matrices over $\l$.  If $a \mapsto \ol{a}$ denotes the involution on $\La$, then  $\alpha^* = (\wbar{\alpha}_{ji}) $ denotes the conjugate transpose matrix, for any $\alpha =(\alpha_{ij}) \in M_{r}(\La)$. If
\[\sigma = \left(\begin{smallmatrix} \alpha & \beta \\ \gamma & \delta \end{smallmatrix}\right) \]
is a $2r\times 2r$ matrix in $GL_{2r}(\La)$ expressed in $r\times r$ blocks, then
$\U_{2r}(\La) \subseteq GL_{2r}(\La)$ is the subgroup consisting of the matrices $\sigma$ with the properties
\begin{enumerate}
\item $\alpha \delta^* + (-1)^n\beta \gamma^* = I$
\item $\alpha\beta^*$ and $\gamma \delta^*$ each have the form 
$\theta - (-1)^n\theta^*$ for some  $\theta \in M_r(\l)$.
\end{enumerate}

\begin{remark} \label{remark:unitary-definitions} 
The unitary group $\U_{2r}(\La)$
  is a subgroup of the \emph{hermitian} unitary group
$ \Isom_\ep(H(\La^r))$, which consists of the matrices $\sigma \in GL_{2r}(\l)$ which preserve the $\varepsilon = (-1)^n$-hyperbolic form $H(\l^r)$ but not necessarily the quadratic refinement. In  other words, the relation
\[ \sigma \left(\begin{smallmatrix} 0 & I \\ \ep I & 0 \end{smallmatrix}\right) \sigma^* = \left(\begin{smallmatrix} 0 & I \\ \ep I & 0 \end{smallmatrix}\right)\]
holds in $GL_{2r}(\l)$, and we have $\U_{2r}(\l)\subseteq \Isom_\ep(H(\La^r))$. If $\sigma$ is given as above, the weaker assumption is equivalent to conditions (i) and 
\begin{enumerate}
	\item[(ii)$'$] $(\alpha \beta^*)^* + (-1)^{n}\alpha \beta^*=0$ and $(\gamma \delta^*)^*+ (-1)^{n}\gamma \delta^*=0$.
\end{enumerate}
Condition (ii)$'$ is implied by (ii) but is a strictly weaker condition over an arbitrary ring with involution $\La$.
\end{remark}

\begin{lemma}
\label{example:unitary}
Let $\l$ be a ring with involution , let $\varepsilon = (-1)^n$ for some $n$ and let $r \ge 1$. Then $\U_{2r}(\l)=\Isom(H_\ep(\La^r))$ provided $\wh H^{n+1}(\Z/2;(\l,+))=0$, where the additive group $(\l,+)$ is viewed as a $\Z/2$-module under the involution.
Furthermore: 
\begin{clist}{(a)}
\item For $n$ even, this holds for $\Lambda = \Z$ or $\Lambda = \Z/m$ for $m$ odd.
\item For $n$ odd, this holds when $2 \in \Lambda^\times$, hence for $\Lambda = \Z/m$ for $m$ odd (but not for $\Lambda = \Z$).
\end{clist}
\end{lemma}

\begin{proof} Let $\l$ be a ring with involution such that $\wh H^{n+1}(\Z/2;(\La,+))=0$. 
We note that  
any matrix $X \in M_r(\La)$ such that $X +  \ep X^* = 0$ can be expressed as $X = D + V - \ep V^*$, where $D$ is diagonal with $D + \ep D^* = 0$ and $V$ is strictly upper triangular. The condition 
$\wh H^{n+1}(\Z/2;(\La,+))=0$ applied to each diagonal entry of $D$ shows that $D = E - \ep E^*$.
It follows that condition (ii)$'$ implies condition (ii) for all $\alpha, \beta, \gamma, \delta \in M_r(\l)$, and hence $\U_{2r}(\l)=\Isom(H_\ep(\La^r))$.
 For the examples (a) and (b), we must show that every $x \in \l$ such that $x+(-1)^n\bar{x}=0$ is of the form $x = y - (-1)^n\bar{y}$ for some $y \in \l$. This is straightforward to check in each case. For example, if $\l = \Z$ and $n$ is even, then the statement is simply that $2x=0$ implies $x=0$ for all $x \in \Z$.
\end{proof}

\subsection{The surgery obstruction groups  $L^h_{2n+1}(\La)$} We will now define the group $L^h_{2n+1}(\La)$ as the quotient of the stabilised unitary group by a certain subgroup generated by elementary unitary matrices.
\begin{definition}\label{def:elemU}
A $2r\times 2r$ matrix in $\U_{2r}(\l)$ is called \emph{elementary unitary} if it is a product of matrices
of the following forms  (see the list in \cite[\S 1]{Lee:1971} for the case $n \equiv 1 \mod 2$):

\begin{enumerate}
\item $\left(\begin{smallmatrix} Q & 0 \\ 0 & (Q^*)^{-1} \end{smallmatrix}\right)$ where $Q \in GL_r(\l)$
\item $\left(\begin{smallmatrix} I & P \\ 0 & I \end{smallmatrix}\right)$ and $\left(\begin{smallmatrix} I & 0 \\ P & I \end{smallmatrix}\right)$ where $P = A - (-1)^nA^*$ for some $A \in M_r(\l)$
\item $\left(\begin{smallmatrix} A & B \\ C & D \end{smallmatrix}\right)$ where 
$A=\left(\begin{smallmatrix} 0 & 0 \\ 0 & I \end{smallmatrix}\right)$, 
$B=\left(\begin{smallmatrix} 1 & 0 \\ 0 & 0 \end{smallmatrix}\right)$, 
$C=\left(\begin{smallmatrix} (-1)^n & 0 \\ 0 & 0 \end{smallmatrix}\right)$ and 
$D=\left(\begin{smallmatrix} 0 & 0 \\ 0 & I \end{smallmatrix}\right)$.
\end{enumerate}
The subgroup in $\U_{2d}(\La)$ of elementary unitary matrices is denoted $\EU_{2r}(\La)$.
\end{definition}

 After orthogonal stabilisation, one defines the inclusions
$$ \U_{2}(\La) \subset \U_{4}(\La) \subset \dots \subset \U_{2r}(\La) \subset \U_{2r+2}(\La) \subset \cdots$$
whose union is the \emph{stable unitary group} $\U(\La)$. 
 More precisely, the inclusion map is given by
\[ i \colon U_{2r}(\La) \to U_{2r+2}(\La), \quad \left(\begin{smallmatrix} \alpha & \beta \\ \gamma & \delta \end{smallmatrix}\right) \mapsto \left(\begin{smallmatrix} \alpha' & \beta' \\ \gamma' & \delta' \end{smallmatrix}\right) \]
for $\alpha, \beta, \gamma, \delta \in M_r(\l)$, and where $\alpha' = \left(\begin{smallmatrix} \alpha & 0 \\ 0 & 1 \end{smallmatrix}\right)$, $\beta' = \left(\begin{smallmatrix} \beta & 0 \\ 0 & 0 \end{smallmatrix}\right)$, $\gamma' = \left(\begin{smallmatrix} \gamma & 0 \\ 0 & 0 \end{smallmatrix}\right)$, $\delta' = \left(\begin{smallmatrix} \delta & 0 \\ 0 & 1 \end{smallmatrix}\right) \in M_{r+1}(\l)$.

\smallskip
The union of the corresponding stabilisations 
$\EU_{2r}(\La)\subset \EU_{2r+2}(\La)$ is the subgroup $\EU(\La) \subset \U(\La)$ generated by all elementary unitary matrices. A key result is that $\EU(\La)$ contains the commutator subgroup of $\U(\La)$ (see Wall \cite[Chapter 6]{Wall:1970}, \cite[Theorem 4.2]{Ranicki:1973}). The abelian quotient group
$$L^h_{2n+1}(\La) = \U(\La) / \EU(\La)$$
provides an algebraic description of the (unbased) surgery obstruction group.

\subsection{The surgery obstruction groups $L^s_{2n+1}(\La)$}
To define the \emph{simple} obstruction groups $L^s_{2n+1}(\La)$, one must take Whitehead torsion into account and work with based free modules (see \cite{Milnor:1966}). For a general ring $\La$ with involution,
we have a natural group homomorphism by composition
\[ \tau\colon \U_{2d}(\La) \hookrightarrow GL_{2d}(\l) \to K_1(\La),  \]
and we define the \emph{special unitary group} $\SU_{2r}(\La) = \ker \tau$. 
This is the group of simple
automorphisms of the $(-1)^n$-hyperbolic  form 
\[ H(\La^r) = \Bigl (\La^r \oplus \La^r, \bigl (\begin{smallmatrix} 0 & I \\ (-1)^nI & 0 \end{smallmatrix}\bigr) \Bigr ) \]
on a based free module, which preserve its quadratic refinement and the preferred class of bases.

 The simple surgery obstruction groups are defined as above by a certain quotient group after stabilisation.  In the notation of 
Wall \cite[Chap.~6]{Wall:1970}, the group $\SU(\La)$ is the limit of the
 automorphism groups $\SU_{2r}(\La)$, where  $\varepsilon = (-1)^n$,
under the natural inclusions
$$ \SU_{2}(\La) \subset \SU_{4}(\La) \subset \dots \subset \SU_{2r}(\La) \subset \SU_{2r+2}(\La) \subset \cdots $$

\begin{definition}\label{def:elemSU}
We say that a $2r\times 2r$ matrix in $\SU_{2r}(\La)$ is an \textit{elementary special unitary} matrix if it is the product of the elementary unitary matrices in \cref{def:elemU}, provided that the elements of the form 
\[ \left(\begin{smallmatrix} Q & 0 \\ 0 & (Q^*)^{-1} \end{smallmatrix}\right) 
\text{\ where $Q \in GL_r(\l)$ \ }\]
satisfy  the additional condition $\tau(Q) = 0 \in K_1(\La)$, so that 
$$Q \in SL_r(\l) = \ker(GL_r(\l) \to K_1(\l)).$$
The subgroup in $\SU_{2d}(\La)$  of elementary special unitary matrices is denoted $\RU_{2r}(\l)$.   
\end{definition}

The corresponding union  $\RU(\La) \subset \SU(\La)$ again contains the commutator subgroup (see \cite[Chapter 6]{Wall:1970}, \cite[Theorem 4.2]{Ranicki:1973}), and the abelian quotient group
$$L^s_{2n+1}(\La) = \SU(\La)/\RU(\La)$$
provides an algebraic description of the (based) surgery obstruction groups.

There are exact sequences relating based and unbased obstruction groups (usually called the Ranicki-Rothenberg sequences \cite[Theorem 5.7]{Ranicki:1973}):
\eqncount
\begin{equation}\label{eq:RRcomp}
\dots \to \wH^{0}(\cy 2; \wK_1(\La)) \xrightarrow[]{\bd} L^s_{2n+1}(\La)\xrightarrow[]{i_*} L^h_{2n+1}(\La) \xrightarrow{\tau_*} \wH^{1}(\cy 2; \wK_1(\La))\to  \dots
\end{equation}
where the relative terms are given by the Tate cohomology groups:
$$\wH^n(\cy 2; \wK_1(\La)) = \frac{\{A \in \wK_1(\l) \, | \, A^* = (-1)^n A\}}
{\{A+(-1)^n A^* \, | \, A \in \wK_1(\l)\}}.$$
 The map  $\bd([A]) = \left[\left(\begin{smallmatrix} A & 0 \\ 0 & (A^*)^{-1} \end{smallmatrix}\right)\right]$ for $A \in GL_r(\l)$ is often referred to as the \textit{hyperbolic map} and the map $i_*$ is induced by the inclusion $i\colon  \SU(\l) \hookrightarrow \U(\l)$ which satisfies $i(\RU(\l)) \subseteq \EU(\l)$. The map $\tau_*\colon  L^h_{2n+1}(\La) \to \wH^{1}(\cy 2; \wK_1(\La))$ is induced by $\tau\colon \U(\La) \to K_1(\La)$.
For more details see Wall \cite[Chapter 6]{Wall:1970}, 
Ranicki \cite[Theorem 4.2, Theorem 5.6]{Ranicki:1973}, 
Lee \cite[\S 1]{Lee:1971}, or Lees \cite[\S 6]{Lees:1973}.

\subsection{Round $L$-groups} 
In \cref{ss:L-theory-computations}
we will use the closely related functors
 $L_{2n+1}^X(\l)$, with torsion $X \subseteq K_1(\La)$ in any involution-invariant subgroup. These $L$-groups are more suitable for computations since they respect products of rings with involution.  The case $X = \{ 0\}$ is denoted $L_*^S(\La)$, and  $L_*^K(\La)$ is defined by  $X = K_1(\La)$. If $X \subseteq Y$ are involution-invariant subgroups of $K_1(\La)$, there is a similar Ranicki-Rothenberg sequence relating them to Tate cohomology $\wH^*(\cy 2; K_1(\La))$ (see  \cite[1.1]{Wall:1976}, \cite{Hambleton:1987}, or  \cite[\S 3]{Hambleton:2000}).

The relation between round $L$-groups (corresponding to subgroups $X \subseteq K_1(\La)$ with $\{\pm 1\} \in X$) and the usual $L$-groups (corresponding to their quotients  $\wX \subseteq K_1(\La) /\{ \pm 1\}$) is given by the isomorphism $ L_{2n}^X(\La) \cong L_{2n}^{\wX}(\La)$, and an exact sequence
\eqncount
\begin{equation}\label{eq:round_comp}
 0 \to \cy 2 \to L_{2n+1}^X(\La) \to L_{2n+1}^{\wX}(\La) \to 0
 \end{equation}
obtained by dividing out a single $\cy 2$ (see \cite[Proposition 3.2]{Hambleton:1987}). 

To compare the  $L^S_*(\La) \to L^{\{\pm 1\}}_*(\La) \to  L^s_*(\La)$  is a two-step process. The following braid diagram will be used in the proof of \cref{prop:except} and combines the exact sequences \eqref{eq:round_comp} and \eqref{eq:RRcomp}:
\eqncount
\begin{equation}\label{eq:braid}
\begin{matrix}
\xymatrix@R=10pt{
0    \ar[dr] \ar@/^2pc/[rr]   &&
\cy 2  \ar[dr] \ar@/^2pc/[rr]^{\beta} &&
\{\pm 1\} \ar[dr]  \ar@/^2pc/[rr]&& 0\\
&\ker\gamma \ar[dr] \ar[ur]  &&
L^{\{\pm 1\}}_{2n+1}(\La) \ar@{->>}[dr]\ar[ur] && \coker\gamma\ar[ur]  \ar[dr]&\\
\{\pm 1\}\ar[ur] \ar@/_2pc/[rr] && L^S_{2n+1}(\La)
  \ar[ur]^{\alpha} \ar@/_2pc/[rr]^\gamma&&
L^s_{2n+1}(\La) \ar[ur] \ar@/_2pc/[rr]&& 0
}
\end{matrix}
\end{equation}
\vskip .4cm
\begin{remark}  For integral group rings, the round $L$-group associated to $L^s_*(\ZG)$ is $L^X_*(\ZG)$, where $X = \{\pm g \vv g \in G\} \subset K_1(\ZG)$. In particular, $L^K_*(\bZ) = L^{(\pm1)}_*(\bZ)$ is the round version of $L^s_*(\bZ) = L^h_*(\bZ)$, which are the surgery obstruction groups for the trivial group $G=1$ \cite[1.4]{Wall:1976}. 
 \end{remark}

\section{Computations of odd-dimensional $L$-theory of abelian groups}
\label{ss:L-theory-computations}
In order to prove \cref{thm:BQ-main}, we need to justify the $L$-group computations used in the proof of
\cref{prop:hom-to-Lgroups} by computing the boundary maps
$$\IM(\bd \colon \wH^0(\cy 2; \wK_1(\La)) \to L^s_{2n +1}(\La)) = \ker(L_{2n+1}^s(\l) \to L_{2n+1}^h(\l))$$
in the Ranicki-Rothenberg sequences for $\La = \bZ$ and $\La = \cy m$. These computations use the more directly computable round $L$-groups and the comparison sequences \eqref{eq:RRcomp} and \eqref{eq:round_comp}.

Recall  from \cref{ss:PB-to-U}
that $D_{2d}(\cy m)$ is a subgroup of $\SU_{2d}(\cy m) \subset \SU(\cy m)$ . We  let 
$$q\colon D_{2d}(\cy m) \cong (\cy m)^\times \to (\cy m)^\times/ \pm(\cy m)^{\times 2}$$
denote the natural surjection induced by     reduction modulo squares and $\la -1\ra$.
\begin{proposition} \label{prop:Lgroups-main} Let $n \in \Z$ be even, $d \geq 1$, and $m \ge 1$. Then there is a commutative diagram
\[
\begin{tikzcd}
	D_{2d}(\cy m) \ar[r,hookrightarrow] \ar[d,twoheadrightarrow] & SU_{2d}(\cy m) \ar[r,hookrightarrow] & SU(\cy m) \ar[d,twoheadrightarrow] \\
	 {(\Z/ m)^\times}/{\pm(\Z/ m)^{\times 2}} \ar[rr,"\cong"] & & L_1^s(\cy m)
\end{tikzcd}
\]
where the lower isomorphism is induced by the hyperbolic map.
\end{proposition}

The proof will follow by first analysing $\ker(L_{2n+1}^S(\l) \to L_{2n+1}^K(\l))$ for $\l = \Z$ and $\Z/m$.

\begin{lemma} \label{lemma:L_1(Z)}
Let $n \in \Z$. Then the hyperbolic map $\bd$ induces an isomorphism
 \[ \ker(L_{2n+1}^S(\Z) \to L_{2n+1}^K(\Z)) \cong  \Z/2 \cdot \left[ \left(\begin{smallmatrix} -1 &\sx 0 \\ \sx 0 & -1 \end{smallmatrix}\right) \right]. \]
\end{lemma}

\begin{proof}
This is a consequence of a result of Wall \cite[Proposition 1.4.1]{Wall:1976} (see also \cite[Proposition III.10 (I)]{Kreck:1984}) that there are isomorphisms
\[ L_1^S(\Z) \cong \Z/2 \cdot \left[ \left(\begin{smallmatrix} -1 &\sx 0 \\ \sx 0 & -1 \end{smallmatrix}\right) \right]\quad \text{and} \quad L_3^S(\Z) = \cy 4 \cdot [\tau]\]
where $\tau = \left[ \left(\begin{smallmatrix} \sx 0 & 1 \\ -1 & 0 \end{smallmatrix}\right) \right ]$ and $[\tau^2] =   \left [\left(\begin{smallmatrix} -1 & \sx 0 \\ \sx 0 & -1 \end{smallmatrix}\right)\right ] =
 \IM(\bd \colon \wh H^0(\cy 2;\Z^\times) \to L_3^S(\Z))$.
\end{proof}

\begin{lemma} \label{lemma:L_1(Z/m)}
Let $m \ge 1$. Then the hyperbolic map $\bd$ induces isomorphisms:
\begin{clist}{(i)}
\item $\ker(L_1^S(\Z/m) \to L_1^K(\Z/m)) \cong \Z/2\cdot \left\{ \left[ \left(\begin{smallmatrix} r_i & 0 \\ 0 & r_i^{-1} \end{smallmatrix}\right) \right] : 1 \le i \le t \right\} \cong  (\Z/m)^\times/(\cy m)^{\times 2}$
where the $r_i \in (\Z/m)^\times$ are coset representatives for $(\Z/m)^\times/(\cy m)^{\times 2}$.
\item $\ker(L_3^S(\Z/m) \to L_3^K(\Z/m)) \cong 
\begin{cases} \Z/2 \cdot \left[ \left(\begin{smallmatrix} -1 &\sx 0 \\ \sx 0 & -1 \end{smallmatrix}\right) \right], & \text{if $8 \mid m$} \\
0, & \text{otherwise}. \end{cases}$
\end{clist}
\end{lemma}

\begin{proof} 
Suppose that $m = m_1 m_2$ for $m_1, m_2 \ge 2$ coprime. Since the $L^S$ and $L^K$ groups respect products of rings with involution \cite[\S 3]{Hambleton:2000}, the isomorphism of rings $\Z/m \cong \Z/m_1 \times \Z/m_2$ induces a commutative diagram
\[
\hspace{15mm}
\begin{tikzcd}
	(\Z/m)^\times/(\cy m)^{\times 2} \ar[d,"\bd^{(m)}"] \ar[r,"\cong"] & (\Z/m_1)^\times/(\Z/m_1)^{\times 2} \oplus (\Z/m_1)^\times/(\Z/m_1)^{\times 2} \ar[d,"\text{$(\bd^{(m_1)},\bd^{(m_2)})$}"] \\
	L_r^S(\Z/m) \ar[r,"\cong"] \ar[d] & L_r^S(\Z/m_1) \oplus L_r^S(\Z/m_2) \ar[d] \\
	L_r^K(\Z/m) \ar[r,"\cong"] & L_r^K(\Z/m_1) \oplus L_i^K(\Z/m_2)
\end{tikzcd}
\]
 for $r \equiv 0,1,2,3 \mod 4$, where $\bd^{(m)}$, $\bd^{(m_1)}$ and $\bd^{(m_2)}$ denote the hyperbolic maps in the respective cases. We may therefore assume that $m = p^k$ is a prime power. 
\begin{clist}{(i)}
\item A starting point is $k=1$, where Wall \cite[\S 1.2]{Wall:1976} shows that when $p$ is an odd prime $L_r^S(\cy p)  = 0, \cy 2, \cy 2, 0 $, for $r \equiv 0,1,2,3 \mod 4$,  $L_1^K(\cy p) = L_3^K(\cy p) =0$, and  $L_0^K(\cy p) = L_2^K(\cy p) =\cy 2$. For $p=2$, we have $L_*^S(\cy 2) = L_*^K(\cy 2) = \cy 2$ in each dimension.

\item  Since $\wh H^0(\cy 2;(\Z/m)^\times) \cong (\Z/m)^\times/(\cy m)^{\times 2}$, it suffices to show that the hyperbolic map $\bd \colon  \wh H^0(\cy 2;(\Z/m)^\times) \to L_1^S(\Z/m)$ is injective. By exactness, this would then imply that $\wh H^0(\cy 2;(\Z/m)^\times) \cong \IM(\bd) = \ker(L_1^S(\Z/m) \to L_1^K(\Z/m))$. This holds since the previous map in the Rothenberg sequence is zero: $L^K_2(\cy m) \cong \cy 2$ generated by a rank two skew-hermitian form of Arf invariant 1 with hyperbolic bilinearisation, and $\det \left(\begin{smallmatrix} \sx 0 &1 \\  -1 & 0 \end{smallmatrix}\right) = 1$.

\item The kernel of the hyperbolic map $\bd \colon  \wh H^0(\cy 2;(\Z/m)^\times) \to L_3^S(\Z/m)$ is the image of the norm map $L_0^K(\cy m) \to \wh H^0(\cy 2;(\Z/m)^\times)$. If $m = p^k$ is a prime power, then the reduction map $\Z/p^k \to \Z/p$ induces isomorphisms
 $L_*^K(\Z/p^k) \cong L_*^K(\Z/p)$ \cite[\S 1.2]{Wall:1976} and $\wh H^r(\cy 2;(\Z/p^k)^\times) \cong \wh H^r(\cy 2;(\Z/p)^\times)$ for any $r \in \Z$. Since $L_0^K(\cy p) = \cy 2$ for $p$ an odd prime, the norm map is surjective and the  image of the hyperbolic map is zero in this case by naturality of the Ranicki-Rothenberg sequence.
 
 \item If $m = 2^k$, the map $L_0^K(\cy {2^k}) \to \wh H^0(\cy 2;(\cy{2^k})^\times)$ is zero for $k=1$ (trivially), but injective for $k \geq 2$. In these cases, $L_0^K(\cy {2^k}) = \cy 2$ is generated by the hermitian form $h :=\left(\begin{smallmatrix} 2 &1 \\  1 & 2 \end{smallmatrix}\right) $ with $\det h = 3$. Since $(\cy{4})^\times = \la -1\ra$ and $(\cy{2^k})^\times = \la -1, 3\ra$ if $k \geq 3$,  we have $\bd \la 3\ra = 0$. Therefore $\ker(L_3^S(\cy{2^k}) \to L_3^K(\cy{2^k})) = 0$ for $k \leq 2$, and for $k \geq 3$ we have
 $\ker(L_3^S(\cy{2^k}) \to L_3^K(\cy{2^k})) = \Z/2 \cdot \left[ \left(\begin{smallmatrix} -1 &\sx 0 \\ \sx 0 & -1 \end{smallmatrix}\right) \right ]$. \qedhere
\end{clist}
\end{proof}

\begin{proof}[Proof of \cref{prop:Lgroups-main}] For $\La = \cy m$, 
the natural map $L^S_{1}(\La) \to L^s_{1}(\La)$ factors through the comparison maps
 $L^S_{1}(\La) \to L^{\{\pm 1\}}_{1}(\La)$ and  $L^{\{\pm 1\}}_{1}(\La) \to L^s_{}(\La)$, for which the maps are explicitly described above. 
 The isomorphisms in \cref{lemma:L_1(Z)} and \cref{lemma:L_1(Z/m)} (i) give explicit representatives for the elements of $\ker(L_{1}^S(\l) \to L_{1}^K(\l))$ by matrices in $D_2(\l)$ for $\La = \bZ$ and $\La = \cy m$. Now we apply the braid diagram \eqref{eq:braid}. The map $\alpha\colon L^S_{1}(\La) \to L^{\{\pm 1\}}_{1}(\La) $ has cokernel $\cy 2$, hence 
  $\coker (\gamma\colon L^S_{1}(\La) \to L^s_{}(\La)) = 0$ and $\ker \gamma = \cy 2$.
To check these results, note that by \cref{lemma:L_1(Z/m)}, $L_1^S(\La) \cong (\Z/ m)^\times/(\Z/ m)^{\times 2}$ via the hyperbolic map, and by \cref{lemma:L_1(Z)} the element 
$\bd (\la -1\ra) \in \IM(L_1^S(\bZ) \to L_1^S(\La))$ generates the image of $\ker\gamma$ in $L_1^S(\La)$.  Therefore  $L_1^s(\cy m)$ is the quotient of $L_1^S(\La)$  by $\bd (\la -1\ra)$, and the required formula  follows.  \end{proof}

\begin{proposition}\label{prop:except} For $m \geq 1$, 
$L_3^s(\cy m) = L_3^h(\cy m) = 0$.
\end{proposition}
\begin{proof} If $m$ is odd or $m \equiv 2 \mod 4$,  part (i) of the proof of \cref{lemma:L_1(Z/m)} shows that $L_3^K(\cy m) = 0$. Therefore $L_3^h(\cy m) = 0$ by \eqref{eq:round_comp}. If $m = 2^k m_1$, with $m_1$ odd and $k \geq 3$, we have $L_3^h(\cy m) = L_3^K(\cy{2^k}) = \cy 2$. 
 In the exceptional case $m = 2^k$, $k \geq 3$, the calculations of Wall \cite[\S 1]{Wall:1976} show that the reduction map $L^K_3(\bZ) \to L^K_3(\cy{2^k})$ is an isomorphism. 
For any $m \geq 1$, we have the commutative diagram
$$\xymatrix{L^K_3(\bZ) \ar[r]\ar[d]&L_3^h(\bZ)\ar[d]\ar[r]& 0\\
L^K_3(\cy m) \ar[r]&L_3^h(\cy m)\ar[r] &0
}$$
Since $L_3^h(\bZ) = 0$, naturality shows that $L_3^h(\cy m) = 0$, for $m \geq 1$. 
The Ranicki-Rothenberg sequence \eqref{eq:RRcomp} combined with part (iv) of the proof of \cref{lemma:L_1(Z/m)}  shows that $L_3^s(\cy m) = 0$. 
\end{proof}

\begin{corollary}\label{cor:exceptcor}
The image of  $D_{2d}(\Z/m)$ in $\SU(\cy m)$ is contained in  $ \RU(\cy m)$, for $\ep = (-1)^n$ and $n$ odd.
\end{corollary}

\section{The numerical functions $e(d,n)$ and $s(d,n)$} \label{s:funct}

We first establish some useful facts about binomial coefficients mod $2$.
We will adopt the convention that $\binom{a}{b}=0$ if $a < b$ and we will use $\equiv$ to denote equivalence mod $2$.
The first property is well-known but is recalled here for convenience.

\begin{lemma} \label{lemma:mod2-useful}
Let $a,b \in \Z_{\ge 0}$. Then $\binom{2a}{2b+1} \equiv 0$
and $\binom{2a}{2b} \equiv \binom{2a+1}{2b} \equiv \binom{2a+1}{2b+1} \equiv \binom{a}{b}$.
\end{lemma}

\begin{proof}
Firstly, $\binom{2a}{2b+1} = \frac{2a}{2b+1} \binom{2a-1}{2b}$ implies that $\binom{2a}{2b+1} \equiv (2b+1) \binom{2a}{2b+1} \equiv (2a) \binom{2a-1}{2b} \equiv 0$.

Next note that
$\binom{2a+1}{2b+1} = \frac{2a+1}{2b+1} \binom{2a}{2b}$. Clearing denominators similarly implies that $\binom{2a+1}{2b+1} \equiv \binom{2a}{2b}$. Similarly, $\binom{2a+1}{2b} = \binom{2a+1}{2(a-b)+1} \equiv \binom{2a}{2(a-b)} = \binom{2a}{2b}$ by applying the result just proven.

Finally, note that $\binom{2(a+1)}{2(b+1)} = \frac{(2a+2)(2a+1)}{(2b+2)(2b+1)} \binom{2a}{2b}$ and so $\binom{2(a+1)}{2(b+1)} \equiv \frac{a+1}{b+1} \binom{2a}{2b}$ by clearing denominators. By induction, this implies that $\binom{2a}{2b} \equiv \frac{a(a-1) \cdots (a-b)}{b(b-1) \cdots 1} \binom{2(a-b)}{0} = \binom{a}{b}$.
\end{proof}

The following is often known as the hockey-stick identity.

\begin{lemma} \label{lemma:hockey}
Let $a, b \in \Z_{\ge 0}$. Then $\sum_{i=0}^n \binom{i+k}{k} = \binom{n+k+1}{k}$.	
\end{lemma}

\begin{proof}
$\sum_{i=0}^n \binom{i+k}{k} = \sum_{i=0}^n \left[\binom{i+k+1}{k+1} - \binom{i+k}{k+1}\right] = \binom{n+k+1}{k+1}$.
\end{proof}

We will now establish our main results concerning the parity of $e(d,n)$ and $s(d,n)$ respectively.

\begin{proposition} \label{prop:e(d,n)-eval}
Let $d, n \ge 2$. Then:
\begin{clist}{(i)}
\item
If $n = 4k$, then
\[ e(d,n) \equiv 
\begin{cases}
 0	, & \text{if $d \equiv 0,1,3 \mod 4$} \\
 \binom{t+k}{t+1} , & \text{if $d=4t+2$.}
\end{cases}
 \]	
\item
If $n = 4k+1$, then
\[ e(d,n) \equiv 
\begin{cases}
 \binom{t+k}{t}	, & \text{if $d=4t$ or $4t+1$} \\
  \binom{t+k+1}{t+1}, & \text{if $d=4t+2$ or $4t+3$.} \\
\end{cases}
 \]	
\item 
 If $n = 4k+2$, then
\[ e(d,n) \equiv 
\begin{cases}
 0	, & \text{if $d \equiv 1, 3 \mod 4$} \\
 \binom{t+k}{t}	 	, & \text{if $d=4t$} \\
 \binom{t+k+1}{t+1}	, & \text{if $d=4t+2$.} \\
\end{cases}
 \]	
\item
 If $n = 4k+3$, then
\[ e(d,n) \equiv 
\begin{cases}
0	, & \text{if $d=4t$ or $4t+1$} \\
  \binom{t+k+1}{t+1}, & \text{if $d=4t+2$ or $4t+3$.} \\
\end{cases}
 \]	
\end{clist}
\end{proposition}

Note that, if $n$ is even and $d$ is odd, then $e(d,n)$ is even.
Throughout the proof, we will make repeated use of \cref{lemma:mod2-useful,lemma:hockey}.

\begin{proof}
(i) $e(d,4k) = \sum_{i=0}^{2k-1} (2k-i) \binom{d+2i-1}{d-2} \equiv \sum_{0 \le i \le 2k-1, \,\text{$i$ odd}} \binom{d+2i-1}{d-2} = \sum_{i=0}^{k-1} \binom{d+4i+1}{d-2}$. If $d$ is odd, then $\binom{d+4i+1}{d-2} \equiv 0$ for all $i$ by \cref{lemma:mod2-useful} and so $e(d,4k) \equiv 0$. If $d = 4t$, then $\binom{d+4i+1}{d-2} = \binom{4t+4i+1}{4t-2} \equiv \binom{2t+2i}{2t-1} \equiv 0$ by \cref{lemma:mod2-useful} and so $e(d,4k) \equiv 0$. If $d = 4t+2$, then $e(d,4k) \equiv \sum_{i=0}^{k-1} \binom{4t+4i+3}{4t} \equiv \sum_{i=0}^{k-1} \binom{t+i}{t} = \binom{t+k}{k}$ by \cref{lemma:mod2-useful,lemma:hockey}.

(ii) $e(d,4k+1) = \sum_{i=0}^{2k} (2k-i+1) \binom{d+2i-2}{d-2} \equiv \sum_{0 \le i \le 2k, \,\text{$i$ even}} \binom{d+2i-2}{d-2} = \sum_{i=0}^k \binom{d+4i-2}{d-2}$.
If $d=2r$, then $e(d,4k+1) \equiv \sum_{i=0}^k \binom{2r+4i-2}{2r-2} \equiv \sum_{i=0}^k \binom{r+2i-1}{r-1}$ by \cref{lemma:mod2-useful}. 
If $r = 2t$, i.e. $d=4t$, then $e(d,4k+1) \equiv \sum_{i=0}^k \binom{2t+2i-1}{2t-1} \equiv \sum_{i=0}^k \binom{t+i-1}{t-1} = \binom{t+k}{t}$ by \cref{lemma:mod2-useful,lemma:hockey}. If $r=2t+1$, i.e. $d=4t+2$, then $e(d,4k+1) \equiv \sum_{i=0}^k \binom{2t+2i+1}{2t+1} \equiv \sum_{i=0}^k \binom{t+i}{t} = \binom{t+k+1}{t+1}$ by \cref{lemma:mod2-useful,lemma:hockey}. 
If $d=2r+1$, then $e(d,4k+1) \equiv \sum_{i=0}^k \binom{2r+4i-1}{2r-1} \equiv \sum_{i=0}^k \binom{r+2i-1}{r-1}$ by \cref{lemma:mod2-useful}.
This coincides with the $d=2r$ case, and so $e(d,4k+1) \equiv \binom{t+k}{t}$ for $d=4t+1$ and $e(d,4k+1) \equiv \binom{t+k+1}{t+1}$ for $d=4t+3$.

(iii) $e(d,4k+2) = \sum_{i=0}^{2k} (2k-i+1) \binom{d+2i-1}{d-2} \equiv \sum_{0 \le i \le 2k, \,\text{$i$ even}} \binom{d+2i-1}{d-2} = \sum_{i=0}^k \binom{d+4i-1}{d-2}$.
If $d$ is odd, then $\binom{d+4i-1}{d-2} \equiv 0$ for all $i$ by \cref{lemma:mod2-useful} and so $e(d,4k+2) \equiv 0$.
If $d=4t$, then $e(d,4k+2) \equiv \sum_{i=0}^k \binom{4t+4i-1}{4t-2} \equiv \sum_{i=0}^k \binom{t+i-1}{t-1} = \binom{t+k}{t}$ by \cref{lemma:mod2-useful,lemma:hockey}.
If $d=4t+2$, then $e(d,4k+2) \equiv \sum_{i=0}^k \binom{4t+4i+1}{4t} \equiv \sum_{i=0}^k \binom{t+i}{t} \equiv \binom{t+k+1}{t+1}$ by \cref{lemma:mod2-useful,lemma:hockey}.

(iv) $e(d,4k+3) = \sum_{i=0}^{2k+1} (2k-i+2) \binom{d+2i-2}{d-2} \equiv \sum_{0 \le i \le 2k, \,\text{$i$ odd}} \binom{d+2i-2}{d-2} = \sum_{i=0}^k \binom{d+4i}{d-2}$.
If $d=2r$, then $e(d,4k+3) \equiv \sum_{i=0}^k \binom{2r+4i}{2r-2} \equiv \sum_{i=0}^k \binom{r+2i}{r-1}$ by \cref{lemma:mod2-useful}. If $r=2t$, i.e. $d=4t$, then $e(d,4k+3) \equiv \sum_{i=0}^k \binom{2t+2i}{2t-1} \equiv 0$ by \cref{lemma:mod2-useful}.
If $r=2t+1$, i.e. $d=4t+2$, then $e(d,4k+3) \equiv \sum_{i=0}^k \binom{2t+2i+1}{2t} \equiv \sum_{i=0}^k \binom{t+i}{t} \equiv \binom{t+k+1}{t+1}$ by \cref{lemma:mod2-useful,lemma:hockey}.
If $d=2r+1$, then $e(d,4k+3) \equiv \sum_{i=0}^k \binom{2r+4i+1}{2r-1} \equiv  \sum_{i=0}^k \binom{r+2i}{r-1}$ by \cref{lemma:mod2-useful}. 
This coincides with the $d=2r$ case, and so $e(d,4k+1) \equiv 0$ for $d=4t+1$ and $e(d,4k+1) \equiv \binom{t+k+1}{t+1}$ for $d=4t+3$.
\end{proof}

The parity of $s(d,n)$ can be evaluated in terms of a single binomial coefficient without splitting into cases for $n$ and $d$ mod $4$. However, in order to make the results more easily comparable to Proposition \ref{prop:e(d,n)-eval}, we will also state the result in these terms.

\begin{proposition} \label{prop:s(d,n)-eval}
Let $d, n \ge 2$. Then $s(d,n) \equiv 1 + \binom{d+n}{d}$. In particular, we have that:
\begin{clist}{(i)}
\item
If $n = 4k$, then
\[ s(d,n) \equiv \textstyle  
1+ \binom{t+k}{t}, \quad \text{if $d=4t$, $4t+1$, $4t+2$ or $4t+3$.}
 \]	
\item
If $n = 4k+1$, then
\[ s(d,n) \equiv 
\begin{cases}
1+ \binom{t+k}{t}, & \text{if $d=4t$ or $4t+2$} \\
1, & \text{if $d=4t+1$ or $4t+3$.} \\
\end{cases}
 \]	
\item 
 If $n = 4k+2$, then
\[ s(d,n) \equiv 
\begin{cases}
1+ \binom{t+k}{t}, & \text{if $d=4t$ or $4t+1$} \\
1, & \text{if $d=4t+2$ or $4t+3$.} \\
\end{cases}
 \]	
\item
 If $n = 4k+3$, then
\[ s(d,n) \equiv 
\begin{cases}
1+ \binom{t+k}{t}, & \text{if $d=4t$ or $4t+2$} \\
1, & \text{if $d=4t+1$ or $4t+3$.} \\
\end{cases}
 \]	
\end{clist}
\end{proposition}

It follows from \cref{lemma:mod2-useful}  that, if $n$ is odd and $d$ is odd, then $s(d,n)$ is odd.

\begin{proof}
$s(d,n) = \sum_{i=0}^{n-1} (-1)^{n+i+1}\binom{d+i}{d-1} \equiv 1+\sum_{i=-1}^{n-1} \binom{d+i}{d-1} \equiv 1+\binom{n+d}{d}$, by \cref{lemma:hockey}.	 To evaluate $s(d,n)$ in the various values of $n, d \mod 4$, we repeatedly apply \cref{lemma:mod2-useful}.
\end{proof}

%\bibliographystyle{ih}
%\bibliography{bias}

\providecommand{\bysame}{\leavevmode\hbox to3em{\hrulefill}\thinspace}
\providecommand{\MR}{\relax\ifhmode\unskip\space\fi MR }
\providecommand{\MRhref}[2]{%
  \href{http://www.ams.org/mathscinet-getitem?mr=#1}{#2}
}
\providecommand{\href}[2]{#2}

\end{document}